\documentclass{svjour3}
\usepackage{graphicx}
\usepackage{amsmath,amssymb,amsfonts}
\usepackage{mathtools}    
\newcommand*{\id}{{\mathrm{id}}}

\newcommand*{\la}{{\langle}}                                                  
\newcommand*{\ra}{{\rangle}}                                                  
                                                        
\newcommand*{\R}{{\mathbb R}}
\newcommand*{\Rb}{{\mathbb R}}

\newcommand*{\CC}{{\mathbb{C}}}                                               
                                                
\newcommand*{\Hb}{{\mathbb H}}                                                
                                                
\newcommand*{\Vb}{{\mathbb V}}
\newcommand*{\Wb}{{\mathbb W}}                                                
\newcommand*{\Sb}{{\mathbb S}}                                                
\newcommand*{\Nb}{{\mathbb N}}

\newcommand*{\Jf}{{\mathfrak J}}

\newcommand*{\pa}{{\partial}}

\newcommand*{\Cb}{{\mathcal C}}
\newcommand*{\rd}{{\mathrm d}}                                                
                                               
\newcommand*{\Gr}{{\mathrm{Gr}}}

\newcommand*{\pf}{{\mathfrak p}} 
\newcommand*{\af}{{\mathfrak a}}  
  
\newcommand*{\wf}{{\mathfrak w}}
\newcommand*{\mm}{m}

\newcommand*{\bs}{\boldsymbol}
\smartqed
\journalname{Partial Differential Equations and Applications}
\allowdisplaybreaks

\begin{document}
\title{Applications of Grassmannian flows to integrable systems}
\titlerunning{Applications of Grassmannian flows}
\author{Anastasia Doikou \and Simon~J.A.~Malham \and Ioannis Stylianidis \and Anke Wiese}    
\authorrunning{Doikou, Malham, Stylianidis and Wiese}

\institute{Maxwell Institute for Mathematical Sciences,        
and School of Mathematical and Computer Sciences,   
Heriot-Watt University, Edinburgh EH14 4AS, UK\\
\email{A.Doikou@hw.ac.uk, S.J.A.Malham@hw.ac.uk, is11@hw.ac.uk, A.Wiese@hw.ac.uk}}
\date{12th January 2022}           
\maketitle

\begin{abstract}
We show how many classes of partial differential systems with local 
and nonlocal nonlinearities are linearisable in the sense that they are
realisable as Fredholm Grassmannian flows. In other words, time-evolutionary solutions
to such systems can be constructed from solutions to the corresponding underyling
linear partial differential system, by solving a linear Fredholm equation.
For example, it is well-known that solutions to
classical integrable partial differential systems can be generated
by solving a corresponding linear partial differential system for the scattering data
and then solving the linear Fredholm (or Volterra) integral equation known
as the Gel'fand--Levitan--Marchenko equation.
In this paper and in a companion paper, \textit{Doikou et al.\/} \cite{DMSW:graphflows},
we both, survey the classes of nonlinear systems that are realisable as
Fredholm Grasssmannian flows, and present new example applications of such flows.
We also demonstrate the usefulness of such a representation.
Herein we extend the work of P\"oppe and demonstrate how 
solution flows of the noncommutative potential Korteweg de Vries and nonlinear Schr\"odinger systems
are examples of such Grassmannian flows. 
In the companion paper we use this Grassmannian flow approach as well as an extension
to nonlinear graph flows, to solve Smoluchowski coagulation and related equations.
\end{abstract}

\section{Introduction}\label{sec:intro}
Our goal herein is to demonstrate that many classes of
partial differential systems with local and nonlocal nonlinearities
are realisable as Fredholm Grassmannian flows. In essence this means their
solution flow can be generated as the solution to a corresponding
set of linear partial differential systems together with
a linear Fredholm integral equation. In the context of classical integrable systems,
the linear Fredholm integral equation is the Gel'fand--Levitin--Marchenko equation.
In the more general unifying context of Grassmannian flows,
the linear Fredholm equation represents the linear relation whose
solution is the linear projection map from the Fredholm Stiefel manifold
to the Fredholm Grassmann manifold in a given Hilbert--Schmidt coordinate chart.
We call it the linear Fredholm kernel equation.
Note, the Fredholm Grassmann manifold can be thought of as consisting 
of all collections of graphs of compatible linear Hilbert--Schmidt maps.
In Beck, Doikou, Malham and Stylianidis~\cite{BDMS1,BDMS2},
we demonstrated the solution flows of many classes of partial differential systems 
with nonlocal nonlinearities are Grassmannian flows.
Here, we extend the class of flows which are Grassmannian flows
to include some classical integrable systems.
In a companion paper, Doikou \textit{et al.\/ } \cite{DMSW:graphflows},
we extend the class of Grassmannian flows to include the Smoluchowski coagulation and related equations
in the case of a constant frequency kernel.
We also we introduce the notion of graph flows, nonlinear generalisations of Grassmannian flows,
which incorporate the Smoluchowski coagulation equation for the cases of
additive and multiplicative frequency kernels.

The unifying approach to solving nonlinear partial differential equations 
we develop herein has its roots in computational spectral theory
where Grassmannian flows were developed to deal with numerical difficulties
associated with different exponential growth rates in the far-field 
for large (high order) systems, see Ledoux, Malham and Th\"ummler~\cite{LMT}
and Ledoux, Malham, Niesen and Th\"ummler~\cite{LMNT} as well as Karambal and Malham~\cite{KM}.
In Beck and Malham~\cite{BM} Grassmannian flows and representative coordinate patches were 
instrumental to the characterisation and computation of the Maslov index for large systems. 
This led to the need to develop the theory of Fredholm Grassmannian flows and thus nonlinear partial differential systems. 
We were then motivated by a sequence of papers by P\"oppe~\cite{P-SG,P-KdV,P-KP}, P\"oppe and Sattinger~\cite{PS-KP} 
and Bauhardt and P\"oppe~\cite{BP-ZS}. P\"oppe's work was recently revisited by McKean~\cite{McKean}.
These papers advocated, at the operator level, the solution of classical
integrable systems by solving a linearised version of the system at hand together with a Marchenko operator relation.
This was the natural development from classical constructions of this type suggested and developed in particular by Miura~\cite{Miura}, 
Dyson~\cite{Dyson} and Ablowitz, Ramani and Segur~\cite{ARSII}. Also see Nijhoff, Quispel, Van Der Linden and Capel~\cite{NQVDLCI},
Nijhoff, Quispel, Capel~\cite{NQVDLCII}, Mumford~\cite{Mumford}, Fordy and Kulisch~\cite{FK} and Tracy and Widom~\cite{TW}.
Of specific interest to us in P\"oppe's papers was that the structure underlying the approach advocated therein appeared to be  
that of a Fredholm Grassmannian. We exploited this perspective in Beck \textit{et al.\/} \cite{BDMS1,BDMS2} where 
we first explored this question and showed how to solve classes of partial differential systems
with nonlocal nonlinearities---we explain what we mean by such systems presently. 

Explicitly, though formally for the moment, the essential ideas 
underlying the P\"oppe programme we advocate here can be summarised as follows.
First let us define the unifying canonical system of operator equations underlying
all the systems we consider herein as well as in \textit{Doikou et al.\/} \cite{DMSW:graphflows}.
\begin{definition}[Canonical system]\label{def:nonlinearcanonicalsystem}
Suppose the time-dependent Hilbert--Schmidt operators $Q=Q(t)$, $P=P(t)$ and $G=G(t)$
satisfy the following system of operator equations, with $Q\coloneqq\id-\hat Q$:
\begin{align*}
\pa_tQ&=A(t,Q,P)\,Q+B(t,Q,P)\,P,\\
\pa_tP&=C(t,Q,P)\,Q+D(t,Q,P)\,P,\\
P&=GQ. 
\end{align*}
Here $A$, $B$, $C$ and $D$ are known operators which, in general, depend on $t$, $Q$ and $P$.
\end{definition}
Suppose $Q(0)=Q_0$ and $P(0)=P_0$ for some given data $Q_0$ and $P_0$ such that $\hat Q_0$ and $P_0$
are Hilbert--Schmidt operators. Further assume there exists a solution to the pair of evolutionary
equations for $Q=Q(t)$ and $P=P(t)$ shown in the canonical system in Definition~\ref{def:nonlinearcanonicalsystem},
in the sense that $\hat Q=\hat Q(t)$ and $P=P(t)$ are  Hilbert--Schmidt operators for $t\in[0,T]$ for some $T>0$.
If the Fredholm operator $Q=Q(t)$ and Hilbert--Schmidt operator $P=P(t)$ are related by the operator $G=G(t)$
shown in the third equation at least for some time $t\in[0,T]$ for some $T>0$,
then a straightforward calculation shows that if $Q=Q(t)$ and $P=P(t)$ satisfy the canonical system 
above, then $Q=Q(t)$ and $G=G(t)$ satisfy the coupled system of equations:
\begin{align*}
  \pa_tQ=&\;A(t,Q,GQ)\,Q+B(t,Q,GQ)\,(GQ),\\
  \pa_tG=&\;\bigl(C(t,Q,GQ)\,Q\bigr)\,Q^{-1}+\bigl(D(t,Q,GQ)\,(GQ)\bigr)\,Q^{-1}\\
  &\;-G\bigl(A(t,Q,GQ)\,Q+B(t,Q,GQ)\,(GQ)\bigr)\,Q^{-1},
\end{align*}
for $t\in[0,T^\prime]$ for some $0<T^\prime\leqslant T$.

Typically we choose the operators $A$, $B$, $C$ and $D$ so that the canonical system of equations
in Definition~\ref{def:nonlinearcanonicalsystem} is linear.
And further, typically, for the choices of $A$, $B$, $C$ and $D$ we make, the
evolution equation for the operator for $G$ decouples from that for $Q$---though the evolution
equation for $Q$ may still depend on $G$. All the systems considered in Beck \text{et al.\/} \cite{BDMS1,BDMS2},
as well as some of the Smoluchowski-type coagulation systems in our companion paper Doikou \text{et al.\/} \cite{DMSW:graphflows},
fall into this category. We briefly survey these cases now. 
Note that we have assumed that for $t\in[0,T]$ for some $T>0$, there exist solutions $Q$ and $P$ 
to the canonical system in Definition~\ref{def:nonlinearcanonicalsystem} such that $\hat Q$ and $P$
are Hilbert--Schmidt operators. As we shall see, as a consequence $G$ is also Hilbert--Schmidt valued.
This means that $\hat Q$, $P$ and $G$ have representations in terms of square-integrable kernel
functions, say, respectively $\hat q=\hat q(x,y;t)$, $p=p(x,y;t)$ and $g=g(x,y;t)$.
Hence the linear relation $P=GQ$ between $Q$ and $P$, which hereafter we shall denote 
the \emph{linear Fredholm equation}, has the form
\begin{equation*}
p(x,y;t)=g(x,y;t)-\int g(x,z;t)\hat q(z,y;t)\,\rd z.
\end{equation*}
The interval of integration, and any further properties assumed for the kernel functions $\hat q$, $p$ and $g$,
depend on the context/application at hand. For example, suppose the operators $A=C=O$ while 
$B=-\pa_x$ and $D=d(\pa_x)$ where $d$ is a constant coefficient polynomial function of its argument.
Then the canonical system reduces to the linear pair of partial differential equations
$\pa_t p=d(\pa_x)p$ and $\pa_t\hat q=-\pa_x p$. The evolution equation for $G=G(t)$, in terms of its
kernel $g=g(x,y;t)$ becomes (note the interval of integration is $\R$):
\begin{equation*}
\pa_tg(x,y;t)=d(\pa_x)g(x,y;t)+\int_\R g(x,z;t)\pa_z g(z,y;t)\,\rd z.
\end{equation*}
This nonlocal nonlinear partial differential equation is typical of
the many examples presented in Beck \text{et al.\/} \cite{BDMS1,BDMS2}.
Such classes of nonlinear partial differential equations have two underlying
characteristics. The first is that the nonlinearity is nonlocal as shown and we think of it 
as the generalisation from a product of matrix operators to the infinite
dimensional analogue of the kernel representation of the product of a
pair of infinite-dimensional Hilbert--Schmidt operators, here between
$G$ and $\pa_x G$. We call the integral term above the `\emph{big matrix product}'
and at the kernel level express it as $g\star (\pa g)$, i.e.\/ for any
two kernels $g=g(x,y)$ and $g'=g'(x,y)$ we write
\begin{equation*}
\bigl(g\star g'\bigr)(x,y)\coloneqq\int_\R g(x,z)g'(z,y)\,\rd z.
\end{equation*}
The second characteristic is that the unbounded operators shown $d=d(\pa x)$ for $D$ and $-\pa_x$ for $B$,
only act on the first arguments of the kernel functions.
Now assume, instead of the form above, the operator  $B$ corresponds to minus the identity operator
so that $\pa_t q=-p$. Then a  special subclass we can consider here is to suppose the
linear relation $P=GQ$ has a convolutional form so that,
\begin{equation*}
p(x,y;t)=g(x-y;t)-\int_\R g(x-z;t)\hat q(z,y;t)\,\rd z.
\end{equation*}
Then the evolution equation for $G=G(t)$, after setting $y=0$ becomes
(again for the moment the interval of integration is $\R$):
\begin{equation*}
\pa_tg(x;t)=d(\pa_x)g(x;t)+\int_\R g(x-z;t)g(z;t)\,\rd z.
\end{equation*}
See Beck \text{et al.\/} \cite[Sec.~1.4]{BDMS1} for more details.
If we restrict the interval of integration to $[0,\infty)$ then
this form of equation demonstrates how we can treat Smoluchowski-type coagulation
equations with constant frequency kernels in this manner.
Indeed this is the basis for the multitude of examples we consider
in our companion paper Doikou \text{et al.\/} \cite{DMSW:graphflows}.
Within this class of `big matrix product' equations, another example is insightful.
Hitherto, the nonlinearity in the resulting equation for $g=g(x,y;t)$ was generated
by the term associated with the operator $B$. This term corresponds to the quadratic
term that appears in the Riccati equation $\pa_t G=C+DG-G(A+BG)$, that results from
assuming $A$, $B$, $C$ and $D$ are functions of $t$ only, or if we allow $B$ and $D$
to also have the dependence on $\pa_x$ indicated above. However, now
suppose in the canonical system, the operators $B=C=O$ while $A(t,Q,P)=-\mathrm{i}f(PP^\dag)$
and $D=-\mathrm{i}h(\pa_x)$ where $h$ and $f$ are real constant coefficient polynomials
of their arguments with $h$ even, and $P^\dag$ represents the operator adjoint to $P$. 
Hence the canonical system reduces to the linear pair of partial differential equations
$\pa_t p=-\mathrm{i}h(\pa_x)p$ and $\pa_t\hat q=-\mathrm{i}f^\star(p\star p^\dag)\star q$, where now $p^\dag$
represents the kernel of $P^\dag$ and $q=\delta-\hat q$, with the `$\delta$' representing the
identity operator at the kernel level, in particular $\delta\star q'=q'\star\delta=q'$ for 
any Hilbert--Schmidt kernel $q'$. Note the `$\star$' is the big matrix product above, and
$f^\star$ represents the same polynomial as $f$, but the product given by `$\star$'. 
In this case the evolution equation for $G=G(t)$, i.e.\/ for the kernel $g=g(x,y;t)$, becomes
(see  Beck \text{et al.\/} \cite{BDMS2} for the details): 
\begin{equation*}
\mathrm{i}\pa_tg=h(\pa_x)g+g\star f^\star(g\star g^\dag).
\end{equation*}
If $f(x)=x$, then this equation represents the nonlocal
nonlinear Schr\"odinger equation, with the product the nonlocal big matrix product.
Note the nonlinearity is generated by the term corresponding to the operator $A$.
Further note, the canonical system here is linear in the sense that we can
solve the linear equation for $p=p(x,y;t)$ first, and then substitute its solution form
into the evolution equation $\pa_t\hat q=-\mathrm{i}f^\star(p\star p^\dag)\star q$ for $q=q(x,y;t)$
which is linear. See Beck \text{et al.\/} \cite{BDMS2} for more details.

Let us now examine the connection between the canonical system in Definition~\ref{def:nonlinearcanonicalsystem}
and classical integrable systems. Some further layers of structure are also required. 
As discussed above, and in particular motivated by the results in Ablowitz, Ramani and Segur~\cite{ARSII}
and in the sequence of papers by P\"oppe, we know that 
solutions to classical nonlinear integrable systems can be generated from solutions to the corresponding
linearised versions of these equations by solving the corresponding Gel'fand--Levitan--Marchenko equation.
At the operator level, say for the Korteweg--de Vries case, we suppose $\pa_t P=\mu\pa_x^nP$ and $\hat Q=P$,
so that $P=G(\id-P)$. Here $\mu\in\R$ is a constant and $n\geqslant 3$.
With reference to the canonical system in Definition~\ref{def:nonlinearcanonicalsystem},
this corresponds to setting $B=C=O$ and $A=D=\mu\pa_x^n$ so that the canonical system and linear
Fredholm relation reduce to,
\begin{align*}
\pa_tQ&=\mu\pa_x^nQ,\\
\pa_tP&=\mu\pa_x^nP,\\
P&=GQ.
\end{align*}
In the Korteweg--de Vries equation context, the scattering data corresponding to the Hilbert--Schmidt operator $P$
with kernel $p$ is assumed to be additive. By this we mean the following. If $y$ and $z$
are the variables parameterising the Hilbert--Schmidt kernel $p=p(y,z;t)$ corresponding
to $P$, then we assume $p$ has the form $p=p(y+z;t)$. Operators with such additive kernels
are known as \emph{Hankel operators}. Further we assume the operator $P$ depends on a parameter $x\in\R$
in an additive way so that it has the form $p=p(y+z+x;t)$. For this form, we can equivalently use $\pa_y$, $\pa_z$ or $\pa_x$
in the system of equations for the operators above, indeed, naturally we only need to solve the corresponding
kernel equations in terms of a single variable, eg.\/ we only need to solve $\pa_tp=\mu\pa_y^np$ for $p=p(y;t)$.
With $p$ and thus $P$ in hand, $\hat Q$ is automatically given as $\hat Q=P$. Further in this context the 
\emph{linear Fredholm kernel equation} has the form (note the interval of integration is now $(-\infty,0]$),
\begin{equation*}
p(y+z+x;t)=g(y,z;x,t)-\int_{-\infty}^0g(y,\xi;x,t)p(\xi+z+x;t)\,\rd\xi.
\end{equation*}
If we set $y=0$ and making a change of variables, this relation can be shown
to be precisely the Gel'fand--Levitan--Marchenko equation; see Remark~\ref{rmk:GLMandRiccati} for the details.
Hence thusfar, we have precisely the setup for the \emph{inverse scattering transform}.
Classically the procedure is now, using that the function $p$ satisfies the linear Korteweg--de Vries equation
and has an additive form, to differentiate the Fredholm equation above and show that $g=g(0,0;x,t)$
satisfies the potential Korteweg--de Vries equation. Note, setting $z=0$ in $g=g(0,z;x,t)$
is equivalent to evaluating the unknown in the Gel'fand--Levitan--Marchenko equation along the diagonal.
However herein, motivated by the work of P\"oppe, we prefer to keep the subsequent computation at
the operator level. Recall our choices for the operators $A$, $B$, $C$ and $D$ in this case. 
The evolution equation for $G=G(t)$ here has the form:
\begin{equation*}
  \pa_tG=\mu\pa_x^n(GQ)\,Q^{-1}-\mu G(\pa_x^nQ)\,Q^{-1}.
\end{equation*}
Setting $U\coloneqq(\id-P)^{-1}$, using that $U=\id+PU=\id+UP$ so $\pa_t G=\pa_t U$ and
also $\pa_t G=U(\pa_t P)U$, a more succinct version of this last relation is,
\begin{equation*}
  \pa_tG=\mu U(\pa_x^nP)U. 
\end{equation*}
This is still not a closed form equation for $G=G(t)$, however our goal is to
derive a closed form nonlinear equation for its kernel $g=g(y,z;x,t)$.
We also want to develop a systematic procedure to do this.
To achieve both, following P\"oppe, we now incorporate some additional implicit structure
we have not utilised thusfar. For any given Hilbert--Schmidt operator say $G$, let $[G]$ 
denote its square-integrable kernel so $[G]=g$. We denote the operator $[\,\cdot\,]$ the \emph{bracket operator}.
Thusfar we have not used that $P$ is a Hankel operator, for which the following crucial product rule holds. 
Suppose $F$ and $F'$ are Hilbert--Schmidt operators and that $H$ and $H'$
are Hilbert--Schmidt Hankel operators, dependent on parameters $x\in\R$ and $t\geqslant0$.
Then the fundamental theorem of calculus implies the following product rule
(the \emph{P\"oppe product rule}, see Lemma~\ref{lemma:productrule}):
\begin{equation*}
[F\pa_x(HH')F'](y,z;x,t)\equiv [FH](y,0;x,t)[H'F'](0,z;x,t).
\end{equation*}
The goal now is to apply the bracket operator to the evolution equation for $G$ above,
and see if the bracket operator applied to the sequence of operators $U(\pa_x^nP)U$
generates a closed form in $[G]$. The answer in this instance is positive, and after
relatively small effort we can show (see Theorem~\ref{thm:KdV}) that for the $n=3$ case,
\begin{equation*}
\bigl[U(\pa_x^3P)U\bigr](y,z;x,t)=\pa_x^3[G](y,z;x,t)-3\,\bigl(\pa_x[G](y,0;x,t)\bigr)\bigl(\pa_x[G](0,z;x,t)\bigr).
\end{equation*}
Thus $g(y,z;x,t)=[G](y,z;x,y)$ satisfies the nonlinear partial differential equation,
\begin{equation*}
  \pa_tg(y,z;x,t)=\mu\pa_x^3g(y,z;x,t)-3\bigl((\pa_x g)(y,0;x,t)\bigr)\bigl((\pa_x g)(0,z;x,t)\bigr).
\end{equation*}
Actually here we have a whole set of nonlocal nonlinear equations for each $y$ and $z$.
The nonlocal nonlinearity is the second notion of nonlocal nonlinearity we mention herein.
Importantly however, we notice that if we set $y=z=0$ then $g=g(0,0;x,t)$ satisfies the
the potential Korteweg--de Vries equation (with usual local nonlinearity),
\begin{equation*}
  \pa_tg(0,0;x,t)=\mu\pa_x^3g(0,0;x,t)-3\mu\bigl((\pa_x g)(0,0;x,t)\bigr)^2. 
\end{equation*}
Let us summarise the procedure just outlined. We start with a simple linear evolutionary
partial differential equation for the scattering data operator $P=P(t)$, which we assume is
a Hilbert--Schmidt Hankel operator. This equation represents the linear version of the target nonlinear integrable
partial differential equation. We set $\hat Q\coloneqq P$. 
We define an evolutionary Hilbert--Schmidt operator $G=G(t)$ via the Fredholm equation $P=G(\id-\hat Q)$.
Then by a systematic computation we derive an evolution equation for $G=G(t)$,
in which the flow-field, i.e.\/ the right-hand side in the equation $\pa_tG=\cdots$, depends in
general on $P$ and $U\coloneqq(\id-\hat Q)^{-1}$. The evolutionary operator $Q=\id-\hat Q$ itself
satisfies an evolutionary equation with a flow-field that depends on $Q$ and $G$---see the discussion
immediately following Definition~\ref{def:nonlinearcanonicalsystem}.
However we focus on the equation for $G=G(t)$, and see if we can write the flow-field,
at the level of the operator kernels as a closed form in terms of
the kernel or $\pa_x$ derivatives of the kernel---this is where we applied the
bracket operator and then utilised the product rule. It turns out, though this is relatively
obscure presently, that this latter procedure where we try to express the flow-field for $[G]=[G](t)$
as a closed form in terms of $[G]$ or spatial partial derivatives of $[G]$ is highly systematic
and boils down to algebraic polynomial construction. See for example Malham~\cite{Malham:KdVhierarchy}.
Our belief is that it is this systematic structure/procedure, initiated by P\"oppe,
that represents the major advantage of his approach.

Let us briefly consider one more example we tackle herein, the noncommutative nonlinear Schr\"odinger equation.
At the operator level we suppose the Hilbert--Schmidt Hankel operator $P=P(t)$ satisfies the linear
dispersive system $\pa_t P=-\mu_n(\mathrm{i}\mathcal I)^{n-1}\pa_x^nP$, where $\mu_n\in\R$ is a constant
and $\mathcal I\coloneqq\mathrm{diag}\{-\id,\id\}$, i.e.\/ the diagonal square matrix with the upper left
and lower right blocks as indicated. Further we set $\hat Q=P^2$. We can in principle fit this 
into the canonical system in Definition~\ref{def:nonlinearcanonicalsystem}, though there is no
need as we already have the necessary prescibed linear evolution in terms of $P$ and $\hat Q$.
As for the Korteweg--de Vries case, it is convenient to set $U\coloneqq(\id-\hat Q)^{-1}$.
A straightforward calculation shows (see Remark~\ref{rmk:NLSsimplify}) that,
\begin{equation*}
  \pa_tG=-\mu_n(\mathrm{i}\mathcal I)^{n-1}\bigl(U(\pa_x^nP)U+(-1)^{n-1}G(\pa_x^nP)G\bigr). 
\end{equation*}
As we did for the potential Korteweg--de Vries equation above, we apply the bracket operator
and see if the bracket operator applied to the right-hand side above generates a closed form in $[G]$.
Again, after relatively small effort we can show (see Theorem~\ref{thm:NLS}) that for the $n=2$ case, 
\begin{equation*}
  \bigl[U(\pa_x^2P)U-G(\pa_x^2P)G\bigr]=\pa_x^2[G](y,z;x,t)-2\,[G](y,0;x,t)[G](0,0;x,t)[G](0,z;x,t).
\end{equation*}
Thus $g(y,z;x,t)=[G](y,z;x,y)$ satisfies the nonlinear partial differential equation,
\begin{equation*}
  \mathrm{i}\pa_tg(y,z;x,t)=\mu_2\mathcal I\bigl(\pa_x^2g(y,z;x,t)-2g(y,0;x,t)g(0,0;x,t)g(0,z;x,t)\bigr).
\end{equation*}
If we set $y=z=0$ then $g=g(0,0;x,t)$ satisfies a noncommutative local cubic nonlinear partial differential equation.
That this system corresponds to the noncommutative nonlinear Schr\"odinger equation follows once we
suppose $g$ has the block form,
\begin{equation*}
g=\begin{pmatrix} O & g_\alpha^\dag\\ g_\alpha & O\end{pmatrix},
\end{equation*}
where $g_\alpha^\dag$ is the complex conjugate transpose of $g_\alpha$ which necessarily satisfies
the noncommutative nonlinear Schr\"odinger equation. We also establish integrability for
the noncommutative modified Korteweg--de Vries equation in this way, as well as nonlocal
reverse space-time versions as well. Indeed to establish these results, we extend the abstract
P\"oppe algebra approach we developed for the non-commutative potential Korteweg--de Vries hierarchy
in Malham~\cite{Malham:KdVhierarchy}. Here we develop a new abstract combinatorial structure for Hilbert--Schmidt operators.
This structure is a vector space which we endow with a triple product based on the P\"oppe product---the triple
product is required for closure of the algebra. We call this structure the \emph{P\"oppe triple system}.
As with the P\"oppe algebra, the P\"oppe triple system breaks down the problem of establishing integrability
to the problem of determining the existence of suitable polynomial expansions in the associated respective algebras,
which translates to solving an overdetermined linear algebraic problem for the polynomial coefficients.
Again, it is this systematic procedure that represents one of the 
advantages of P\"oppe's Hankel operator approach for integrable systems.

Our work combines applications of Hankel operators to noncommutative integrable systems,
including recently discovered nonlocal integrable systems. Noncommutative integrable
systems have received a lot of recent attention. See for example Adamopoulou and Papamikos~\cite{AP},
Buryak and Rossi~\cite{BuryakRossi}, Carillo and Schoenlieb~\cite{CSI,CSIa,CSII}, Degasperis and Lombardo~\cite{DL2},
Ercolani and McKean~\cite{EM}, Pelinovsky and Stepanyants~\cite{Pelinovsky} and Treves~\cite{TI,TII}.
The P\"oppe programme relies on the property that the linear operator $P$ above is a Hankel operator.
The connection between Hankel operators and integrable systems has also been recently highlighted in a series of papers,
see for example, Blower and Newsham~\cite{BM}, Grellier and Gerard~\cite{Gerard} and Grudsky and Rybkin~\cite{GRI,GRII}.
Partial differential systems with nonlocal nonlinearities which are integrable are also currently a very active research area.
See for example the many nonlocal systems considered in Ablowitz and Musslimani~\cite{AMusslimani}, the extension to
multi-dimensions in Fokas~\cite{Fokas}, as well as Grahovski, Mohammed and Susanto~\cite{GMS} and
G\"urses and Pekcan~\cite{GP2018,GP2019a,GP2019b,GP2020}. Also see Ablowitz and Musslimani~\cite{AMshift}
who consider space-time shifted nonlocal nonlinear equations.
There is a long history of the connection between integrable systems and Fredholm Grassmannians.
Infinite dimensional Grassmann manifolds are also called Sato Grassmannians in recognition
of the seminal work by Sato~\cite{SatoI,SatoII} making this connection. Also see Miwa, Jimbo and Date~\cite{MJD},
Mulase~\cite{Mulase}, Pressley and Segal~\cite{PS}, Segal and Wilson~\cite{SW}) and Wilson~\cite{W}.
For more recent work on this connection, see for example, Dupr\'e \textit{et al.\/} \cite{DGP2006,DGP2007,DGP2013},
Hamanaka and Toda~\cite{HT} and Kasman~\cite{Kasman1995,Kasman1998}.
For more details on the theory of infinite dimensional frames, see Balazs~\cite{Balazs} and
Christensen~\cite{Christensen}, and for further background on infinite dimensional Grassmann manifolds,
see Abbondandolo and Majer~\cite{AM}, Andruchow and Larotonda~\cite{AL}, Furutani~\cite{F} and
Piccione and Tausk~\cite{PT}. Algebraic approaches that are close to that we adopt herein
can be found in Dimakis and M\"uller--Hoissen~\cite{DM-H2005}, and a connection to shuffle
and Rota--Baxter algebras can be found in M\"uller--Hoissen~\cite{DM-H2008}. For more
on shuffle algebras as well as the abstract formalsim we consider here, see for example, Reutenauer~\cite{Reutenauer},
Malham and Wiese~\cite{MW} and Ebrahimi--Fard \textit{et al.\/} \cite{EFMKLMW}.

In this paper we:
\begin{enumerate}
\item Introduce in detail, the Fredholm Grassmann manifold and the characteristics of evolutionary
  flows on them;
\item Develop new tests for functions on $(-\infty,0]$ that establish whether they generate Hilbert--Schmidt
  Hankel operators;
\item Define the quasi-trace for Hilbert--Schmidt operators with matrix-valued integral
  kernels and give a new solution formula for the noncommutative Korteweg--de Vries equation;
\item Develop a new (inflated) system of linear dispersive partial differential equations that
  underlie the noncommutative potential Korteweg--de Vries equation on the one hand with a particular
  choice of linear Fredholm equation, and the noncommutative nonlinear Schr\"odinger and modified Korteweg--de Vries equations
  on the other hand, with a slightly different choice of linear Fredholm equation. This linear system
  considerably simplifies the operator analysis and algebra used to establish integrability of these equations; 
\item Review the abstract P\"oppe algebra based on the P\"oppe product for Hankel operators developed
  in Malham~\cite{Malham:KdVhierarchy}. We use the algebra to demonstrate how establishing integrability  
  for the noncommutative potential Korteweg--de Vries equation corresponds to establishing polynomial
  expansions in the algebra;
\item Introduce a new abstract algebraic triple system, the \emph{P\"oppe triple system}, based on the P\"oppe product.
  We use the triple system to establish integrability for the noncommutative nonlinear Schr\"odinger
  and modified Korteweg--de Vries equations. We show how establishing integrability corresponds to
  establishing polynomial expansions in the triple system which in turn boils down to solving an
  overdetermined linear algebraic system of equations for the polynomial coefficients.
  This abstract approach simulatnaeously establishes integrability for the nonlocal reverse space-time
  versions of these equations.
\end{enumerate}

Our paper is structured as follows. In Section~\ref{sec:FredholmGrassmannian}
we introduce Fredholm Grassmann manifolds together with some relevant associated properties and flows.
Then in Section~\ref{sec:Hankelandlinear} we introduce Hankel operators, the P\"oppe product
and conditions/properties for operators to be Hilbert--Schmidt valued.
We also introduce the quasi-trace. We analyse linear dispersive partial differential systems
in Section~\ref{sec:dispersivelinearPDES} establishing existence, uniqueness and regularity results 
as well as kernel properties we require for subsequent sections. We also provide a
new formula for the solution of the noncommutative Korteweg--de Vries equation.
In Section \ref{sec:KdV} we introduce the P\"oppe algebra and establish 
the integrability of the noncommutative potential Korteweg--de Vries equation,
as an example polynomial in that algebra. In Section~\ref{sec:NLS}
we introduce a new triple system algebra based on the P\"oppe product,
which we use to establish the integrability for the noncommutative
nonlinear Schr\"odinger and modified Korteweg--de Vries equations.
We discuss further applications and possible extensions in Section~\ref{sec:discussion}.
In Appendices~\ref{sec:numericalsimulations}, \ref{sec:scatteringproblem} and \ref{sec:Evansfunction} 
we, respectively, demonstrate numerical simulations based on the P\"oppe method,
derive the scattering and inverse scattering transformations in the noncommutative
context for completeness, and outline how we compute the transmission and reflection
coefficients in practice.

\section{Fredholm Grassmannian}\label{sec:FredholmGrassmannian}
An important structure underlying a large class of nonlinear systems 
is the Fredholm Grassmann manifold. Herein we introduce its structure.
More details and background information can be found in 
Beck \textit{et al.\/ }~\cite{BDMS1,BDMS2}, Segal and Wilson~\cite{SW}
and Pressley and Segal~\cite{PS}. Towards the end of this section
we also introduce some specialist sub-structures we will
need for the different applications to come. The Fredholm Grassmann manifold
or Fredholm Grassmannian is also known as the Sato Grassmannian or 
Segal--Wilson Grassmannian, amongst other nominations; see
Sato~\cite{SatoI,SatoII}, Miwa, Jimbo and Date~\cite{MJD},
Segal and Wilson~\cite[Section~2]{SW}) and Pressley and Segal~\cite[Chapters~6,7]{PS}. 

Before we proceed to introduce the Fredholm Grassmann manifold, to be complete,
we recall some facts on compact operators we shall need; see Reed and Simon~\cite{RS,RSIV},
Simon~\cite{Simon:Traces} and Gohberg \textit{et al.\/} \cite{GGK}.
Suppose we have a separable Hilbert space $\Hb=\Hb(\CC)$ with unitary basis $\{\varphi_j\}_{j\in\mathbb N}$ 
and standard inner product $\langle\,\cdot\,,\,\cdot\,\rangle_{\Hb}$.
We use $\Jf_\infty(\Hb)$ to denote the set of compact operators in $\Hb$.
An operator $A\in\Jf_\infty(\Hb)$ is positive if $\langle\varphi,A\varphi\rangle_{\Hb}\geqslant0$
for all $\varphi\in\Hb$. The operator $A^\dag A$ is positive as  
$\langle A^\dag A\varphi,\varphi\rangle_{\Hb}=\|A\varphi\|_{\Hb}^2\geqslant0$ and
we define $|A|=\sqrt{A^\dag A}$. Further, there is a unique unitary operator $U$ such that $A=U|A|$.
For any $A\in\Jf_\infty(\Hb)$, we define the trace by
$\mathrm{tr}\,A\coloneqq\sum_{j\in\mathbb N}\langle\varphi_j,A\varphi_j\rangle_{\Hb}$.
When it exists, the trace is linear and independent of the unitary basis chosen.
We will be particular;y concerned with two subclasses of operators of $\Jf_\infty(\Hb)$,
namely the \emph{Hilbert--Schmidt class} $\Jf_2(\Hb)$, and the \emph{trace class} $\Jf_1(\Hb)$
set of operators. These are characterised by the property ($N=1$ or $2$ although there is
a whole set of Schatten--von Neumann classes for general $N$):
\begin{equation*}
\Jf_N(\Hb)\coloneqq\{A\in\Jf_\infty(\Hb)\colon \mathrm{tr}\,|K|^N<\infty\}.
\end{equation*}
We have the natural inclusions $\Jf_1(\Hb)\hookrightarrow\Jf_2(\Hb)\hookrightarrow\Jf_\infty(\Hb)$.
Further, the Hilbert--Schmidt class of operators $\Jf_2(\Hb)$ is a Hilbert space with
inner product $\langle A_1,A_2\rangle_{\Jf_2(\Hb)}\coloneqq\mathrm{tr}\,A_1^\dag A_2$.
We can also characterise $\Jf_1(\Hb)$ and $\Jf_2(\Hb)$ as follows.
The eigenvalues $\{\lambda_j\}_{j\in\mathbb N}$ of any compact operator $A\in\Jf_\infty(\Hb)$
are finite in number away from the origin and the origin itself is the only accumulation point.
The singular values $\{s_j\}_{j\in\mathbb N}$ of $A\in\Jf_\infty(\Hb)$ are the eigenvalues of $\sqrt{A^\dag A}$.
Then we equivalently have, for $N=1$ or $2$, that $\mathrm{tr}\, A^N=\sum_{j\in\mathbb N}\lambda_j^N$
and $\mathrm{tr}\, |A|^N=\sum_{j\in\mathbb N}s_j^N$, with $\mathrm{tr}\, A^N\leqslant\mathrm{tr}\, |A|^N$.
Both $\Jf_1(\Hb)$ and $\Jf_2(\Hb)$ are operator ideals and indeed, the following properties hold.
If $A,B\in\Jf_2(\Hb)$ then $AB\in\Jf_1(\Hb)$. For either $N=1$ or $2$ we have, if $A\colon\Hb\to\Hb$
is a bounded operator and $B\in\Jf_N(\Hb)$, then $AB\in\Jf_N(\Hb)$. These results are
encapsulated in the inequalities (for $N=1$ or $2$):
\begin{equation*}
  \|AB\|_{\Jf_1(\Hb)}\leqslant\|A\|_{\Jf_2(\Hb)}\|B\|_{\Jf_2(\Hb)}\qquad\text{and}\qquad
  \|AB\|_{\Jf_N(\Hb)}\leqslant\|A\|_{\mathrm{op}}\|B\|_{\Jf_N(\Hb)},
\end{equation*}
where $\|\,\cdot\,\|_{\mathrm{op}}$ denotes the operator norm. Lastly, and
fundamentally, for any trace class operator $A\in\Jf_1(\Hb)$ there exists
a Fredholm determinant $\det_1(\id+A)$, while for a Hilbert--Schmidt operator
$A\in\Jf_2(\Hb)$ there exists a regularised Fredholm determinant $\det_2(\id+A)$.
Using the identity $\det\equiv\exp\,\mathrm{tr}\,\log$
these two  Fredholm determinants (for $N=1$ or $2$ as shown) are given for any $\epsilon\in\CC$ by:
\begin{equation*}
\mathrm{det}_N(\id+\epsilon A)\coloneqq\exp\sum_{\ell\geqslant N}\frac{(-1)^{\ell-1}}{\ell}\epsilon^\ell\mathrm{tr}\,A^\ell.
\end{equation*}
We discuss such traces and determinants in more detail in Section~\ref{sec:Hankelandlinear} 
in the context of our practical applications.

Recall the sequence space $\ell^2(\CC)$ of square summable complex sequences. 
It is sufficient for us to parametrise the sequences in $\ell^2(\CC)$ by $\mathbb N$.
Any sequence $\af\in\ell^2(\CC)$ can, for example, be represented as a column vector
$\af=(\af(1),\af(2),\af(3),\ldots)^{\mathrm{T}}$ where $\af(n)\in\CC$ for each $n\in\mathbb N$.
For two elements $\af,\mathfrak b\in\ell^2(\CC)$ the natural inner product 
on $\ell^2(\CC)$ is given by $\af^\dag\mathfrak b$, where $\dag$
denotes the complex conjugate transpose. And of course, $\af\in\ell^2(\CC)$ if and only if
$\af^\dag\af=\sum_{n\in\mathbb N}\af^\ast(n)\af(n)<\infty$, where $\ast$ denotes the complex conjugate only.
Further, a natural canonical orthonormal basis for $\ell^2(\CC)$ are the 
vectors $\{\mathfrak e_n\}_{n\in\mathbb N}$ where $\mathfrak e_n$
is the vector with one in its $n$th component and zeros elsewhere.
The reason for introducing $\ell^2(\CC)$ here is that all separable Hilbert spaces
are isomorphic to the sequence space $\ell^2(\CC)$; see Reed and Simon~\cite[p.~47]{RS}.
It provides a natural context in which to envisage the Fredholm Grassmann manifold.
Though the separable Hilbert function spaces we consider are isomorphic to $\ell^2(\CC)$,
in practice, they will not be isometric, and will correspond to subspaces of $\ell^2(\CC)$.
For example, the $\mathbb Z^2$-parameterised set of Haar or Daubechies wavelets give a complete,
orthonormal basis in $L^2(\R)$; see Qian and Weiss~\cite{QW}.
However in practice we consider smooth functions that are square integrable with respect to a weight function
and whose derivatives are also square integrable.

The Fredholm Grassmannian of all subspaces of a separable Hilbert space $\Hb$ that are
comparable in size to a given closed subspace $\Vb\subset\Hb$ is defined as follows;
see Pressley and Segal~\cite{PS}. 
\begin{definition}[Fredholm Grassmannian]\label{def:FredholmGrassmannian}
Let $\Hb$ be a separable Hilbert space with a given decomposition
$\Hb=\Vb\oplus\Vb^\perp$, where $\Vb$ and $\Vb^\perp$ are infinite
dimensional closed subspaces. The Grassmannian $\Gr(\Hb,\Vb)$
is the set of all subspaces $\Wb$ of $\Hb$ such that:
\begin{enumerate}
\item[(i)] The orthogonal projection $\mathrm{pr}\colon\Wb\to\Vb$ is 
a Fredholm operator, indeed it is a Hilbert--Schmidt perturbation
of the identity; and
\item[(ii)] The orthogonal projection $\mathrm{pr}\colon\Wb\to\Vb^\perp$ 
is a Hilbert--Schmidt operator.
\end{enumerate}
\end{definition}
Since $\Hb$ is separable, any element in $\Hb$ has a representation on a countable
basis, for example via the sequence of coefficients of the basis elements---as we
discussed for the case of $\ell^2(\CC)$ above. Suppose we are given a set of independent
sequences in $\Hb=\Vb\oplus\Vb^\perp$ which span $\Vb$ and we record them as
columns in the infinite matrix
\begin{equation*}
W=\begin{pmatrix} \id-\hat Q \\ P \end{pmatrix}.
\end{equation*}
Here each column of $\id-\hat Q\in\Vb$ and each column of $P\in\Vb^\perp$.
In our actual applications we take $\Hb=\Vb\times\Vb$ and thus $\Vb^\perp$
isomorphic and isometric to $\Vb$. Assume this context hereafter.
Now assume that when we constructed $\id-\hat Q$ and $P$, we ensured 
that $\id-\hat Q$ was a Fredholm operator on $\Vb$ with $\hat Q\in\Jf_2(\Vb)$
and $P$ is a Hilbert--Schmidt operator $P\in\Jf_2(\Vb)$.
Recall from above that for such Fredholm operators, the regularised
Fredholm determinant $\det_2(\id-\hat Q)$ is well-defined, and the
operator $\id-\hat Q$ is invertible if and only if $\det_2(\id-\hat Q)\neq0$.
With this in hand, let $\mathbb W$ denote the subspace of $\Hb$ spanned by the columns of $W$. 
Further we denote by $\Vb_0$ the canonical subspace with the representation
\begin{equation*}
V_0=\begin{pmatrix} \id \\ O \end{pmatrix},
\end{equation*}
where $O$ is the infinite matrix of zeros. The projections $\mathrm{pr}\colon\Wb\to\Vb_0$ and 
$\mathrm{pr}\colon\Wb\to\Vb_0^\perp$ respectively give 
\begin{equation*}
W^\parallel=\begin{pmatrix}
\id-\hat Q\\
O
\end{pmatrix}
\qquad\text{and}\qquad
W^\perp=\begin{pmatrix}
O\\
P
\end{pmatrix}.
\end{equation*}
This projection is achievable if and only $\mathrm{det}_2\,(\id-\hat Q)\neq0$. 
Assume this holds for the moment. We address what happens when it does not hold momentarily.
We observe that the subspace of $\Hb$ spanned by the columns of $W^\parallel$
coincides with the subspace spanned by $V_0$ which is $\Vb_0$.
Ineed the transformation $(\id-\hat Q)^{-1}\in\mathrm{GL}_{\mathrm{HS}}(\Vb)$
transforms $W^\parallel$ to $V_0$. Here $\mathrm{GL}_{\mathrm{HS}}(\Vb)$
is the restricted general linear group of transformations on $\Vb$
consisting of Hilbert--Schmidt perturbations of the identity.
See Remark~\ref{rmk:generallineargroup} for some further discussion
on such groups. Under the transformation, the representation $W$ for $\Wb$ becomes 
\begin{equation*}
\begin{pmatrix} \id \\ G \end{pmatrix},
\end{equation*}
where $G=P(\id-\hat Q)^{-1}$. Any subspace that can be projected onto $\Vb_0$
can be represented in this way and vice-versa. In this representation, the operators
$G\in\mathfrak J_2(\Vb)$ parametrise all the subspaces $\Wb$ that can be projected onto $\Vb_0$.
See Pressley and Segal~\cite{PS} and Segal and Wilson~\cite{SW} for more details.
The Fredholm index of the Fredholm operator `$\id-\hat Q$' is called the virtual dimension of $\Wb$.
When $\det_2(\id-\hat Q)=0$, we cannot project $\Wb$ onto $\Vb_0$. 
This occurs when $\mathrm{dim}(\Wb\cap\Vb_0^\perp)>0$ and corresponds 
to the poor choice of a representative coordinate patch.
Given a subset $\Sb=\{i_1,i_2,\ldots\}\subseteq\mathbb N$ let $\Vb_0(\Sb)$
denote the subspace given by $\mathrm{span}\{\mathfrak e_{i_1},\mathfrak e_{i_2},\ldots\}$.
The vectors $\{\mathfrak e_i\}_{i\in\Sb^{\circ}}$ span the subspace 
$\Vb_0^\perp(\Sb)$ which is orthogonal to $\Vb_0(\Sb)$.
From Pressley and Segal~\cite[Prop.~7.1.6]{PS} we know
for any $\Wb\in\mathrm{Gr}(\Hb,\Vb)$ there exists a set $\Sb\subseteq\mathbb N$
such that the orthogonal projection $\Wb\to\Vb_0(\Sb)$ is an isomorphism. 
The collection of all such subspaces $\Vb_0(\Sb)$ form an open covering
and represent the coordinate charts of $\mathrm{Gr}(\Hb,\Vb)$.
Explicitly the projections $\mathrm{pr}\colon\Wb\to\Vb_0(\Sb)$ 
and $\mathrm{pr}\colon\Wb\to\Vb_0^\perp(\Sb)$ 
would have the representations
\begin{equation*}
W^\parallel=\begin{pmatrix}
W_{\Sb}\\
O_{\Sb^\circ}
\end{pmatrix}
\qquad\text{and}\qquad
W^\perp=\begin{pmatrix}
O_{\Sb}\\
W_{\Sb^\circ}
\end{pmatrix},
\end{equation*}
where $W_{\Sb}$ represents the $\Sb$ rows of $W$ and so forth.
For any subspace $\Wb=\mathrm{span}\{W\}$ there exists 
a set $\Sb$ such that $\mathrm{det}_2(W_{\Sb})\neq0$ and we can make 
a transformation of coordinates $\Vb_0(\Sb)\to\Vb_0(\Sb)$ using
$W_{\Sb}^{-1}\in\mathrm{GL}(\Vb)$ so $W$ becomes 
\begin{equation*}
\begin{pmatrix} \id_{\Sb} \\ G_{\Sb^\circ} \end{pmatrix}.
\end{equation*}
For further details on submanifolds, the stratification
and the Schubert cell decomposition of the Fredholm Grassmannian,
see Pressley and Segal~\cite[Chap.~7]{PS}.
\begin{remark}[Top cell]
As will become clear in our applications, there is a canonical natural encoding for our data,
so that, initially and for a short time at least, the natural parameterisation
for the flow involves the coordinate chart/patch corresponding to
that for the original $W$ we considered above. In other words
we can assume $\Sb=\{1,2,3,\ldots\}$ and $W_{\Sb}=W$ with $\det_2(\id-\hat Q)\neq0$.
We call this canonical coordinate chart the \emph{top cell} and denote it as
$\Lambda_0(\Hb,\Vb)$.
\end{remark}
Let us now consider flows on the Fredholm Grassmannian.
Recall our formal discussion in the introduction and in particular
the `canonical system' we introduced in Definition~\ref{def:nonlinearcanonicalsystem}.
A lot of our applications fit into the linear version of this `canonical system';
namely that in Definition~\ref{prescription:canonicalinearsystem} below.
For example, as we saw in the Introduction, our applications to nonlocal `big matrix' partial differential systems
fit into this prescription, as does our application to a general Smoluchowski-type coagulation equation
with constant frequency kernel in our companion paper \textit{Doikou et al.\/} \cite{DMSW:graphflows}.
In principle our application to the noncommutative potential Korteweg--de Vries equation in Section~\ref{sec:KdV} does
as well, however there is a simpler prescription for this system, as there is for 
our application to the noncommutative nonlinear Schr\"odinger equation in Section~\ref{sec:NLS}.
For both these cases, as well as the noncommutative modified Korteweg--de Vries equation,
see Definition~\ref{prescription:KdVNLS} for the relevant prescription.
For either prescription, i.e.\/ either the `canonical linear system'
in Definition~\ref{prescription:canonicalinearsystem} just below or that in Definition~\ref{prescription:KdVNLS}
in Section~\ref{sec:dispersivelinearPDES}, 
we consider evolutionary flows where the operators $Q=Q(t)$ and $P=P(t)$ evolve in time $t\geqslant0$ as solutions
to the system of linear equations shown. We assume $Q$ to be of the form $Q=\id-\hat Q$ and require both $\hat Q$ 
and $P$ to be Hilbert--Schmidt operators (there are some possible refinements to this in the classical integrable
systems cases as we will see). Assume the Hilbert subspaces $\Vb$ and $\Vb^\perp$ of $\Hb$ are the same, as above. 
With initial data $\hat Q_0, P_0\in\Jf_2(\Vb)$ we assume there exists a $T>0$ such that for $t\in[0,T]$ we can establish
solutions $Q(t)=\id-\hat Q(t)$ and $P=P(t)$, to either of the linear systems
in Definitions~\ref{prescription:canonicalinearsystem} and \ref{prescription:KdVNLS},
such that $\hat Q,P\in C^\infty\bigl([0,T];\Jf_2(\Vb)\bigr)$. (In typical applications we can establish this for all $T>0$.)
With these two operators in hand on $[0,T]$, there exists a smooth path of subspaces $\Wb=\Wb(t)$ of $\Hb$ such that
the projections $\mathrm{pr}\colon\Wb(t)\to\Vb$ and $\mathrm{pr}\colon\Wb(t)\to\Vb^\perp$ are respectively
parametrised by the operators $Q(t)=\id-\hat Q(t)$ and $P(t)$.
In Lemma~\ref{lemma:EUCanonical} just below, we prove that for a finite
time $[0,T^\prime]$ (at least) with possibly $0<T^\prime\leqslant T$, we can parameterise the subspace $\Wb=\Wb(t)$ in the
top cell $\Lambda_0(\Hb,\Vb)$ (or coordinate chart) of the Fredholm Grassmannian $\mathrm{Gr}(\Hb,\Vb)$,
by the Hilbert--Schmidt valued operator $G\in C^\infty\bigl([0,T];\Jf_2(\Vb)\bigr)$.
\begin{lemma}[Existence and Uniqueness: Fredholm equation]\label{lemma:EUCanonical}
  Assume for some $T>0$ we know that $\hat Q\in C^\infty\bigl([0,T];\Jf_2(\Vb)\bigr)$ and
  $P\in C^\infty\bigl([0,T];\Jf_N(\Vb)\bigr)$, where either $N=1$ or $N=2$.
  Further assume that $\mathrm{det}_2(\id-\hat Q_0)\neq0$.
  Then there exists a $T^\prime>0$ with $T^\prime\leqslant T$ such that for
  $t\in[0,T^\prime]$ we have $\mathrm{det}_2(\id-\hat Q(t))\neq0$ and 
  there exists a unique solution $G\in C^\infty\bigl([0,T^\prime];\Jf_N(\Vb)\bigr)$ 
  to the linear Fredholm equation $P=G(\id-\hat Q)$.
\end{lemma}
\begin{proof}
Our proof relies on results from Simon~\cite{Simon:Traces} and Beck \textit{et al.\/} \cite{BDMS2}.
First, since $\hat Q\in C^\infty\bigl([0,T];\Jf_2(\Vb)\bigr)$
and $\mathrm{det}_2(\id-\hat Q_0)\neq0$, by continuity
there exists a $T^\prime>0$ with $T^\prime\leqslant T$ such that 
for $t\in[0,T^\prime]$ we know $\mathrm{det}_2\bigl(\id-\hat Q(t)\bigr)\neq0$.
Second, if $\|\,\cdot\,\|_{\mathrm{op}}$
denotes the operator norm for bounded operators on $\Vb$,
then for any $\ell\in\mathbb N$, $H\in\mathfrak J_N(\Vb)$ and $K\in\mathfrak J_2(\Vb)$ we have: 
$\|HK^\ell\|_{\Jf_N}\leqslant\|H\|_{\Jf_N}\|K\|_{\mathrm{op}}^\ell\leqslant\|H\|_{\Jf_N}\|K\|_{\mathfrak J_2}^\ell$.
Third, we can establish (see Beck \textit{et al.\/} \cite[Lemma 2.4, proof]{BDMS2}) using the identity
$(\id-\hat Q)^{-1}\equiv(\id-\hat Q_0)^{-1}\bigl(\id-(\hat Q-\hat Q_0)(\id-\hat Q_0)^{-1}\bigr)^{-1}$, that
\begin{equation*}
\|P(\id-\hat Q)^{-1}\|_{\Jf_N}\leqslant\|P(\id-\hat Q_0)^{-1}\|_{\Jf_N}
\Bigl(1-\bigl\|(\hat Q-\hat Q_0)(\id-\hat Q_0)^{-1}\bigr\|_{\mathfrak J_2}\Bigr)^{-1}.
\end{equation*}
With this inequality in hand, our goal is to establish $\|P(\id-\hat Q_0)^{-1}\|_{\Jf_N}$ is bounded 
and $\|(\hat Q-\hat Q_0)(\id-\hat Q_0)^{-1}\|_{\mathfrak J_2}<1$.
Fourth, note we can write $(\id-\hat Q_0)^{-1}=\id+\hat Q_*$ where $\hat Q_*\coloneqq\hat Q_0(\id-\hat Q_0)^{-1}$.
Therefore we observe using the ideal property we can establish the two inequalities
\begin{align*}
\|P(\id-\hat Q_0)^{-1}\|_{\Jf_N}&\leqslant\|P\|_{\Jf_N}\bigl(1+\|\hat Q_*\|_{\Jf_2}\bigr),\\
\|(\hat Q-\hat Q_0)(\id-\hat Q_0)^{-1}\|_{\mathfrak J_2}
&\leqslant\|(\hat Q-\hat Q_0)\|_{\mathfrak J_2}\bigl(1+\|\hat Q_*\|_{\mathfrak J_2}\bigr).
\end{align*}
By assumption we know $\|P\|_{\Jf_N}$ is bounded and by continuity $\|(\hat Q-\hat Q_0)\|_{\mathfrak J_2}$ is as small as we require
for sufficiently small times. Fifth, our goal now is to demonstrate $\|\hat Q_*\|_{\mathfrak J_2}$ is bounded.
We define $f$ to be the regularised Fredholm determinant function $f(A)\coloneqq\mathrm{det}_2(\id+A)$ for any $A\in\Jf_2$. 
From Simon~\cite[Theorem~9.2]{Simon:Traces} we know there exists a constant $\Gamma$ such that
$|\mathrm{det}_2(\id+A)|\leqslant\exp(\Gamma\|A\|_{\Jf_2}^2)$.
From Simon~\cite[p.~76]{Simon:Traces} we have the following identity, 
\begin{equation*}
\mathrm{det}_2\bigl((\id+A)(\id+B)\bigr)
=\mathrm{det}_2(\id+A)\mathrm{det}_2(\id+B)\mathrm{e}^{-\mathrm{tr}\,(AB)},
\end{equation*}
for any $A,B\in\Jf_2$. Hence see for any $\mu\in\CC$ we have
\begin{align*}
\mathrm{det}_2(\id+A+\mu B)
&=\mathrm{det}_2(\id+A)\mathrm{det}_2\bigl(\id+\mu(\id+A)^{-1}B\bigr)
\mathrm{e}^{-\mu\,\mathrm{tr}\,(A(\id+A)^{-1}B)}\\
&=\mathrm{det}_2(\id+A)\bigl(1-\mu\,\mathrm{tr}\,(A(\id+A)^{-1}B)+\mathcal O(\mu^2)\bigr),
\end{align*}
after using the definition of the regularised determinant to expand the second determinant factor above.
Hence the derivative of $f$ is $Df(A)=-A(\id+A)^{-1}\mathrm{det}_2(\id+A)$.
Further from Simon~\cite[Theorem~5.1]{Simon:Traces} we deduce $\|Df(A)\|_{\Jf_2}\leqslant\exp(\Gamma(\|A\|_{\Jf_2}+1)^2)$.
Finally since $\hat Q_\ast=-Df(-\hat Q_0)/\mathrm{det}_2(\id-\hat Q_0)$
we can establish the required bound for $\|\hat Q_*\|_{\Jf_2}$ and the conclusion for $G=P(\id-\hat Q)^{-1}$ follows.
\qed
\end{proof}
The `canonical linear system' we mentioned above is as follows.
\begin{definition}[Prescription: Canonical linear system]\label{prescription:canonicalinearsystem}
Assume for given data $\hat Q_0, P_0\in\Jf_2(\Vb)$ with $\mathrm{det}_2(\id-\hat Q_0)\neq0$,
there exists a $T>0$ such that $\hat Q,P,G\in C^\infty\bigl([0,T];\Jf_2(\Vb)\bigr)$,
satisfy the \emph{canonical linear system} of equations (with $Q=\id-\hat Q$),
\begin{align*}
\pa_tQ&=AQ+BP,\\
\pa_tP&=CQ+DP,\\
P&=G(\id-\hat Q),
\end{align*}
with $\hat Q(0)=\hat Q_0$ and $P(0)=P_0$. Here the operators $A=A(t)$, $B=B(t)$, $C=C(t)$ and $D=D(t)$ may be
either bounded or unbounded operators.
\end{definition}
We call the evolution equations for $Q$ and $P$ the \emph{linear base equations}
and the linear integral equation defining $G=G(t)$ the \emph{linear Fredholm equation}.
Assuming we have established suitable solutions $Q=Q(t)$ and $P=P(t)$
as indicated, then Lemma~\ref{lemma:EUCanonical} implies the following.
\begin{corollary}[Canonical decomposition]
Assume the operators $Q$, $P$ and $G$ satisfy the canonical prescription in Definition~\ref{prescription:canonicalinearsystem}
and the conditions stated therein. Assume further that we know that $A,C,BG,DG\in C^\infty\bigl([0,T];\Jf_2(\Vb)\bigr)$ for
the time $T>0$ stated in the Canonical Prescription. Then the results of the existence and uniqueness Lemma~\ref{lemma:EUCanonical}
imply there exists a $T'>0$ with $T'\leqslant T$ such that for all $t\in[0,T']$
the operator $G\in C^\infty\bigl([0,T'];\Jf_2(\Vb)\bigr)$ satisfies (with $G_0=P_0(\id-\hat Q_0)^{-1}$),
\begin{equation*}
\pa_t G=C+DG-G(A+BG).
\end{equation*}
\end{corollary}
\begin{proof}
Differentiating the relation $P=GQ$ with respect to time using the product rule, using the base equations and 
feeding through the relation $P=GQ$ once more generates $(\pa_tG)Q=\pa_tP-G\pa_tQ=(C+DG)Q-G(A+BG)Q$. 
Postcomposing by $Q^{-1}$, which exists at least for some finite time interval, thus establishes the result.
\qed
\end{proof}
\begin{remark}[Canonical matching]\label{rmk:canonicalmatching}
  In the introduction we matched many examples to the canonical linear system above.
  For example we matched most of the nonlocal `big matrix' partial differential systems as well as
  a general constant frequency Smoluchowski-type model.   Examples such as the potential Korteweg de Vries equation
  or the nonlocal `big matrix' and local nonlinear Schr\"odinger equations, require a slightly more general form.
  For all the classical integrable systems we consider herein such as the noncommutative potential Korteweg de Vries
  and nonlinear Schr\"odinger equations, using the result of Lemma~\ref{lemma:EUCanonical},
  we establish local in time existence, uniqueness and well-posedness, separately. 
\end{remark}
We round off this section with some observations concerning the Fredholm Grassmannian we defined above.
\begin{remark}[General linear group]\label{rmk:generallineargroup}
There is a natural group action on $\mathrm{Gr}(\Hb,\Vb)$.
This group is the restricted general linear group $\mathrm{GL}_{\mathrm{HS}}(\Hb)$ of Hilbert--Schmidt perturbations
of the identity, a subgroup of $\mathrm{GL}(\Hb)$.
Elements of $S\in\mathrm{GL}_{\mathrm{HS}}(\Hb)$ can be expressed in the block form
\begin{equation*}
S=\begin{pmatrix} s_{11} & s_{12}\\ s_{21} & s_{22}\end{pmatrix},
\end{equation*}
which respects the decomposition $\Hb=\Vb\oplus\Vb^\perp$, and for which
$s_{11}$ and $s_{22}$ are Fredholm operators which are Hilbert--Schmidt perturbations 
of the identity, while $s_{12}$ and $s_{21}$ are Hilbert--Schmidt operators;
see Pressley and Segal~\cite[p.~80]{PS}. Note by the ideal property of
of Hilbert--Schmidt operators if $R=\id+\hat R$ and $S=\id+\hat S$
are two elements in $\mathrm{GL}_{\mathrm{HS}}((\Hb)$ then $RS=\id+\hat R+\hat S+\hat R\hat S$
also lies in $\mathrm{GL}_{\mathrm{HS}}((\Hb)$. Further suppose $\sigma\in\Jf_2(\Hb)$
then $\exp\sigma\in\mathrm{GL}_{\mathrm{HS}}((\Hb)$. To see this we observe, using that
$\exp\sigma=\id+\sigma+\frac{1}{2!}\sigma^2+\frac{1}{3!}\sigma^3+\cdots$, and the  
ideal property of $\Jf_2(\Hb)$ and the associated inequalities we outlined at the beginning of
this section, we have $\|\exp\sigma-\id\|_{\Jf_2(\Hb)}\leqslant\exp\bigl(\|\sigma\|_{\Jf_2(\Hb)}\bigr)-1$.
Since the upper bound shown is bounded we deduce $\exp\sigma\in\mathrm{GL}_{\mathrm{HS}}((\Hb)$.
Consequently, in this context we define the general linear algebra on $\Hb$ as follows:
\begin{equation*}
\mathfrak{gl}(\Hb)\coloneqq\bigl\{\sigma\in\Jf_2(\Hb)\colon\exp(\mu\sigma)\in\mathrm{GL}_{\mathrm{HS}}((\Hb),~\text{for all}~\mu\in\R\bigr\}.
\end{equation*}
We observe $\mathfrak{gl}(\Hb)$ is closed under addition and commutation.
Further, for any two elements $\sigma_1,\sigma_2\in\mathfrak{gl}(\Hb)$, we have
$\exp(\mu\sigma_1)\exp(\mu\sigma_2)=\exp\bigl(\mathrm{BCH}(\sigma_1,\sigma_2)\bigr)$, where
\begin{equation*}
\mathrm{BCH}(\sigma_1,\sigma_2)=\sigma_1+\sigma_2+\tfrac12[\sigma_1,\sigma_2]
+\tfrac{1}{12}[\sigma_1,[\sigma_1,\sigma_2]]-\tfrac{1}{12}[\sigma_2,[\sigma_1,\sigma_2]]+\cdots,
\end{equation*}
is the \emph{Baker--Campbell--Hausdorff series}, converging for
$\|\sigma_1\|_{\Jf_2(\Hb)}+\|\sigma_2\|_{\Jf_2(\Hb)}<\frac12\log 2$.
Hence the algebra structure of $\mathfrak{gl}(\Hb)$ locally determines that of $\mathrm{GL}_{\mathrm{HS}}(\Hb)$,
and models the tangent space at the origin of $\mathrm{GL}_{\mathrm{HS}}((\Hb)$.
\end{remark}
\begin{remark}[Principle fibre bundle]\label{rmk:fibrebundle}
  Formally, and by analogy with the finite dimensional case, the results just above suggest
  the existence of the following structure (we do not prove this here). We define the 
  Fredholm Stiefel manifold $\mathrm{St}(\Hb,\Vb)$ as the set of independent sequences
  in $\Hb\coloneqq\Vb\times\Vb$ which span $\Vb$ of the form
  \begin{equation*}
    W=\begin{pmatrix} \id-\hat Q \\ P \end{pmatrix},
  \end{equation*}
  where $\hat Q, P\in\mathfrak J_2(\Vb)$. Then the Fredholm Stiefel manifold $\mathrm{St}(\Hb,\Vb)$ has a principle fibre
  bundle structure so that $\mathrm{GL}_{\mathrm{HS}}(\Vb)\to\mathrm{St}(\Hb,\Vb)\to\mathrm{Gr}(\Hb,\Vb)$.
  Here, as in the finite dimensional case, we think of $\mathrm{Gr}(\Hb,\Vb)$ as the \emph{base space},
  and with reference to the bundle projection $\pi\colon\mathrm{St}(\Hb,\Vb)\to\mathrm{Gr}(\Hb,\Vb)$,
  at each element in $\mathrm{Gr}(\Hb,\Vb)$, the inverse image $\pi^{-1}$ of that element
  is homeomorphic to the fibre space $\mathrm{GL}_{\mathrm{HS}}(\Vb)$. Further, we can consider the 
  associated tangent bundle. Consider the induced decomposition of the tangent space
  $T_W\mathrm{St}(\Hb,\Vb)$ at $W\in\mathrm{St}(\Hb,\Vb)$. We can decompose $T_W\mathrm{St}(\Hb,\Vb)$
  into horizontal and vertical subspaces, $T_W\mathrm{St}(\Hb,\Vb)=\mathbb T_W^\parallel\oplus\mathbb T_W^\perp$. 
  Here the horizontal subspace $\mathbb T_W^\parallel$ is associated with the tangent space
  of the Fredholm Grassmannian base space, while the vertical subspace $\mathbb T_W^\perp$
  is associated with the fibres homeomorphic to $\mathrm{GL}_{\mathrm{HS}}(\Vb)$.
  Consider a coordinate patch representation characterised by the subset $\Sb=\{i_1,i_2,\ldots\}\subseteq\mathbb N$,
  as above, so $\Vb_0(\Sb)$ denotes the subspace $\mathrm{span}\{\mathfrak e_{i_1},\mathfrak e_{i_2},\ldots\}$
  and the vectors $\{\mathfrak e_i\}_{i\in\Sb^{\circ}}$ span the subspace $\Vb_0^\perp(\Sb)$, orthogonal to $\Vb_0(\Sb)$.
  We can additively decompose any tangent vector $V=\mathrm{pr}_{\Sb}V+\mathrm{pr}_{\Sb^\circ}V$ and so,
  \begin{align*}
    \mathbb T_W^\parallel&=\bigl\{\mathrm{pr}_{\Sb}V\colon V\in T_W\mathrm{St}(\Hb,\Vb)\bigr\}\cong\Jf_2(\Vb),\\
    \mathbb T_W^\perp&=\bigl\{\mathrm{pr}_{\Sb^\circ}V\colon V\in T_W\mathrm{St}(\Hb,\Vb)\bigr\}\cong\mathfrak{gl}(\Vb).
  \end{align*}
  We saw in the introduction, in the discussion following Definition~\ref{def:nonlinearcanonicalsystem} for the canonical system,
  how this decomposition generates a coupled flow, for the field $Q_{\Sb}$ through the fibres
  and the field $G_{\Sb^\circ}$ through the base space.
\end{remark}
\begin{remark}[Integrable systems and Hankel subflows]
  When we study integrable systems in the subsequent sections, we consider evolutionary linear subflows
  of Hilbert--Schmidt operators which are Hankel operators. We discuss Hankel operators in
  detail in Section~\ref{sec:Hankelandlinear}, and evolutionary Hilbert--Schmidt Hankel subflows
  in detail in Section~\ref{sec:dispersivelinearPDES}.
\end{remark}

\section{Hilbert--Schmidt and Hankel operators}\label{sec:Hankelandlinear} 
We outline the main analytical ingredients we need. We consider Hilbert--Schmidt operators
which depend on both a spatial parameter $x\in\R$ and a time parameter $t\in[0,\infty)$.
In this section $\pa_t$ represents the partial derivative with respect to time while $\pa=\pa_x$
represents the partial derivative with respect to the parameter $x\in\R$.
Let $\Vb$ be the Hilbert space of square-integrable, complex matrix-valued functions on $(-\infty,0]$,
i.e.\/ $\Vb\coloneqq L^2\bigl((-\infty,0];\CC^{m}\bigr)$ for some $m\in\mathbb N$.
For any given Hilbert--Schmidt operator $F=F(x,t)\in\mathfrak J_2(\Vb)$
there exists a unique square-integrable kernel $f=f(y,z;x,t)$ with
$f\in L^2\bigl((-\infty,0]^{\times 2};\CC^{m\times m}\bigr)$
such that for any $\phi\in\Vb$ we have
\begin{equation*}
(F\phi)(y;x,t)=\int_{-\infty}^0f(y,z;x,t)\phi(z)\,\rd z.
\end{equation*}
Conversely, any such function $f\in L^2((-\infty,0]^{\times 2};\CC^{m\times m})$
defines an operator $F=F(x,t)$ in $\mathfrak J_2(\Vb)$ with (for each $x,t$):
\begin{equation*}
\|F\|_{\mathfrak J_2(\Vb)}=\|f\|_{L^2((-\infty,0]^{\times 2};\CC^{m\times m})}. 
\end{equation*}
See for example Simon~\cite[p.~23]{Simon:Traces}.
\begin{definition}[Kernel bracket]\label{def:Kernelbracket}
With reference to the Hilbert--Schmidt operator $F=F(x,t)$ just above, we
use the \emph{kernel bracket} notation $[F]$ to refer to the kernel $f=f(y,z;x,t)$
of $F$:
\begin{equation*}
[F](y,z;x,t)=f(y,z;x,t). 
\end{equation*}
For general Hilbert--Schmidt operators, $f$ only exists almost everywhere on $(-\infty,0]^{\times 2}$.
However in all applications below the operators we consider will have continuous kernels and
so $f$ makes sense pointwise, and in particular at $y=z=0$.
\end{definition}
Hilbert--Schmidt \emph{Hankel} operators play an essential role
in the inverse scattering prescription of integrable systems
such as the Korteweg--de Vries and nonlinear Schr\"odinger equations. 
Indeed, they are the crucial ingredient in their prescription as Fredholm Grassmannian flows. 
We now define what we mean by Hankel (or additive) operators with parameters,
prove a key product rule which relies on the properties of such operators, and
derive further properties associated with such operators that will be useful in
our Fredholm Grassmannian prescription of Korteweg--de Vries and nonlinear Schr\"odinger flows.
\begin{definition}[Hankel operator with parameter]\label{def:Hankel}
We say a time-dependent Hilbert--Schmidt operator $H\in\mathfrak J_2(\Vb)$ 
with corresponding square-integrable kernel $h$ 
is \emph{Hankel} or \emph{additive} with a parameter $x\in\R$ 
if its action, for any square-integrable function $\phi\in\Vb$, 
is given by (here $y\in(-\infty,0]$),
\begin{equation*}
(H\phi)(y;x,t)\coloneqq\int_{-\infty}^0h(y+z+x;t)\phi(z)\,\rd z.
\end{equation*}
\end{definition}
The crucial role played by such Hankel operators in the prescription and solution formulae
for classical integrable systems was recognised by P\"oppe~\cite{P-SG,P-KdV}. 
It is P\"oppe's formulation which we follow and generalise herein.
The key property of such kernels that we utilise is that the
kernel associated with the derivative $\pa=\pa_x$ of the product of
an arbitrary pair of such Hankel operators can be expressed as the matrix product
of the of their respective kernels. In particular for a serial composition
of operators of the form $FHH^\prime F^\prime$, where $H$ and $H^\prime$ are 
Hankel operators, we have the following key product rule given by P\"oppe~\cite{P-SG,P-KdV}. 
\begin{lemma}[Product rule]\label{lemma:productrule}
Assume $H$ and $H^\prime$ are \emph{Hankel} Hilbert--Schmidt operators with parameter $x\in\R$,
and $F$ and $F^\prime$ are Hilbert--Schmidt operators, all on $\Vb$.
Assume further that the kernels corresponding
to $F$ and $F^\prime$ are continuous and those corresponding to $H$ and $H^\prime$
are continuously differentiable, on $(-\infty,0]^{\times 2}$.
Then the following \emph{product rule} holds for all $(y,z)\in(-\infty,0]^{\times 2}$:
\begin{equation*}
[F\pa_x(HH^\prime)F^\prime](y,z;x)=\bigl([FH](y,0;x)\bigr)\bigl([H^\prime F^\prime](0,z;x)\bigr).
\end{equation*}
\end{lemma}
\begin{proof}
This is a consequence of the fundamental theorem
of calculus and Hankel properties of $H$ and $H^\prime$.
Let $f$, $h$, $h^\prime$ and $f^\prime$ denote the integral kernels 
of $F$, $H$, $H^\prime$ and $F^\prime$ respectively.
By direct computation $[F\pa_x(HH^\prime)F^\prime](y,z;x)$ equals
\begin{align*}
&\int_{(-\infty,0]^{\times 3}}
f(y,\xi_1)\pa_x\bigl(h(\xi_1+\xi_2+x)h^\prime(\xi_2+\xi_3+x)\bigr)
f^\prime(\xi_3,z)\,\rd \xi_3\,\rd \xi_2\,\rd \xi_1\\
&=\int_{(-\infty,0]^{\times 3}}
f(y,\xi_1)\pa_{\xi_2}\bigl(h(\xi_1+\xi_2+x)h^\prime(\xi_2+\xi_3+x)\bigr)
f^\prime(\xi_3,z)\,\rd \xi_3\,\rd \xi_2\,\rd \xi_1\\
&=\int_{(-\infty,0]^{\times 2}}
f(y,\xi_1)h(\xi_1+x)h^\prime(\xi_3+x)f^\prime(\xi_3,z)\,\rd \xi_3\,\rd \xi_1\\
&=\int_{(-\infty,0]}f(y,\xi_1)h(\xi_1+x)\,
\rd \xi_1\cdot\int_{(-\infty,0]}h^\prime(\xi_3+x)f^\prime(\xi_3,z)\,\rd \xi_3\\
&=\bigl([FH](y,0;x)\bigr)\bigl([H^\prime F^\prime](0,z;x)\bigr),
\end{align*}
giving the result. 
\qed
\end{proof}
\begin{remark}
  We naturally assume the complex-matrix valued kernels corresponding to the operators
  in the statement of Lemma~\ref{lemma:productrule} are commensurate so the compositions
  of the operators shown makes sense.
\end{remark}
As we stated above, a given matrix-valued function on $(-\infty,0]^{\times 2}$
generates a Hilbert--Schmidt operator on $\Vb$ iff the matrix-valued function
lies in $L^2\bigl((-\infty,0]^{\times 2};\CC^{m\times m}\bigr)$. This provides
a simple test for whether a given operator is Hilbert--Schmidt valued, or conversely,
if a given matrix-valued function generates a Hilbert--Schmidt operator.
If the Hilbert--Schmidt operator in question is a Hankel operator, or
we are given a function on $(-\infty,0]$ and we wish to know if it
generates a Hilbert--Schmidt Hankel operator, then the following 
result provides another simple test. Note, for a non-negative function $w$,
we denote the weighted $L^2$ norm of any complex matrix-valued function
$f$ on the domains $(-\infty,0]$ or $\R$ by
\begin{equation*}
\|f\|_{L_w}^2\coloneqq\int_{-\infty}^0\mathrm{tr}\,\bigl(f^\dag(x)f(x)\bigr)\,w(x)\,\rd x,
\end{equation*}
with the integral over $\R$ for the corresponding domain.
%
\begin{lemma}[Test for a Hilbert--Schmidt Hankel operator]\label{lemma:HankelHS}
Let $w_2\colon(-\infty,0]\to(0,\infty)$ denote the weight function $w_2\colon x\mapsto(1-x)^2$.
Suppose the matrix-valued function $h\in L^2_{w_2}((-\infty,0];\CC^{m\times m})$.
Then the Hankel operator $H$ generated by $h$ on $(-\infty,0]$ is Hilbert--Schmidt valued,
i.e.\/ $H\in\mathfrak J_2(\Vb)$ with $\Vb\coloneqq L^2((-\infty,0];\CC^{m})$.
\end{lemma}
\begin{proof}
By a standard change of variables $\xi=y+z$ and $\eta=y-z$ 
we have (see for example Power~\cite{Power})
\begin{align*}
\int_{-\infty}^0\int_{-\infty}^0\mathrm{tr}\,\bigl(h^\dag(y+z;t)h(y+z;t)\bigr)\,\rd y\,\rd z
&=\int_{-\infty}^0\int_{\xi}^{-\xi}\mathrm{tr}\,\bigl(h^\dag(\xi;t)h(\xi;t)\bigr)\,\rd \eta\,\rd\xi\\
&=\int_{-\infty}^0\mathrm{tr}\,\bigl(h^\dag(\xi;t)h(\xi;t)\bigr)\,(2|\xi|)\,\rd\xi\\
&\leqslant 4\,\|h\sqrt{w_2}\|_{L^2((-\infty,0];\CC^{m\times m})}^2.
\end{align*}
Thus if $h\in L^2_{w_2}((-\infty,0];\CC^{m\times m})$
then we observe that the corresponding Hankel operator $H\in\mathfrak J_2(\Vb)$.
\qed
\end{proof}
Proving that a given kernel function generates a trace class operator is generally much
more involved. See Smithies~\cite{Smithies}, Simon~\cite[p.~23--4]{Simon:Traces},
as well as Bornemann~\cite[p.~878--9]{Bornemann}, who provides a useful list of criteria.

We round off this section with an important result which will help to reveal the connection
between the bracket operator, in particular when the image kernel is evaluated at $(y,z)=(0,0)$,
and classical expressions for the solution to the Korteweg--de Vries equation in terms
of Fredholm determinants. We show this connection at the end of Section~\ref{sec:KdV}.
We adapt and extend results that can be found in P\"oppe~\cite{P-KdV}.
To state our first result we need to define the linear \emph{quasi-trace} operator as follows.
\begin{definition}[Quasi-trace]\label{def:quasi-trace}
  Suppose that $F\in\Jf_1(\Vb)$ is a trace-class operator with matrix valued kernel $f$.
  Then we define the \emph{quasi-trace} `$\mathrm{qtr}\,F$' of $F$ to be
  \begin{equation*}
    \mathrm{qtr}\,F\coloneqq\int_{-\infty}^0f(y,y)\,\rd y,
  \end{equation*}
  provided the integral shown of each of the matrix components of $f$ is finite.
\end{definition}
The usual trace of $F$ would involve the matrix trace of $f$ in the integrand in the definition.
For scalar kernels, naturally the quasi-trace `$\mathrm{qtr}$' and the usual operator trace `$\mathrm{tr}$' coincide.
For operators $F$ with continuous kernels $f=f(y,z;x,t)$, which may for example also depend on the parameters $x\in\R$ and $t\geqslant0$,
we define the bracket operation $[F]_{0,0}$ to be $f(0,0;x,t)$. In other words we can think of $[\,\cdot\,]_{0,0}$
as the composition of the bracket operator and evaluating the kernel at $y=z=0$:
\begin{equation*}
[F]_{0,0}=\left.f(y,z)\right|_{y=z=0}.
\end{equation*}
\begin{remark}[Continuous kernels]
All the operators we consider herein, to which we apply the bracket operator, have continuous kernels. 
\end{remark}
Our result is as follows. 
\begin{lemma}[Quasi-trace identity]\label{lemma:quasi-traceidentity}
Assume $H$ and $H^\prime$ are \emph{Hankel} Hilbert--Schmidt operators,
and $F$ is a Hilbert--Schmidt operator, all on $\Vb$ and
with parameter $x$. Assume the kernel corresponding
to $F$ is continuous  on $(-\infty,0]^{\times 2}$ and those corresponding to $H$ and $H^\prime$
are continuously differentiable on $(-\infty,0]$.
Then we have the following \emph{quasi-trace identity}:
\begin{equation*}
[HFH^\prime]_{0,0}\equiv\mathrm{qtr}\,\bigl((\pa_xH)FH^\prime+HF(\pa_xH^\prime)\bigr).
\end{equation*}
\end{lemma}
\begin{proof}
We observe that the right-hand side equals
\begin{align*}
  \int_{-\infty}^0\int_{-\infty}^0\int_{-\infty}^0
  \bigl(\pa_xh(y+&z+x)\,f(z,\zeta;x)\,h^\prime(\zeta+y+x)\\
  &+h(y+z+x)\,f(z,\zeta;x)\,\pa_xh^\prime(\zeta+y+x)\bigr)\,\rd z\rd\zeta\rd y.
\end{align*}
We can replace both partial derivatives with respect to $x$ shown in the integrand by partial derivatives 
with respect to $y$. If we then focus on the integral involving the first term in the integrand and use
the partial integration formula, the boundary term generated equals the left-hand side quantity in the
identity stated in the lemma, while the integral term generated exactly cancels the second term in the 
integrand shown.
\qed
\end{proof}
%

\section{Linear dispersive partial differential equations}\label{sec:dispersivelinearPDES}
In our main applications, we construct Hankel operators from the solutions to linear dispersive partial differential
systems representing the linearised versions of equations from the noncommutative Korteweg-de Vries and nonlinear Schr\"odinger
hierarchies. Hence in general consider complex matrix-valued solutions $p=p(\,\cdot\,;t)$ to the
following linear dispersive partial differential system:
\begin{equation*}
\pa_tp=-\mu_n(\mathrm{i}\mathcal I)^{n-1}\pa^np,
\end{equation*}
where the constant $\mu_n\in\R$ and the block diagonal matrix $\mathcal I$ has the form,
\begin{equation*}
\mathcal I\coloneqq\begin{pmatrix} -\id & O\\ O & \id\end{pmatrix}.
\end{equation*}
Note that the identity `$\id$' and zero `$O$' operators shown are matrix-valued.
Indeed we suppose $p$ is $m\times m$ matrix-valued, with $m$ \emph{even}.
Further note that $\mathcal I^2=\id$, the $m\times m$ matrix-valued identity operator.
For convenience we denote the operator shown as,
\begin{equation*}
d(\pa)\coloneqq-\mu_n(\mathrm{i}\mathcal I)^{n-1}\pa^n.
\end{equation*}
This operator satisfies the \emph{dispersive property} $(d(\mathrm{i}\kappa))^\dag=-d(\mathrm{i}\kappa)$
for any $\kappa\in\R$, since $\mathcal I^\dag=\mathcal I$. 
In principle we could consider more general forms for $d=d(\pa)$ which satisfy
the dispersive property. For example, we could consider constant coefficient
polynomial forms for $d=d(\pa)$ satisfying the dispersive property; see for example Malham~\cite{Malham:quinticNLS}
However the form above is sufficiently general for our purposes herein.
We note that if $n=2$, then $\pa_tp=-\mathrm{i}\mu_2\mathcal I\pa^2p$.
This matches the form for the linearised noncommutative nonlinear Schr\"odinger equation,
with the scalar case corresponding to $m=1$; see Section~\ref{sec:NLS} for how this is interpreted.
If $n=3$, then $\pa_tp=\mu_3I_{2m}\pa^3p$. This matches the form for the
linearised noncommutative, either modified or standard, Korteweg--de Vries equation.
The scalar case corresponds to $m=1$ for the modified Korteweg--de Vries equation; again see Section~\ref{sec:NLS}.
For $n=4$, we get $\pa_tp=\mathrm{i}\mu_4\mathcal I\pa^4p$. This matches
the quartic order linearised noncommutative nonlinear Schr\"odinger equation, and so forth.
With regards the Korteweg--de Vries equation itself, suppose $n$ is odd so that $n=2\ell+1$.
Then the equation for $p$ becomes $\pa_tp=(-1)^\ell\mu_{2\ell+1}I_{2m}\pa^{2\ell+1}p$, and indeed, there
is no need to restrict ourselves to even $m$ but can suppose $m\in\mathbb N$.
Alternatively, in the setting of Section~\ref{sec:NLS} with $m$ even,
we can suppose we have two exact copies of the equation present
(take $P_\beta\equiv P_\alpha$ therein, and so forth).

We have already seen from Lemma~\ref{lemma:HankelHS} that, if we know
$p(\,\cdot\,;t)\in L^2_{w_2}((-\infty,0];\CC^{m\times m})$ with $w_2\colon y\mapsto(1-y)^2$, then the 
Hankel operator $P=P(t)$ generated by $p$ is such that $P(t)\in\mathfrak J_2(\Vb)$
with $\Vb=L^2((-\infty,0];\CC^m)$.
We need to construct Hankel operators $P=P(x;t)$ with kernels of the form $p=p(y+z+x;t)$ from
solutions $p=p(y;t)$ to the linear system $\pa_tp=-\mu_n(\mathrm{i}\mathcal I)^{n-1}\pa_y^np$,
where $(y,z)\in(-\infty,0]^{\times2}$, with the parameter $x\in\R$.
Naturally any statements we make for $p(\,\cdot\,;t)$ on $(-\infty,0]$,
\emph{equivalently translate}, for each $x\in\R$, to statements for $p=p(\,\cdot\,+x;t)$ on $(-\infty,x]$. 
This subtlety is important, as some natural solutions $p=p(y;t)$ to the linear dispersive partial differential system above
that occur, in particular those that generate soliton solutions as we see in Section~\ref{sec:KdV} 
and Appendix~\ref{sec:scatteringproblem}, are unbounded as $y\to+\infty$. In fact, those corresponding
solutions $p=p(y;t)$ that generate soliton solutions, grow exponentially.
In particular consider the solutions $p=p(y;t)$ of the form $p(y;t)=2A\exp(Ay+A^3t)$, for some constant square matrix $A$,
all of whose entries are positive. Such solutions satisfy $p(\,\cdot\,;t)\in L^2_{w_2}((-\infty,0];\CC^{m\times m})$,
or equivalently $p(\,\cdot\,+x;t)\in L^2_{w_2}((-\infty,x];\CC^{m\times m})$ for each $x\in\R$.
The corresponding Hankel operator $P=P(x;t)$ constructed from such a solution is thus Hilbert--Schmidt valued.   
Indeed, more generally via this argument, solutions $p=p(y;t)$ that grow as $y\to+\infty$, are thus not
precluded from generating Hilbert--Schmidt or trace class operators $P(x;t)$, provided they satisfy
the respective regularity conditions on $(-\infty,0]$ stated above.

Establishing regularity conditions for solutions to the linear system of equations
$\pa_tp=-\mu_n(\mathrm{i}\mathcal I)^{n-1}\pa^np$ on $(-\infty,0]$ is not an endeavour we pursue herein,
other than through important case-by-case observations such as for $p(y;t)=2A\exp(Ay+A^3t)$, as just mentioned,
or further cases mentoned in Section~\ref{sec:KdV} and Appendix~\ref{sec:scatteringproblem}, as well as
those which are profiled in detail for the scalar case for the Korteweg--de Vries equation by P\"oppe~\cite{P-KdV}.
We discuss these in Section~\ref{sec:KdV}. However if we focus on solutions to the linear system
on the whole real line $\R$, establishing regularity results for solutions $p=p(y;t)$ is more straightforward,
as we see in Lemma~\ref{lemma:dispersivePDE} just below. Naturally any regularity properties we establish for
$p=p(\,\cdot\,;t)$ on $\R$ are inhereted by $p=p(y+z+x;t)$. Indeed we can establish quite general conditions
for which the solutions $p=p(\,\cdot\,;t)$ generate corresponding Hilbert--Schmidt operators $P(x;t)$.
For any given integrable complex matrix-valued function $f\in L^1(\R;\CC^{m\times m})$,
we denote its Fourier transform by $\mathfrak f=\mathfrak f(k)$
and define its inverse by the respective formulae,
\begin{equation*}
\mathfrak f(k)\coloneqq\int_\R f(x)\mathrm{e}^{2\pi\mathrm{i}kx}\,\rd x
\qquad\text{and}\qquad
f(x)\coloneqq\int_\R \mathfrak f(k)\mathrm{e}^{-2\pi\mathrm{i}kx}\,\rd k.
\end{equation*}
The following result summarises the properties of solutions to
linear dispersive system above on the whole real line $\R$.
A similar result is stated in Doikou \textit{et al.\/} \cite{DMS}.
For any $\ell\in\mathbb N$, $H^\ell(\R;\CC^{m\times m})$ denotes
the space of square-integrable $\CC^{m\times m}$-valued functions, all of
whose derivatives up to and including the $\ell$th derivative, are also square-integrable.
\begin{lemma}[Dispersive linear equation properties on the real line]\label{lemma:dispersivePDE}
Assume the $\CC^{m\times m}$-valued function $p=p(y;t)$ is a solution to the general 
dispersive linear partial differential equation $\pa_tp=-\mu_n(\mathrm{i}\mathcal I)^{n-1}\pa_y^np$ on $\R$.
Let $w=w(\,\cdot\,)$ denote an arbitrary real polynomial function on the real line with constant non-negative coefficients, 
and let $\wf=\wf(k)$ denote the Fourier transform of the operator $w(\pa_y)$ that acts multiplicatively in Fourier space,
i.e.\/ the Fourier transform of $w(\pa_y)p(y;t)$ is $\wf(k)\pf(k;t)$, where $\pf=\pf(k;t)$ is the Fourier transform of $p=p(y;t)$.
Further, let $W\colon\R\to(0,\infty)$ denote the specific function $W\colon y\mapsto1+y^2$. Then $p=p(y;t)$ and $\pf=\pf(k;t)$
satisfy the following properties for all $k\in\R$ and $t\geqslant0$
($d'$ below is the derivative of $d$):\smallskip

(i) $\|\pa\pf\|_{L^2}^2\leqslant(2\pi)^2\|W^{1/2}p\|_{L^2}^2$;\smallskip

(ii) $p(0)\in H^\infty(\R;\CC^{m\times m})$ $\Rightarrow$ $\pf(0)\in L^2_{\wf}(\R;\CC^{m\times m})$;\smallskip

(iii) $p(0)\in L^2_W(\R;\CC^{m\times m})$ $\Rightarrow$ $\pf(0)\in H^1(\R;\CC^{m\times m})$;\smallskip

(iv) $\pf(k;t)=\exp\bigl(td(2\pi\mathrm{i}k)\bigr)\pf(k;0)$;\smallskip

(v) $\pf^\dag(k;t)\pf(k;t)=\pf^\dag(k;0)\pf(k;0)$;\smallskip

(vi) $\|w(\pa)p(t)\|_{L^2}^2=\|\wf\pf(t)\|_{L^2}^2=\|\wf\pf(0)\|_{L^2}^2=\|w(\pa)p(0)\|_{L^2}^2$;\smallskip

(vii) $p(0)\in H^\infty(\R;\CC^{m\times m})$ $\Rightarrow$ $p(t)\in H^\infty(\R;\CC^{m\times m})$;\smallskip

(viii) $\|\pa\pf(t)\|_{L^2}^2\leqslant 2\bigl((2\pi)^2t^2\|d^\prime\pf(0)\|_{L^2}^2+\|\pa\pf(0)\|_{L^2}^2\bigr)$;\smallskip

(ix) $\pf(0)\in H^1(\R;\CC^{m\times m})\cap L^2_{|d^\prime|^2}(\R;\CC^{m\times m})$ $\Rightarrow$ $\pf(t)\in H^1(\R;\CC^{m\times m})$;\smallskip

(x) $\|W^{1/2}p(t)\|_{L^2}^2=\|p(0)\|_{L^2}^2+(2\pi)^{-2}\|\pa\pf(t)\|_{L^2}^2$;\smallskip

(xi) $p(0)\in L^2_W(\R;\CC^{m\times m})\cap H^{\mathrm{deg}(d^\prime)}(\R;\CC^{m\times m})$ $\Rightarrow$ $p(t)\in L^2_W(\R;\CC^{m\times m})$.
\end{lemma}
\begin{proof}
Let $V\colon\R\to[0,\infty)$ denote the function $V\colon y\mapsto y^2$.
We establish the results in order as follows: 
(i) By the Plancherel Theorem and the definition of the Fourier transform
we observe $\|\pa\pf\|_{L^2}^2=(2\pi)^2\|V^{1/2}p\|_{L^2}^2$ which we 
then combine with the fact $\|V^{1/2}p\|_{L^2}^2\leqslant\|W^{1/2}p\|_{L^2}^2$;
(ii) This follows again by the Plancherel Theorem and standard properties
of the Fourier transform;
(iii) This follows from (i) applied at time $t=0$;
(iv) and (v) These follow by directly solving the linear differential
equation for the general dispersive equation in Fourier space;
(vi) The first and third equalities follow by the Plancherel theorem, 
while the second uses (v); (vii) Follows from (vi) and that $w$ is arbitrary;
(viii) The derivative with respect to $k$ of the explicit solution from (iv)
generates $\pa\pf(t)=\exp\bigl(td(2\pi\mathrm{i}k)\bigr)(t(2\pi\mathrm{i})d^\prime\pf(k;0)+\pa\pf(k;0))$,
where $d^\prime$ denotes the derivative of $d$. Taking the complex conjugate
of this and using both expressions to expand $\pa\pf^\dag(t)\pa\pf(t)$
generates the inequality shown when we integrate with respect to $k$ and 
use the Cauchy--Schwarz and Young inequalities;
(ix) Follows from (viii); (x) We observe 
$\|W^{1/2}p(t)\|_{L^2}^2=\|p(t)\|_{L^2}^2+\|V^{1/2}p(t)\|_{L^2}^2=\|p(0)\|_{L^2}^2+(2\pi)^{-2}\|\pa\pf(t)\|_{L^2}^2$,
where we used the equality stated in the proof of (i) just above;
(xi) This follows from (iii), (ix) and (x).
\qed
\end{proof}
\begin{corollary}[Hankel Hilbert--Schmidt operator on the real line]\label{corollary:HankelHilbertSchmidt}
If the inital data $p(0)\in L^2_W(\R;\CC^{m\times m})\cap H^{\mathrm{deg}(d^\prime)}(\R;\CC^{m\times m})$
then for any $t\geqslant0$, we know that $p(t)\in L^2_W((-\infty,0];\CC^{m\times m})$ and $P(t)\in\mathfrak J_2(\Vb)$
with $\Vb=L^2((-\infty,0];\CC^m)$.
\end{corollary}
\begin{proof}
  Using (xi) from Lemma~\ref{lemma:dispersivePDE}, the assumptions
  for $p(0)$ imply $p(t)\in L^2_W(\R;\CC^{m\times m})$ which in turn implies
  $p(t)\in L^2_W\bigl((-\infty,0];\CC\bigr)$. Invoking Lemma~\ref{lemma:HankelHS}
  establishes the result after noting that for $y\in(-\infty,0]$ we know $(1-y)^2\leqslant 3\,(1+y^2)$.
  \qed
\end{proof}
We refer back to our discussion preceding Lemma~\ref{lemma:dispersivePDE}. 
In Corollary~\ref{corollary:HankelHilbertSchmidt}, we see how simple assumptions
on the initial data $p(\,\cdot\,;0)$ on the whole real line $\R$,
establish that Hankel operators $P=P(x,t)$ constructed from the corresponding solutions $p=p(y+z+x;t)$
to the linear dispersion equation, are Hilbert--Schmidt valued, i.e.\/ $P(t)\in\mathfrak J_2(\Vb)$
with $\Vb=L^2\bigl((-\infty,0];\CC^m\bigr)$.
However, the assumptions stated in Lemma~\ref{lemma:dispersivePDE} and Corollary~\ref{corollary:HankelHilbertSchmidt},
preclude data $p(0)$ that might grow as $y\to+\infty$, and thus 
preclude classes of scattering data such as those generating soliton or multi-soliton solutions.

The following prescription of linear equations representing a Fredholm Grassmannian flow in the
top cell $\Lambda_0(\Hb,\Vb)$ covers both the noncommutative Korteweg--de Vries and 
nonlinear Schr\"odinger hierarchies.
\begin{definition}[Prescription: linearised integrable system]\label{prescription:KdVNLS}
Assume the Hilbert--Schmidt operator $P=P(x,t)$ is a Hankel operator in the sense given in Definition~\ref{def:Hankel}, 
with kernel $p=p(y+z+x;t)$. Suppose $Q=Q(x,t)$ is a Fredholm operator of the form $Q=\id-\hat Q$
with $\hat Q=\hat Q(x,t)$ a Hilbert--Schmidt operator and $G=G(x,t)$ is a Hilbert--Schmidt operator.
Assume that for some integer $n\geqslant2$, the operators $P$, $\hat Q$ and $G$ satisfy
the system of linear equations for some constant $\mu_n\in\R$: 
\begin{align*}
\pa_t P&=-\mu_n(\mathrm{i}\mathcal I)^{n-1}\pa_x^nP,\\
\hat Q&=P^2,\\
P&=G(\id-\hat Q),
\end{align*}
in the noncommutative nonlinear Schr\"odinger and modified Korteweg--de Vries cases.
In the Korteweg--de Vries case we set $\hat Q=P$ instead.
The first equation is equivalent to the condition that
the Hankel kernel $p$ satisfies $\pa_tp=\mu_n(\mathrm{i}\mathcal I)^{n-1}\pa_x^np$.
We assume $p$ is sufficiently regular for this equation to make sense.
\end{definition}
\begin{remark}[Motivation]
  That we set $Q=\id-P$ or $Q=\id-P^2$ respectively in the Korteweg--de Vries or nonlinear Schr\"odinger cases,
  is motivated by the corresponding assumptions in Zakharov and Shabat~\cite{ZS} and Ablowitz \emph{et al.}\/ \cite{ARSII}.
\end{remark}
We establish the sense in which a solution to the prescription above exists, for both cases.
The following result collates the two cases outlined in Lemmas~\ref{lemma:HankelHS}, \ref{lemma:dispersivePDE}
and Corollary~\ref{corollary:HankelHilbertSchmidt}; it also utilises the results of Lemma~\ref{lemma:EUCanonical}.
Recall, for some $\nu>0$, the functions $w_\nu\colon(-\infty,0]\to(0,\infty)$ denote weight functions
of the form $w_\nu\colon x\mapsto(1-x)^\nu$. The linear Fredholm equation $P=G(\id-\hat Q)$ given in the
Definition~\ref{prescription:KdVNLS} just above, assuming the kernels concerned exist, which
we establish in the present result, takes the form:
\begin{equation*}
p(y+z+x;t)=g(y,z;x,t)-\int_{-\infty}^0g(y,\xi;x,t)\hat q(\xi,z;x,t)\,\rd\xi.
\end{equation*}
For the Korteweg--de Vries case $\hat q(\xi,z;x,t)\coloneqq p(\xi+z+x;t)$, whereas for the
nonlinear Schr\"odinger and modified Kortewed--de Vries cases it is given by
\begin{equation*}
\hat q(y,z;x,t)\coloneqq\int_{-\infty}^0p(y+\xi+x;t)p(\xi+z+x;t)\,\rd\xi.
\end{equation*}
\begin{lemma}[Existence and Uniqueness: Linear integrable system]\label{lemma:KdVprescription}
  Assume the smooth initial data $p_0=p_0(\,\cdot\,)$ for $p=p(\,\cdot\,;t)$ is such that $\mathrm{det}_2(\id-\hat Q_0)\neq 0$,
  where $\hat Q_0$ is the operator defined in terms of the Hankel operator $P_0$ generated by $p_0$, as described
  just above for the two cases. We have the following results:

(I) Assume there exists a $T>0$ such that there is a solution
\begin{equation*}
p\in C^\infty\Bigl([0,T];L^2_{w_2}\bigl((-\infty,0];\CC^{m\times m}\bigr)\cap C^\infty\bigl((-\infty,0];\CC^{m\times m}\bigr)\Bigr),
\end{equation*}
to the linear equation for $p$ in Definition~\ref{prescription:KdVNLS}.
Then there exists a $T'>0$ with $T'\leqslant T$ such that for $t\in[0,T']$ we know:
(i) The Hankel operator $P=P(x,t)$ with parameter $x\in\R$ generated by $p$ is Hilbert--Schmidt valued, i.e.\/ $P(x,t)\in\mathfrak J_2(\Vb)$
with $\Vb\coloneqq L^2\bigl((-\infty,0];\CC^{m}\bigr)$;
(ii) The determinant $\mathrm{det}_2\bigl(\id-\hat Q(x,t)\bigr)\neq0$ and hence
(iii) There is a unique Hilbert--Schmidt valued solution $G=G(x,t)$ with $G(x,t)\in C^\infty\bigl([0,T'];\Jf_2(\Vb)\bigr)$
to the linear Fredholm equation $P=G(\id-\hat Q)$. 

(II) Assume $p_0\in L^2_W(\R;\CC^{m\times m})\cap H^{\infty}(\R;\CC^{m\times m})$, where $W\colon x\mapsto1+x^2$.
Then for any $T>0$ there is a solution
\begin{equation*}
p\in C^\infty\bigl([0,T];L^2_W(\R;\CC^{m\times m})\cap H^\infty(\R;\CC^{m\times m})\bigr),
\end{equation*}
to the linear equation for $p$ in Definition~\ref{prescription:KdVNLS}.
Further, for $t\in[0,T]$ we know that:
(i) The Hankel operator $P=P(x,t)$ with parameter $x\in\R$ generated by $p$ is Hilbert--Schmidt valued,
i.e.\/ $P(x,t)\in\mathfrak J_2(\Vb)$ with $\Vb\coloneqq L^2\bigl((-\infty,0];\CC^{m}\bigr)$;
(ii) The determinant $\det_2\bigl(\id-\hat Q(x,t)\bigr)\neq0$ and hence
(iii) There is a unique Hilbert--Schmidt valued solution $G=G(x,t)$ with $G(x,t)\in C^\infty\bigl([0,T'];\Jf_2(\Vb)\bigr)$
to the linear Fredholm equation given by $P=G(\id-\hat Q)$.
\end{lemma}
%
%
\begin{remark}
In result (I) we \emph{assume} there exists a solution $p$ to the 
linear equation for $p$ in Definition~\ref{prescription:KdVNLS} for $t\in[0,T]$.
However in result (II), from the assumptions on the initial data, we utilise
Lemma~\ref{lemma:dispersivePDE} to \emph{establish} a solution $p$ exists with the
appropriate properties. Note that in this result, the time regularity of $p$ follows directly from
the spatial regularity assumed on the initial data. In each of the three results, the
statements (ii) and (iii) follow from the existence and uniqueness Lemma~\ref{lemma:EUCanonical}.
As already mentioned, result (II) does not apply to scattering data $p$ relevant to
soliton solutions due to the integrability assumptions we require for $p_0$ on $\R$.
For the nonlinear Schr\"odinger case with $\hat Q=P^2$, for results (I) and (II),
the operators $\hat Q=\hat Q(x,t)$ will be trace class, and the usual Fredholm determinant
$\mathrm{det}_1\bigl(\id-\hat Q(x,t)\bigr)$ can be used in place of the regularised determinant shown.
\end{remark}
\begin{remark}[Smoothness]
The smoothness of $G=G(x,t)$ is determined by the smoothness of $P=P(x,t)$.
\end{remark}
\begin{remark}
Classical regularity results for the linear Korteweg--de Vries equation, for example, can be found 
in Craig and Goodman~\cite{CG}. 
\end{remark}
We now provide some solution formulae for solutions to the linearised integrable system in Definition~\ref{prescription:KdVNLS}
for the Korteweg--de Vries case when $\hat Q=P$, keeping in mind the results of Lemma~\ref{lemma:KdVprescription}.
Recall the quasi-trace from Definition~\ref{def:quasi-trace}.
\begin{lemma}[Korteweg--de Vries solution formulae]
Suppose $P=P(x,t)$ and $G=G(x,t)$ are solutions to the linearised integrable system
in Definition~\ref{prescription:KdVNLS} with $\hat Q=P$, corresponding to either
of the results (I) or (II) in Lemma~\ref{lemma:KdVprescription}.
Then the square matrix-valued kernel $g$ corresponding to to $G$ satisfies
\begin{equation*}
  g(0,0;x,t)=\mathrm{qtr}\,\bigl((\pa_x P)U+U(\pa_x P)\bigr).
\end{equation*}
If all kernels are scalar, then $g(0,0;x,t)=2\,\mathrm{tr}\,\bigl((\pa_x P)U\bigr)$,
and if we restrict $P=P(x,t)$ to multi-soliton scattering data so that it is of finite-rank and
thus trace class (see our discussion preceding Lemma~\ref{lemma:dispersivePDE} as well as Remark~\ref{rmk:generalsolutions}
and Appendix~\ref{sec:scatteringproblem}), then we have the classical result $g(0,0;x,t)=-2\pa_x\log\mathrm{det}_1(\id-P)$.
\end{lemma}
\begin{proof}
Using the quasi-trace identity Lemma~\ref{lemma:quasi-traceidentity}
with $H=H^\prime=P$ and $F=U$ we have,
\begin{align*}
  &&[PUP]_{0,0}&=\mathrm{qtr}\,\bigl((\pa_xP)UP+PU(\pa_xP)\bigr)\\
  \Leftrightarrow &&
[PU]_{0,0}-[P]_{0,0}&=\mathrm{qtr}\,\bigl((\pa_xP)U+U(\pa_xP)-2(\pa_xP)\bigr), 
\end{align*}
where we used the identity $U=\id+PU=\id+UP$. We observe that $2\,\mathrm{qtr}\,(\pa_xP)$ equals
\begin{equation*}
2\int_{-\infty}^0\pa_xp(y+y+x;t)\,\rd y
=\int_{-\infty}^0\pa_xp(\eta+x;t)\,\rd\eta
=\int_{-\infty}^0\pa_\eta p(\eta+x;t)\,\rd\eta,
\end{equation*}
which in turn equals $p(x;t)$ which equals $[P]_{0,0}$, giving the first result since $G=PU$.
The second (trace) result follows, as in the scalar case the quasi-trace coincides with
the usual trace, and we can invariantly rotate factors in the argument of the trace operator.
The third (determinant) result follows, since by standard analysis, we have
\begin{equation*}
  \mathrm{tr}\,\bigl((\pa_xP)U\bigr)=-\mathrm{tr}\,\Bigl(\bigl(\pa_x(\id-P)\bigr)(\id-P)^{-1}\Bigr)
  =-\frac{\pa_x\mathrm{det}_1(\id-P)}{\mathrm{det}_1(\id-P)},
\end{equation*}
giving the result.\qed
\end{proof}

\section{Potential Korteweg de Vries equation}\label{sec:KdV}
We carry through P\"oppe's programme for the noncommutative potential Korteweg--de Vries equation.
Consider the linear integrable system in Definition~\ref{prescription:KdVNLS}
for the case of the potential Korteweg--de Vries equation when $\hat Q=P$,
with $n=3$ and $\mu\coloneqq\mu_3\in\R$. The following result was 
originally proved by P\"oppe~\cite{P-KdV} in the scalar case.
\begin{theorem}[Potential Korteweg--de Vries decomposition]\label{thm:KdV}
Assume the kernel $p=p(\,\cdot\,;t)$ generating the Hankel operator $P=P(t)$
satisfies either of the sets of assumptions in (I) or (II),
in the existence and uniqueness Lemma~\ref{lemma:KdVprescription}.
Then assume the Hankel Hilbert--Schmidt operator $P=P(x,t)$ generated by $p=p(\,\cdot\,+x;t)$
for any $x\in\R$, and the (necessarily) Hilbert--Schmidt operator $G=G(x,t)$ 
satisfy the linear system of equations given in Definition~\ref{prescription:KdVNLS}
for the case of the potential Korteweg--de Vries equation, for $t\in[0,T]$ for some $T>0$.
Then the integral kernel $g=g(y,z;x,t)$ corresponding to $G=G(x,t)$ 
satisfies the non-local noncommutative potential Korteweg--de Vries equation for every $t\in[0,T]$ given by,
\begin{equation*}
\pa_tg(y,z;x,t)=\mu\pa_x^3g(y,z;x,t)-3\mu\pa_x g(y,0;x,t)\pa_x g(0,z;x,t).
\end{equation*}
We observe $g=g(0,0;x,t)$ satisfies the potential Korteweg--de Vries equation,
\begin{equation*}
\pa_tg=\mu\pa_x^3g-3\mu(\pa_x g)^2.
\end{equation*}
\end{theorem}
To establish Theorem~\ref{thm:KdV}, we need to show the integral
kernel $g(y,z;x,t)=[G](y,z;x,t)$ satisfies the nonlocal equation shown.
We achieve this via the P\"oppe algebra which we introduce here.
The original proof of this result can be found in Doikou \textit{et al.\/} \cite{DMS}.
However the P\"oppe algebra approach provides the natural structure for investigating
the integrability of higher order equations in the non-commutative potential Korteweg--de Vries
hierarchy, as was demonstrated in Malham~\cite{Malham:KdVhierarchy}, providing the rationale for
our pursuit of this approach here. We adapt this combinatorial algebraic structure for Hilbert--Schmidt operators
in Section~\ref{sec:NLS}, to establish integrability for the noncommutative nonlinear Schr\"odinger
and modified Korteweg--de Vries equations.
These combinatorial algebraic structures break down the problem of establishing integrability to the problem of determining
the existence of suitable polynomial expansions, which in turn translates to the problem of
solving an overdetermined linear algebraic equation for the polynomial coefficients.
The rest of this section is devoted to developing the P\"oppe algebra and then proving
Theorem~\ref{thm:KdV}, together with, at the end, some obersvations on specific solutions
and solution forms.

Let us begin with some calculus results we require just below as well as in Section~\ref{sec:NLS}.
First we need to define the following character coefficients.
Let $\mathbb N^\ast$ denote the free monoid of words on $\mathbb N$, i.e.\/ the set
of all possible words of the form $a_1a_2\cdots a_k$ we can construct from letters $a_1, a_2, \ldots, a_k\in\mathbb N$.
\begin{definition}[Signature character]\label{def:signaturecharacter}
  Suppose $a_1a_2\cdots a_n\in\mathbb N^\ast$. The \emph{signature character} $\chi\colon\mathbb N^\ast\to\R$
  associated with any such word is given by
  \begin{equation*}
    \chi\bigl(a_1a_2\cdots a_n\bigr)\coloneqq\begin{pmatrix} a_1+\cdots+a_n\\a_n\end{pmatrix}
    \begin{pmatrix} a_{1}+a_{2}+\cdots+a_n\\a_{2}\end{pmatrix}
    \cdots\begin{pmatrix} a_{n-1}+a_n\\a_{n-1}\end{pmatrix}\begin{pmatrix} a_n\\a_n\end{pmatrix},
  \end{equation*}
  where each of the factors shown on the right is a Leibniz coefficient. For example,
  the penultimate factor is $a_{n-1}+a_n$ choose $a_{n-1}$.  
\end{definition}
The reason for calling this coefficient the `signature character' will become apparent presently.
Let $\mathcal C(n)$ denote the set of all compositions of $n\in\mathbb N$.
The following result is equivalent to that in Malham~\cite[Lemma~2]{Malham:KdVhierarchy}.
\begin{lemma}[Inverse operator expansions]\label{lemma:IOE}
  Suppose the operator $\hat Q$ depends on a parameter with respect to which we
  wish to compute derivatives---assume these derivatives exist to any order we require. 
  Further suppose $U\coloneqq(\id-\hat Q)^{-1}$ exists. Then we observe:
\begin{equation*}
 U\equiv\id+U\hat Q\equiv\id+\hat QU.
\end{equation*}
For any non-negative integer $k$, set $\hat Q_k\coloneqq\pa^k Q$ and $U_k\coloneqq\pa^kU$.
Then we observe that $U_1=U\hat Q_1U$ and more generally
\begin{equation*}
U_n=\sum\chi(a_1a_2\cdots a_k)\,U\hat Q_{a_1}U\hat Q_{a_2}U\cdots U\hat Q_{a_k}U,
\end{equation*}
where the sum is over all compositions $a_1a_2\cdots a_k\in\mathcal C(n)$.
The real-valued coefficients $\chi(a_1a_2\cdots a_k)$ are the signature characters defined above. 
\end{lemma}
\begin{proof}
  The first result in the lemma is immmediate, as is the result for $U_1$, which also
  establishes that the more general case for $U_n$ holds for $n=1$. Assume that the
  general result holds for $U_k$ for all $2\leqslant k\leqslant n$.
  Suppose we now compute the $(n+1)$th derivative of the identity $U=\id+PU$. 
  After using that term $PU_{n+1}$ in the Leibniz expansion can be combined 
  with the original term $U_{n+1}$ on the left, we arrive at the expression:
  \begin{equation*}
    U_{n+1}=\begin{pmatrix} n+1\\n+1\end{pmatrix} UP_{n+1}U+\begin{pmatrix} n+1\\n\end{pmatrix} UP_{n}U_1
    +\cdots+\begin{pmatrix} n+1\\1\end{pmatrix} UP_{1}U_n.    
  \end{equation*}
  The first term on the right is the only composition of $n+1$ starting with the digit $n+1$
  and matches the single digit composition of $n+1$ with the correct coefficient.
  Now consider the second term on the right. Substituting for $U_1$ this term becomes $UP_{n}UP_1U$,
  which is the only composition of $n+1$ starting with the digit $n$. The coefficient is
  $n+1$ choose $n$ which matches $\chi(n1)$. Consider the $k$th term on the right
  of the form $UP_{n+1-k}U_k$ with the corresponding Leibniz coefficient corresponding to
  $n+1$ choose $n+1-k$, with $k\leqslant n$. If we substitute for $U_k$ as the linear
  combination of all the compositions $w$ of $k$ with coefficients $\chi(w)$, then
  we exhaust all the possible compositions of $n+1$ that start with the digit $n+1-k$.
  The corresponding coefficient of the term $(n+1-k)w$ is
  \begin{equation*}
    \begin{pmatrix} n+1\\n+1-\ell\end{pmatrix} \chi(w)=\chi\bigl((n+1-\ell)w\bigr).
  \end{equation*}
  This follows from the definition of $\chi(w)$. This completes the proof. \qed
\end{proof}

Let us now focus on the \emph{P\"oppe algebra} associated with the noncommutative potential Korteweg--de Vries equation.
The results here for this particular P\"oppe algebra can largely be found in Malham~\cite{Malham:KdVhierarchy}.
In the prescription for the noncommutative potential Korteweg--de Vries equation, see Definition~\ref{prescription:KdVNLS},
we set $\hat Q=P$. Hence in the statements just above, we replace $\hat Q$ by $P$.
Further, in Definition~\ref{prescription:KdVNLS}, the linear Fredholm equation has the form $P=G(\id-\hat Q)$
or equivalently $G=PU$. In particular since $U=\id+UP=\id+PU$ we observe, using the notation $G_n\coloneqq\pa^n G$
for any $n\in\mathbb N$, that $G_n\equiv U_n$. In particular, after simply specialising the final result
of Lemma~\ref{lemma:IOE} to $\hat Q=P$ and then applying the kernel bracket operator, we have
\begin{equation*}
[U_n]=\sum\chi(a_1a_2\cdots a_k)\,[UP_{a_1}UP_{a_2}U\cdots UP_{a_k}U],
\end{equation*}
with the sum is over all compositions $a_1a_2\cdots a_k\in\mathcal C(n)$ as above.
This means that any such terms $[U_n]$ for $n\in\mathbb N$ can be expanded in terms
of a basis of monomials of the form $[P_{a_1}UP_{a_2}U\cdots UP_{a_k}U]$.
In other words, we can construct a vector space of such monomials. Further still,
the vector space is closed under the P\"oppe product in the form just below, and thus
we can construct an algebra of such monomials. The following result is a direct
conequence of the P\"oppe product rule and can be found in Malham~\cite[Lemma~3]{Malham:KdVhierarchy}.
Let $\Cb\coloneqq\cup_{n\geqslant0}\Cb(n)$ denote the set of all compositions.
\begin{lemma}[P\"oppe product for kernel monomials]\label{lemma:Poppeprodkernels}
Suppose $u=a_1\cdots a_k\in\Cb$ and $v=b_1\cdots b_\ell\in\Cb$ 
while $a,b\in\mathbb N$. Let $\mm_u$ and $\mm_v$ denote the kernel monomials
$\mm_u\coloneqq UP_{a_1}UP_{a_2}U\cdots UP_{a_k}U$ and
$\mm_v\coloneqq UP_{b_1}UP_{b_2}U\cdots UP_{b_\ell}U$.
Then we have 
\begin{equation*}
[\mm_uP_aU][UP_b\mm_v]
=[\mm_uP_{a+1}UP_b\mm_v]+[\mm_uP_aUP_{b+1}\mm_v]+2\,[\mm_uP_aUP_1UP_b\mm_v].
\end{equation*}
\end{lemma}
\begin{proof}
Using that $U\equiv\id+UP\equiv\id+PU$,
the P\"oppe product rule from Lemma~\ref{lemma:productrule}
and that $\pa U=U(\pa P)U$, we see that 
\begin{align*}
[\mm_uP_aU][UP_b\mm_v]&=[\mm_uP_a+\mm_uP_aUP][PUP_b\mm_v+P_b\mm_v]\\
&=\;[\mm_uP_a][P_b\mm_v]+[\mm_uP_a][PUP_b\mm_v]\\
&\;+[\mm_uP_aUP][P_b\mm_v]+[\mm_uP_aUP][PUP_b\mm_v]\\
=&[\mm_u\pa(P_aP_b)\mm_v]+[\mm_u\pa(P_aP)UP_b\mm_v]\\
&\;+[\mm_uP_aU\pa(PP_b)\mm_v]+[\mm_uP_aU\pa(PP)UP_b\mm_v]\\
=&[\mm_u(\pa P_a)P_b\mm_v]+[\mm_uP_a(\pa P_b)\mm_v]\\
&\;+[\mm_u(\pa P_a)PUP_b\mm_v]+[\mm_uP_a(\pa P)UP_b\mm_v]\\
&\;+[\mm_uP_aU(\pa P)P_b\mm_v]+[\mm_uP_aUP(\pa P_b)\mm_v]\\
&\;+[\mm_uP_aU(\pa P)PUP_b\mm_v]+[\mm_uP_aUP(\pa P)UP_b\mm_v]\\
=&[\mm_u(\pa P_a)UP_b\mm_v]+[\mm_uP_aU(\pa P_b)\mm_v]\\
&\;+2\,[\mm_uP_aU(\pa P)UP_b\mm_v],
\end{align*}
establishing the result.\qed
\end{proof}
\begin{remark}
Recall from Lemma~\ref{lemma:productrule}, we implicitly interpret 
kernel products of the form $[\,\cdot\,][\,\cdot\,]\cdots[\,\cdot\,][\,\cdot\,]$
as $[\,\cdot\,](y,0;x)[\,\cdot\,](0,0;x)\cdots[\,\cdot\,](0,0;x)[\,\cdot\,](0,z;x)$.
\end{remark}
At this juncture some clarity can be gained via abstraction. 
We replace the monomials $[UP_{a_1}UP_{a_2}U\cdots UP_{a_k}U]$ simply by the words $a_1a_2\cdots a_k\in\mathcal C$
and define the following product. Let $\Rb\la\Cb\ra_\ast$ denote the non-commutative 
polynomial algebra over $\Rb$ generated by composition elements from $\Cb$, 
endowed with the following \emph{P\"oppe product} for compositions. We call the resulting algebra
$\Rb\la\Cb\ra_\ast$, the \emph{P\"oppe algebra}.
\begin{definition}[P\"oppe product for compositions]\label{def:Poppeproductcompositions}
  Consider two compositions $ua$ and $bv$ in $\Cb$, where we distinguish the last letter of $ua$ as $a$
  and the first letter of $bv$ as $b$. We define the \emph{P\"oppe product}
  $\ast\colon\Rb\la\Cb\ra_\ast\times\Rb\la\Cb\ra_\ast\to\Rb\la\Cb\ra_\ast$
  for the compositions $ua$ and $bv$ to be, 
  \begin{equation*}
    (ua)\ast(bv)=u(a+1)bv+ua(b+1)v+2\cdot(ua1bv).
  \end{equation*}
  The empty composition $\nu\in\Rb\la\Cb\ra_\ast$ satisfies $\nu\ast w=w\ast\nu=w$ for any $w\in\Rb\la\Cb\ra_\ast$.
\end{definition}
We observe, the real algebra of kernel monomials of the form $[UP_{a_1}UP_{a_2}U\cdots UP_{a_k}U]$
with the product given in Lemma~\ref{lemma:Poppeprodkernels}, is isomorphic to $\Rb\la\Cb\ra_\ast$.
The abstract version of the expansions above for $[U_n]$ are the following signature expansions
in $\Rb\la\Cb\ra_\ast$.
\begin{definition}[Signature expansions]\label{def:polycomp}
For any $n\in\Nb$ we define the following linear \emph{signature expansions} $\bs{n}\in\Rb\la\Cb\ra_\ast$:
\begin{equation*}
\boldsymbol{n}\coloneqq\sum_{w\in\Cb(n)}\chi(w)\cdot w.
\end{equation*}
\end{definition}
Now recall from Definition~\ref{prescription:KdVNLS} in the case when $n\geqslant3$ is odd,
we assume $P=P(x,t)$ satisfies $\pa_t P=\mu\pa_x^nP$ for some constant $\mu\in\R$.
Suppose $\hat Q=\hat Q(x,t)$ and $G=G(x,t)$ satisfy the other two linear equations in the prescription,
namely $\hat Q=P$ and $P=G(\id-\hat Q)$. Suppose $U\coloneqq(\id-P)^{-1}$ and thus $G=PU$.
Further, for convenience, we set $P_n\coloneqq\pa^nP$ and $U_n\coloneqq\pa^nU$ for non-negative integers $n$.
Recall the first two identities in Lemma~\ref{lemma:usefulNLSidentities} which here
translate to $U\equiv\id+UP\equiv\id+PU$. Hence for all non-negative integers $n$ we
observe $G_n\equiv U_n$. Hence, using that $\pa_t P=\mu\pa^nP$, and the
identities from Lemma~\ref{lemma:usefulNLSidentities}, we observe that,
\begin{equation*}
  \pa_tG=U(\pa_tP)U=\mu\,UP_nU. 
\end{equation*}
Consider the right-hand side, modulo the $\mu$ factor, expressed in terms of the
abstract P\"oppe algebra, it is simply the basic letter, $n\in\Rb\la\Cb\ra_\ast$.
A natural question now is to ask ourselves if we can write $n\in\Rb\la\Cb\ra_\ast$
in terms of a polynomial expansion $\pi_n=\pi_n(\bs 1,\bs 2,\ldots,\bs{(n-2)},\bs{n})$, a P\"oppe polynomial,
of the form,  
\begin{equation*}
\pi_n\coloneqq \sum_{k=1}^{\frac12(n+1)}\sum_{a_1a_2\cdots a_{k}\in\mathcal C^\ast(n)} c_{a_1a_2\cdots a_{k}}\cdot\bs a_1\ast\bs a_2\ast\cdots\ast\bs a_{k},
\end{equation*}
where $\mathcal C^\ast(n)\subset\mathcal C(n)$ represents the subset of compositions $w=a_1a_2\cdots a_k$ of $n$
such that $a_1+a_2+\cdots+a_k=n-k+1$ and the coefficients $c_{a_1a_2\cdots a_{k}}$ are real.

We now establish integrability of the noncommutative potential Korteweg--de Vries equation 
in the context of the P\"oppe algebra via the following example with $n=3$. 
\begin{example}[Korteweg--de Vries integrability]\label{ex:illustrativeKdVex}
Assume $n=3$. Using Definition~\ref{def:signaturecharacter} for $\chi$, the first three signature
expansions are $\bs{1}=\chi(1)\cdot 1$, $\bs{2}=\chi(2)\cdot 2+\chi(11)\cdot 11$ and
\begin{align*}
&&\bs 3&=\chi(3)\cdot 3+\chi(21)\cdot 21+\chi(12)\cdot12+\chi(111)\cdot111\\
\Leftrightarrow&&
3&=\bs 3-3\cdot(21+12)-6\cdot111.
\end{align*}
Now consider the P\"oppe product, which by bilinearity can be expressed in the form
$\bs 1\ast\bs 1=\bigl(\chi(1)\cdot 1\bigr)\ast\bigl(\chi(1)\cdot 1\bigr)=\chi(1)\chi(1)\cdot(1\ast1)$.
Explicitly computing $1\ast 1$ from Definition~\ref{def:Poppeproductcompositions}
and using that $\chi(1)=1$, we observe
\begin{equation*}
\bs 1\ast\bs 1=\chi(1)\chi(1)\cdot\bigl(21+12+2\cdot(111)\bigr)=21+12+2\cdot 111.
\end{equation*}
Comparing this result with the expansion for the basic letter $3\in\Rb\la\Cb\ra_\ast$ above,
we see, 
\begin{equation*}
3=\bs 3-3\cdot\bigl(\bs 1\ast\bs 1\bigr).
\end{equation*}
Thus $3\in\Rb\la\Cb\ra_\ast$ has a polynomial expansion in terms of signature expansions as shown.
In terms of the kernel monomials this expresses the fact that $[UP_3U]=[U_3]-3[U_1][U_1]$.
This statement establishes that the nonocommutative potential Korteweg--de Vries equation is integrable.
\end{example}
\begin{remark}[Noncommutative potential Korteweg de Vries hierarchy]
  The result of Example~\ref{ex:illustrativeKdVex} can be reformulated into a systematic
  procedure. For any odd order $n\geqslant3$, the goal is to determine real coefficients
  in the P\"oppe polynomial $c_{a_1\cdots a_k}$ such that $n=\pi_n$. By expanding $\pi_n$,
  and equating coefficients of all the independent words generated, 
  this leads to an overdetermined linear algebraic system of equations for the coefficients.
  It is possible to then show that at each order $n$, there is a unique solution.
  See Malham~\cite{Malham:KdVhierarchy} for more details.
\end{remark}
The result of the Existence and Uniqueness Lemma~\ref{lemma:KdVprescription} implies we can generate
solutions to the noncommutative potential Korteweg--de Vries equation as follows.
Suppose we are given the arbitrary smooth data function $p_0=p_0(x)$ and are 
able to generate a corresponding solution $p=p(x;t)$ to the linear
equation $\pa_tp=\mu\pa_x^3p$ analytically, either in Fourier space or 
as a convolution integral with the Airy function kernel.
In other words we analytically evaluate $p=p(x;t)$ at the prescribed time $t>0$. 
Then \emph{in practice}, for any given time $t\in[0,T]$, 
we can generate a corresponding solution $g=g(0,0;x,t)$ to the noncommutative potential Korteweg--de Vries equation
by solving the linear Fredholm equation in Lemma~\ref{lemma:KdVprescription} for $g(y,z;x,t)$ as follows.
Since our goal is to generate solutions $g(0,0;x,t)$ only, rather than more generally
to generate solutions $g=g(y,z;x,t)$ to the family of nonlocal equations, we can set $y=0$
from the outset, so the linear Fredholm equation becomes
\begin{equation*}
p(z+x;t)=g(0,z;x,t)-\int_{-\infty}^0g(0,\xi;x,t)p(\xi+z+x;t)\,\rd\xi.
\end{equation*}
We solve this Fredholm equation for $g=g(0,z;x,t)$ and then set $z=0$,
to generate a solution $g=g(0,0;x,t)$ corresponding to the scattering data $p=p(x;t)$.
The following example demonstrates this procedure for a solitary wave solution.
\begin{example}[Solitary wave]\label{ex:solitarywave}
Suppose $A$ is a constant square matrix, whose eigenvalues are distinct and all are strictly positive. 
Consider the scattering data,
\begin{equation*}
p(x;t)=-2A\exp(Ax+A^3t).
\end{equation*}
In the standard fashion, see for example Drazin and Johnson~\cite[p.~73--4]{DJ},
we look for a solution $g=g(0,z;x,t)$ to the linear Fredholm equation above of the separated form
\begin{equation*}
g(0,z;x,t)=L(x,t)\exp(Az),
\end{equation*}
for some unknown square-matrix valued function $L=L(x,t)$ to be determined.
Substituting the forms above for $p=p(x;t)$ and $g=g(0,z;x,t)$ into the
linear Fredholm equation and postmultiplying through by $\exp(-Az)$, reveals that $L=L(x,t)$ satisfies,
\begin{equation*}
  -2A\exp(Ax+A^3t)=L(x,t)+2L(x,t)A\int_{-\infty}^0\exp(2A\xi)\,\rd\xi\,\exp(Ax+A^3t).
\end{equation*}
Evaluating the integral on the right-hand side, solving the resulting linear algebraic problem for $L=L(x,t)$
and noting that $g(0,0;x,t)=L(x,t)$, we find that
\begin{equation*}
g(0,0;x,t)=-2A\exp(Ax+A^3t)\bigl(I+\exp(Ax+A^3t)\bigr)^{-1}.
\end{equation*}
Recall $g=g(0,0;x,t)$ is the solution to the potential Korteweg--de Vries equation 
and the solution to the actual Korteweg--de Vries equation is given by
\begin{equation*}
\pa_xg(0,0;x,t)=-\tfrac12 A^2\mathrm{sech}^2\bigl(\tfrac12 A(x+A^2t)\bigr). 
\end{equation*}
In the special case $A>0$ is scalar this corresponds to the one-soliton solution;
see Remarks~\ref{rmk:GLM} and \ref{rmk:KdVsolutions} just below.
\end{example}
\begin{remark}[Gel'fand--Levitin--Marchenko equation]\label{rmk:GLM}
The celebrated \emph{Gel'fand--Levitin--Marchenko} equation usually takes the form, 
\begin{equation*}
  F(Z+X)+K(Z,X)+\int_X^{\infty}K(Y,X)F(Y+Z)\,\rd Y=0,
\end{equation*}
for the function $K=K(Z,X)$ given the scattering data $F=F(Z)$.
A simple transformation of variables reveals the relation between the
Gel'fand--Levitin--Marchenko equation and the linear Fredholm equation for the
kernel $\hat g(z;x)\coloneqq g(0,z;x)$ given by,
\begin{equation*}
p(z+x)=\hat g(z;x)-\int_{-\infty}^0\hat g(\eta;x)p(\eta+z+x)\,\rd\eta,
\end{equation*}
where $p=p(\,\cdot\,)$ corresponds to the scattering data. We suppress the
implicit time-dependence. Making the change of variables $z\coloneqq -Z-\frac12x$, $\eta\coloneqq-Y-\frac12x$
and then $x\coloneqq-2X$ in the linear Fredholm equation, we arrive at the relation,
\begin{equation*}
p(-Z-X)=\hat g(-Z+X;-2X)-\int_X^{\infty}\hat g(-Y+X;-2X)p(-Y-Z)\,\rd Y.
\end{equation*}
We now observe that if we identify $K(Z,X)\coloneqq\hat g(z;x)\equiv\hat g(-Z+X;-2X)$
and $F(Z+X)\coloneqq -p(x+z)\equiv-p(-Z-X)$,
then $K(Z,X)$ satisfies the Gel'fand--Levitin--Marchenko equation above.
Note that setting $z=0$ corresponds to setting $Z=X$.
\end{remark}

\begin{remark}[Relation to the classical Korteweg--de Vries solution]\label{rmk:KdVsolutions}
  Classically, the solution $U=U(X,t)$ to the scalar Korteweg--de Vries equation is given
  in terms of the solution $K=K(Z,X;t)$ to the Gel'fand--Levitin--Marchenko equation by
  \begin{equation*}
  U(X,t)=-2\pa_XK(X,X;t).
  \end{equation*}
  Since from Remark~\ref{rmk:GLM} we have $X=-\tfrac12x$ so $\pa_X=-2\pa_x$ and
  $K(Z,X;t)=\hat g(z;x,t)=g(0,z;x,t)$ we observe that (using that $z=0$ corresponds to $Z=X$)
  \begin{equation*}
  U(X,t)=4\pa_xg(0,0;x,t).
  \end{equation*}
  Hence, for example, consider the solitary wave in Example~\ref{ex:solitarywave} 
  when $A>0$ is a scalar constant. This case corresponds to the one-soliton solution.
  At time $t=0$, in the $x$-frame the solution corresponding to the scattering data
  $p(x;0)=-2A\mathrm{e}^{Ax}$ is the `$\mathrm{sech}^2$' profile $\pa_x g(0,0;x,0)=-\frac12A^2\mathrm{sech}^2(\frac12Ax)$.
  However, in the $X$-frame this corresponds to,
  $U(X,0)=4\,\bigl(-\tfrac12A^2\mathrm{sech}^2(\tfrac12Ax)\bigr)=-2A^2\mathrm{sech}^2(AX)$,
  matching the expected solution from the data $F(Z+X)=-p(-Z-X)=2A\mathrm{e}^{-A(Z+X)}$; see Keener~\cite[p.~415]{Keener}.
\end{remark}
\begin{remark}[General solutions]\label{rmk:generalsolutions}
P\"oppe~\cite[Sections 3.3 \& 3.4]{P-KdV} outlines in detail further solutions to
the scalar Korteweg--de Vries equation that can be generated by his approach, i.e.\/ the
approach above. For example, therein he outlines in detail, how,   
by linear combinations of terms of the form $\mathrm{e}^{ax+a^3t}$ in $p=p(x;t)$
for $N$ different discrete values of $a$, we can generate the $N$-soliton
solution to the Korteweg--de Vries equation; with each solitary wave phase shifted 
according to the size of its coefficient factor. Further solutions can
be generated from scattering data $p=p(x,t)$ consisting of a continuous
linear combination of such exponential forms for which the corresponding operator is positive definite.
Further, breather solutions and degenerate solitons can be generated within the context of the P\"oppe approach.
However rational solutions cannot presently be generated in this way, as they do not decay as $x\to-\infty$.
See P\"oppe~\cite[Section~3.4]{P-KdV} for more details.
\end{remark}
Thusfar , our discussion has revolved around generating
solutions to the Korteweg--de Vries equation from given scattering data $p=p(x,t)$, which
satisfy the noncommutative linear Korteweg--de Vries equation. 
Such scattering data can be quite arbitrary, though we restrict ourselves to smooth scattering data.
The corresponding solution $g(0,0;x,t)$ to the noncommutative nonlinear
Korteweg--de Vries equation is generated by solving the linear Fredholm equation $P=G(\id-P)$ for $G=G(x,t)$.
However, if we are given initial data $g(0,0;x,0)$ then we need to generate the initial scattering
data $p_0=p_0(x)$ and thus $p=p(x;t)$. To generate $p_0$ from $g(0,0;x,0)$, we need to solve the 
classical ``scattering problem''. We address this in Appendices~\ref{sec:scatteringproblem} and \ref{sec:Evansfunction}.

\section{Nonlinear Schr\"odinger equations}\label{sec:NLS}
We give new proofs that the noncommutative nonlinear Schr\"odinger and modified
Korteweg--de Vries (mKdV) equations are realisable as Fredholm Grassmannian flows and
are thus linearisable in the sense that their solution can be determined via
the solution of the corresponding linearised dispersive equation
and a linear Fredholm equation. Indeed our approach extends to the nonlocal reverse time and
reverse space-time versions of these equations. To begin we first consider the following
coupled linear system of equations for the Hilbert--Schmidt operators
$P_\alpha$, $P_\beta$, $\hat Q_\alpha$, $\hat Q_{\beta}$, $G_\alpha$ and $G_\beta$:
\begin{equation*}
\begin{aligned}
\pa_tP_\alpha&=d_\alpha(\pa) P_\alpha,\\
\hat Q_\alpha&=P_\beta P_\alpha,\\
P_\alpha&=G_\alpha(\id-\hat Q_\alpha),
\end{aligned}
~\qquad\text{and}~\qquad
\begin{aligned}
\pa_tP_\beta&=d_\beta(\pa) P_\beta,\\
\hat Q_\beta&=P_\alpha P_\beta,\\
P_\beta&=G_\beta(\id-\hat Q_\beta).
\end{aligned}
\end{equation*}
where the partial differential operators $d_\alpha=d_\alpha(\pa)$ and $d_\beta=d_\beta(\pa)$
for some integer order $n\geqslant 2$ are given by $d_\alpha=\mu_n\pa^n$ and $d_\beta=(-1)^{n-1}d_\alpha$,
for some parameter $\mu_n\in\CC$ which ensures the linear partial differential equations above are
dispersive. This amounts to $\mu_n$ being pure imaginary for even $n$ and real for odd $n$.
We also suppose the kernels of $P_\alpha$ and $P_\beta$ are square matrix-valued. 

It is naturally convenient to combine the coupled system into one `enhanced' system, and we can achieve
this as follows. Note the two Fredholm equations are equivalent to the single, block equation,
\begin{equation*}
  \begin{pmatrix} O & P_\beta \\ P_\alpha & O \end{pmatrix}
  =\begin{pmatrix} O & G_\beta \\ G_\alpha & O \end{pmatrix}
  -\begin{pmatrix} O & G_\beta \\ G_\alpha & O \end{pmatrix}\begin{pmatrix} P_\beta P_\alpha & O \\ O & P_\alpha P_\beta\end{pmatrix}.
\end{equation*}
Indeed if we set,
\begin{equation*}
  P\coloneqq\begin{pmatrix} O & P_\beta \\ P_\alpha & O \end{pmatrix}\qquad\text{and}\qquad 
  G\coloneqq\begin{pmatrix} O & G_\beta \\ G_\alpha & O \end{pmatrix},
\end{equation*}
then the block equation above is equivalent to $P=G(\id-P^2)$, and further
the pair of linear dispersive partial differential equations for $P_\alpha$ and $P_\beta$
can be consolidated into the single equation $\pa_t P=-\mu_n(\mathrm{i}\mathcal I)^{n-1}\pa^nP$
with $\mu_n\in\R$, corresponding to the presciption given in Definition~\ref{prescription:KdVNLS}.
\begin{theorem}[Nonlinear Schr\"odinger and mKdV decompositions]\label{thm:NLS}
Assume the kernel $p=p(\,\cdot\,;t)$ generating the Hankel operator $P=P(t)$
satisfies either of the sets of assumptions in (I) or (II),
in the existence and uniqueness Lemma~\ref{lemma:KdVprescription}.
Then assume the Hankel Hilbert--Schmidt operator $P=P(x,t)$ generated by $p=p(\,\cdot\,+x;t)$
for any $x\in\R$, the (necessarily) trace class operator $\hat Q=P^2$ and 
the (necessarily) Hilbert--Schmidt operator $G=G(x,t)$ 
satisfy the linear system of equations given in Definition~\ref{prescription:KdVNLS}, for $t\in[0,T]$ for some $T>0$.
Then on $[0,T]$, the kernel $g=g(y,z;x,t)$ corresponding to $G=G(x,t)$ 
satisfies the non-local noncommutative:

(A) Nonlinear Schr\"odinger equation when $n=2$, given by,
\begin{equation*}
\pa_tg(y,z;x,t)=-\mu_2\mathrm{i}\mathcal I\bigl(\pa_x^2g(y,z;x,t)-2g(y,0;x,t)g(0,0;x,t)g(0,z;x,t)\bigr), 
\end{equation*}

(B) Modified Korteweg--de Vries equation when $n=3$, given by,
\begin{align*}
  \pa_tg(y,z;x,t)=&\;\mu_3\Bigl(\pa_x^3g(y,z;x,t)-3\bigl(\pa_x g(y,0;x,t)g(0,0;x,t)g(0,z;x,t)\\
  &\;\qquad\qquad\qquad\qquad\qquad+g(0,z;x,t)g(0,0;x,t)\pa_x g(0,z;x,t)\bigr)\Bigr).
\end{align*}
In our discussion just below and Examples~\ref{ex:NLS} and \ref{ex:nonlocalNLS}, we explicitly show how
for $g=g(0,0;x,t)$, the equations above lead to both the usual (local) noncommutative nonlinear Schr\"odinger
and modified Korteweg--de Vries equations, as well as versions where the the cubic nonlinearities involve
nonlocal reverse space-time central factors.
\end{theorem}
We give the proof of Theorem~\ref{thm:NLS} below, in the context of the
\emph{P\"oppe Triple System}. Before doing so,
we examine some of the systems implicit in the equation 
for $g=g(y,z;x,t)$ in Theorem~\ref{thm:NLS}. The off-diagonal block form
for $G$ implies that $g=[G]$ has an off-diagonal block form. Suppose the
upper right and lower left blocks of $g$ are $g_\beta=g_\beta(y,z;x,t)$
and $g_\alpha=g_\alpha(y,z;x,t)$, respectively.
Then the equation for $g=g(0,0;x,t)$ in (A) in Theorem~\ref{thm:NLS} implies
$g_\alpha=g_\alpha(0,0;x,t)$ and $g_\beta=g_\beta(0,0;x,t)$ satisfy,
\begin{align*}
&&\mathrm{i}\pa_t\begin{pmatrix} O & g_\beta \\ g_\alpha & O \end{pmatrix}
&=\mu_2\begin{pmatrix} -\id & O \\ O & \id \end{pmatrix}\pa^2\begin{pmatrix} O & g_\beta \\ g_\alpha & O \end{pmatrix}
-2\mu_2\begin{pmatrix} -\id & O \\ O & \id \end{pmatrix}\begin{pmatrix} O & g_\beta \\ g_\alpha & O \end{pmatrix}^3\\  
\Leftrightarrow&&
\mathrm{i}\begin{pmatrix} O & \pa_tg_\beta \\ \pa_tg_\alpha & O \end{pmatrix}
&=\mu_2\begin{pmatrix} O & -\pa^2g_\beta \\ \pa^2g_\alpha & O \end{pmatrix}
-2\mu_2\begin{pmatrix} O & -g_\beta g_\alpha g_\beta \\ g_\alpha g_\beta g_\alpha & O \end{pmatrix}\\
\Leftrightarrow&&
\mathrm{i}\pa_tg_\alpha&=\mu_2\pa^2 g_\alpha-2\mu_2 g_\alpha g_\beta g_\alpha,\\
&&\mathrm{i}\pa_tg_\beta&=-\mu_2\pa^2g_\beta+2\mu_2 g_\beta g_\alpha g_\beta.
\end{align*}
Note, the linear partial differential equations
$\mathrm{i}\pa_tp_\alpha=\mu_2\pa^2p_\alpha$ and $\mathrm{i}\pa_tp_\beta=-\mu_2\pa^2p_\beta$ are both dispersive.
There are several different consistent choices we can make for
$P_\alpha$ and $P_\beta$ that generate different noncommutative nonlinear 
Schr\"odinger systems as indicated in the following examples.
\begin{example}[Non-commutative nonlinear Schr\"odinger equation]\label{ex:NLS} 
One consistent choice is to set $P_\beta=P_\alpha^\dag$, the adjoint operator
to $P_\alpha$ with respect to the $L^2$ inner product. For this choice
$G_\beta=G_\alpha^\dag$, the corresponding adjoint operator to $G_\alpha$. 
Then we observe $g_\beta=g_\beta(0,0;x,t)$ is given by $g_\beta=g_\alpha^\dag$, the
complex conjugate transpose of $g_\alpha$. In that case $g_\alpha=g_\alpha(0,0;x,t)$
satisfies the non-commutative nonlinear Schr\"odinger equation:
\begin{equation*}
\mathrm{i}\pa_tg_\alpha=\mu_2\pa_x^2g_\alpha-2\mu_2 g_\alpha g_\alpha^\dag g_\alpha. 
\end{equation*}
Note with this choice for $P_\beta$ the block form operator $P$ is Hermitian
with respect to the $L^2$ inner product, i.e.\/ $P^\dag=P$. Since $U=(\id-P^2)^{-1}$
is a power series in $P^2$ we observe that $U^\dag=U$. Hence for example
the adjoint of the composition $G=PU$ is given by $G^\dag=(PU)^\dag=U^\dag P^\dag=UP=G$.
We remark on two conclusions from this relation. The first is that the block form operator $G$
is also Hermitian, and the second is that the operator adjoint of $PU$ is $UP$.
Indeed we further conclude that the operator adjoint to $P_{a_1}UP_{a_2}P_{a_3}U\cdots UP_{a_{k-1}}P_{a_k}U$
is $UP_{a_{k}}P_{a_{k-1}}U\cdots UP_{a_3}P_{a_2}UP_{a_1}$. Hence in our
encoding of such operator monomials (see below) we observe that for any directed word,
$\bigl(a_1(a_2,a_3)\cdots(a_{k-1},a_k)\bigr)^\dag=(a_k,a_{k-1})\cdots(a_3,a_2)a_1$, and vice-versa.
\end{example}
\begin{example}[Reverse space-time nonlocal noncommutative nonlinear Schr\"odinger equation]\label{ex:nonlocalNLS} 
Another consistent choice is to set $P_\beta(x,t)=P_\alpha^{\mathrm{T}}(-x,-t)$, where $P_\alpha^{\mathrm{T}}$
is the Hilbert--Schmidt operator whose operator kenrel is the transpose of the matrix kernel corresponding
to $P_\alpha$. In other words we suppose $p_\beta(x,t)=p_\alpha^{\mathrm{T}}(-x,-t)$;
see Doikou \textit{et al.\/} \cite[Cor.~3.6]{DMS} and Malham~\cite[Cor.~3.2]{Malham:quinticNLS}.
Note, this choice is consistent with the linear dispersive partial differential equations 
$\pa_tp_\alpha=\mu\pa^2p_\alpha$ and $\pa_tp_\beta=-\mu\pa^2p_\beta$.
Further we observe that,  
$G_\alpha^{\mathrm{T}}(-x,-t)=\bigl(\id-P_\alpha^\mathrm{T}(-x,-t)P_\alpha(x,t)\bigr)^{-1}P_\alpha^\mathrm{T}(-x,-t)$,
and also that, 
$G_\beta(x,t)=\bigl(\id-P_\alpha(x,t)P_\alpha^\mathrm{T}(-x,-t)\bigr)^{-1}P_\alpha^\mathrm{T}(-x,-t)$.
Consequently we deduce $G_\beta(x,t)=G_\alpha^{\mathrm{T}}(-x,-t)$ and $g_\beta(0,0;x,t)=g_\alpha^{\mathrm{T}}(0,0;-x,-t)$.
Hence $g_\alpha=g_\alpha(0,0;x,t)$ satisfies the following version of the noncommutative nonlinear
Schr\"odinger equation with a nonlinearity involving a nonlocal reverse space-time factor
$g_\beta=g_\beta(0,0;x,t)\equiv g_\alpha^{\mathrm{T}}(0,0;-x,-t)$ as follows:
\begin{equation*}
\mathrm{i}\pa_tg_\alpha=\mu_2\pa_x^2g_\alpha-2\mu_2 g_\alpha g_\beta g_\alpha. 
\end{equation*}
We note that the operator $P_\beta(x,t)=P_\alpha^{\mathrm{T}}(-x,-t)$ is the
adjoint operator to $P_\alpha(x,t)$ with respect to the following space-time $\CC$-valued bilinear form
on the underlying Hilbert space:
\begin{equation*}
\int_{-T}^T\int_{-\infty}^\infty\psi^{\mathrm{T}}(-x,-t)\phi(x,t)\,\rd x\rd t.
\end{equation*}
Here $T>0$ is a fixed time horizon and note that we consider functions on the whole real line.
In particular if $p_\alpha=p_\alpha(y+z+x;t,s)$ is the space-time kernel corresponding to $P_\alpha(x,t)$
that is Hankel with respect to the spatial variables,
then $p_\beta(y+z+x;t,s)=p_\alpha^{\mathrm{T}}(-z-y-x;-s,-t)$ is the kernel corresponding to $P_\beta(x,t)$,
With this choice of $P_\beta$, the block form operator $P$ is self-adjoint with respect to
the corresponding block-extended space-time $\CC$-valued bilinear form, i.e.\/ $P^\dag=P$ 
where $P^\dag$ denotes the adjoint operator to $P$ with respect to the bilinear form.
As in the last example, we observe that $U^\dag=U$ while $G^\dag=(PU)^\dag=U^\dag P^\dag=UP=G$.
Hence $G^\dag=G$ and
$\bigl(P_{a_1}UP_{a_2}P_{a_3}U\cdots UP_{a_{k-1}}P_{a_k}U\bigr)^\dag=UP_{a_{k}}P_{a_{k-1}}U\cdots UP_{a_3}P_{a_2}UP_{a_1}$.
Hence in our encoding of such operator monomials we observe that for any directed word,
we have $\bigl(a_1(a_2,a_3)\cdots(a_{k-1},a_k)\bigr)^\dag=(a_k,a_{k-1})\cdots(a_3,a_2)a_1$, and vice-versa.
\end{example}
\begin{remark}
Similar arguments lead to the obvious corresponding cases for the noncommutative modified
Korteweg--de Vries equation, for which respectively, $g_\beta=g_\alpha^{\mathrm{T}}$
and in the reverse space-time case, $g_\beta(0,0;x,t)=g_\alpha^{\mathrm{T}}(0,0;-x,-t)$.
We can also generalise to complex valued noncommutative modified Korteweg--de Vries fields.
In the nonlinear Schr\"odinger case, we can also consider the case a nonlocal reverse time (only)
central factor in the cubic nonlinearity. See Doikou \textit{et al.\/} \cite{DMS} for more details.
Further we can also include the case of space-time shifted nonlocal linearities that were considered
by Ablowitz and Musslimani~\cite{AMshift}; see Malham~\cite{Malham:quinticNLS}.
\end{remark}  
In the following lemma we record some identities that will be useful below.
Note we set $U\coloneqq(\id-\hat Q)^{-1}$. Recall $P$ is a block operator with
non-zero blocks in the off-diagonal and $\hat Q\coloneqq P^2$ is thus a
block operator with non-zero blocks on the diagonal. Hence, using the
power series expansion for $(\id-\hat Q)^{-1}$, we observe that $U$
must have a diagonal block form with zero blocks on the off-diagonal.
\begin{lemma}\label{lemma:usefulNLSidentities}
The block operators $P$ and $U\coloneqq(\id-\hat Q)^{-1}$ satisfy the following identities:
\begin{equation*}
PU\equiv UP,\qquad \mathcal IU\equiv U\mathcal I\qquad\text{and}\qquad P\mathcal I\equiv-\mathcal IP.
\end{equation*}
\end{lemma}
\begin{proof}
  Since  $U\coloneqq(\id-P^2)^{-1}$, it has a power series expansion in $P^2$.
  Hence left composition of this form with $P$ equals right composition with $P$,
  thus establishing that $PU\equiv UP$. Since $P^2$ is block diagonal, and thus $U$
  is as well, pre- or post-composition with the block diagonal operator $\mathcal I$
  is the same. Hence we know $\mathcal IU\equiv U\mathcal I$. Lastly, using that
  $P$ is block off-diagonal, noting the block form of $\mathcal I$, direct computation
  reveals the final identity.
\qed
\end{proof}
We now proceed with the proof of Theorem~\ref{thm:NLS}.
We focus on the \emph{P\"oppe triple system} associated with the noncommutative nonlinear Schr\"odinger
and modified Korteweg--de Vries equations. In the prescription in Definition~\ref{prescription:KdVNLS},
we set $\hat Q=P^2$ and the linear Fredholm equation is $P=G(\id-\hat Q)$
or equivalently $G=PU$. Recall we also have $U=\id+U\hat Q=\id+\hat QU$ and
since $PU\equiv UP$, we also know $U=\id+PUP$. Recall the inverse operator expansions for $U_n$
in terms of monomials of the form $UQ_{a_1}U\hat Q_{a_2}U\cdots U\hat Q_{a_k}U$ from Lemma~\ref{lemma:IOE}.
Since $G_n=(PU)_n=(UP)_n$, by applying the Leibniz rule and then the expansions for $U_n$ from
Lemma~\ref{lemma:IOE}, we observe that we have,
\begin{align*}
  [G_n]&=\sum_{\ell=0}^n\begin{pmatrix} n\\\ell\end{pmatrix}
  \sum_{a_1\cdots a_k\in\mathcal C(\ell)}\chi(a_1\cdots a_k)\,[P_{n-\ell}U\hat Q_{a_1}U\cdots U\hat Q_{a_k}U]\\
  &\equiv\sum_{\ell=0}^n\begin{pmatrix} n\\\ell\end{pmatrix}
  \sum_{a_1\cdots a_k\in\mathcal C(n-\ell)}\chi(a_1\cdots a_k)\,[U\hat Q_{a_1}U\cdots U\hat Q_{a_k}UP_{\ell}],  
\end{align*}
after applying the kernel bracket operator at the very last step. In the expansions above, naturally $P_0=P$ and the terms
corresponding to $\ell=0$ and $\ell=n$ in the respective first and second terms on the right, are respectively $P_nU$ and $UP_n$.
We can rationalise these expressions further as follows.
From Definition~\ref{def:signaturecharacter} for the signature character $\chi$,
we observe that for any composition $w\in\mathcal C(n)$, we have (also see the proof of Lemma~\ref{lemma:IOE}),
\begin{equation*}
\begin{pmatrix} n+1\\n+1-\ell\end{pmatrix} \chi(w)=\chi\bigl((n+1-\ell)w\bigr).
\end{equation*}
Then since $n$ choose $\ell$ is the same as $n$ choose $n-\ell$, we observe we can write
\begin{align*}
[G_n]&=\sum_{a_1\cdots a_k\in\mathcal C(n)}\chi(a_1\cdots a_k)\,\bigl([P_{a_1}UQ_{a_2}U\cdots UQ_{a_k}U]+[PUQ_{a_1}U\cdots UQ_{a_k}U]\bigr)\\
     &\equiv\sum_{a_1\cdots a_k\in\mathcal C(n)}\chi(a_1\cdots a_k)\,\bigl([UQ_{a_1}U\cdots UQ_{a_{k-1}}UP_{a_k}]+[UQ_{a_1}U\cdots UQ_{a_k}UP]\bigr).  
\end{align*}
We require the following modification of Lemma~\ref{lemma:Poppeprodkernels} for the
P\"oppe product for kernel monomials. In the present context we have two monomial
forms, namely those of the form $[P_{a_1}U\hat Q_{a_2}U\cdots U\hat Q_{a_k}U]$ and
those of the form $[U\hat Q_{a_1}U\cdots U\hat Q_{a_{k-1}}UP_{a_k}]$.
Recall from Examples~\ref{ex:NLS} and \ref{ex:nonlocalNLS}, we regard these
monomials as adjoints of each other. For any $a\in\mathbb N$,
each term of the form $\hat Q_a$ within these monomials has the expansion,
\begin{equation*}
\hat Q_a=(PP)_a=\sum_{\ell=0}^a\begin{pmatrix} a\\\ell\end{pmatrix}P_{a-\ell}P_\ell.
\end{equation*}
Hence we should consider the monomials $\bigl[P_{a_1}U(P_{a_2}P_{a_3})U\cdots U(P_{a_{k-1}}P_{a_k})U\bigr]$
as our basis, and those of the form $\bigl[U(P_{a_1}P_{a_2})U\cdots U(P_{a_{k-2}}P_{a_{k-1}})UP_{a_k}\bigr]$
as adjoint monomials. We observe that particular expansions such as $G_n$, for all $n\in\mathbb N\cup\{0\}$,
are self-adjoint. It will be convenient to retain the expansion forms  
$[P_{a_1}U\hat Q_{a_2}U\cdots U\hat Q_{a_k}U]$ and $[U\hat Q_{a_1}U\cdots U\hat Q_{a_{k-1}}UP_{a_k}]$
in terms of these latter basis monomials.
The modification of Lemma~\ref{lemma:Poppeprodkernels} is as follows.
\begin{lemma}[Modified product for kernel monomials]\label{lemma:PoppeNLSrule}
Suppose $u=a_1\cdots a_k\in\Cb$ and $v=b_1\cdots b_\ell\in\Cb$ 
while $a,b,c,d\in\mathbb N\cup\{0\}$. Let $\mm_u$ and $\mm_v$ denote the monomials
$\mm_u\coloneqq P_{a_1}U(P_{a_2}P_{a_3})U\cdots U(P_{a_{k-1}}P_{a_k})U$ and
$\mm_v\coloneqq U(P_{b_1}P_{b_2})U\cdots U(P_{b_{\ell-2}}P_{b_{\ell-1}})UP_\ell$.
Then we have 
\begin{align*}
[\mm_u(P_aP_b)U][U(P_cP_d)\mm_v]
=&\;[\mm_u(P_aP_{b+1})U(P_cP_d)\mm_v]+[\mm_u(P_aP_b)U(P_{c+1}P_d)\mm_v]\\
&\;+[\mm_u(P_aP_b)UQ_1U(P_cP_d)\mm_v].
\end{align*}
\end{lemma}
\begin{proof}
Using that $U\equiv\id+UP^2\equiv\id+P^2U$ and the P\"oppe product rule from Lemma~\ref{lemma:productrule}, we see that 
\begin{align*}
  [\mm_u&(P_aP_b)U][U(P_cP_d)\mm_v]\\
  =&[\mm_u(P_aP_b)+\mm_u(P_aP_b)U(PP)][(PP)U(P_cP_d)\mm_v+(P_cP_d)\mm_v]\\
=&[\mm_u(P_aP_b)][(P_cP_d)\mm_v]+[\mm_u(P_aP_b)][(PP)U(P_cP_d)\mm_v]\\    
&\;+[\mm_u(P_aP_b)U(PP)][(P_cP_d)\mm_v]+[\mm_u(P_aP_b)U(PP)][(PP)U(P_cP_d)\mm_v]\\
=&[\mm_u(P_aP_{b+1})(P_cP_d)\mm_v]+[\mm_u(P_aP_b)(P_{c+1}P_d)\mm_v]\\
&\;+[\mm_u(P_aP_{b+1})(PP)U(P_cP_d)\mm_v]+[\mm_u(P_aP_b)(P_1P)U(P_cP_d)\mm_v]\\    
&\;+[\mm_u(P_aP_b)U(PP_1)(P_cP_d)\mm_v]+[\mm_u(P_aP_b)U(PP)(P_{c+1}P_d)\mm_v]\\
&\;+[\mm_u(P_aP_b)U(PP_1)(PP)U(P_cP_d)\mm_v]+[\mm_u(P_aP_b)U(PP)(P_1P)U(P_cP_d)\mm_v]\\
=&\;[\mm_u(P_aP_{b+1})U(P_cP_d)\mm_v]+[\mm_u(P_aP_b)U(P_{c+1}P_d)\mm_v]\\&\;+[\mm_u(P_aP_b)U(P_1P)U(P_cP_d)\mm_v]+[\mm_u(P_aP_b)U(PP_1)U(P_cP_d)\mm_v],
\end{align*}
which, upon combining the last two terms, gives the result.\qed
\end{proof}
As we did for the Korteweg--de Vries equation, it is now useful to abstract the structure above.
We code the two forms of monomials $\bigl[P_{a_1}U(P_{a_2}P_{a_3})U\cdots U(P_{a_{k-1}}P_{a_k})U\bigr]$ and
$\bigl[U(P_{a_1}P_{a_2})U\cdots U(P_{a_{k-2}}P_{a_{k-1}})UP_{a_k}\bigr]$ as follows. For $w=a_1a_2\cdots a_k$ we write:
\begin{align*}
  \bigl[P_{a_1}U(P_{a_2}P_{a_3})U\cdots U(P_{a_{k-1}}P_{a_k})U\bigr]
  &\quad\to\quad w=a_1(a_2,a_3)\cdots(a_{k-1},a_k)\\
  \bigl[U(P_{a_k}P_{a_{k-1}})U\cdots U(P_{a_{3}}P_{a_{2}})UP_{a_1}\bigr]
  &\quad\to\quad w^\dag=(a_k,a_{k-1})\cdots(a_{3},a_{2})a_1.
\end{align*}
The single letters at the ends without brackets represent the single $P_{a_1}$ operators.
We call the forms $w$ and $w^\dag$ \emph{directed words}. Words of the form $w$ will be our basis, while
those of the form $w^\dag$ form the adjoint basis.
As we have seen, the operators $[G_n]$ can be expressed in terms of expansions
of the monomial forms $[P_{a_1}U\hat Q_{a_2}U\cdots U\hat Q_{a_k}U]$ and $[U\hat Q_{a_1}U\cdots U\hat Q_{a_{k-1}}UP_{a_k}]$, which
are themselves expansions once the Leibniz rule has been applied to each of the terms $Q_a=(PP)_a$. As mentioned above,
it is convenient to retain these expansion terms and we encode them as follows,
\begin{align*}
  \bigl[P_{a_1}UQ_{a_2}U\cdots UQ_{a_{k-1}}UQ_{a_k}U\bigr]  &\quad\to\quad w_\ast=a_1(a_2)\cdots(a_{k-1})(a_k)\\
  \bigl[UQ_{a_k}UQ_{a_{k-1}}U\cdots UQ_{a_{2}}UP_{a_1}\bigr] &\quad\to\quad w_\ast^\dag=(a_k)(a_{k-1})\cdots(a_{2})a_1.
\end{align*}
Here, for pure form expansions involving only $Q_a$ factors in the middle,
note the bracket the single letters `$(a)$' in the middle corresponding to the $Q_a$,
however we leave the letters correpsonding to the operators $P_{a_1}$ at the ends
without brackets to distinguish them. We reserve the notation $w_\ast$ and $w_\ast^\dag$
for the forms above involving only $Q_a$ terms in the middle.
Note, we can mix these two notations, so that for example, we encode:
\begin{equation*}
  \bigl[P_{a_1}UQ_{a_2}U(P_{a_3}P_{a_4})UQ_{a_5}U\cdots UQ_{a_k}U\bigr]
  \quad\to\quad w=a_1(a_2)(a_3,a_4)(a_5)\cdots(a_k).
\end{equation*}
Let $\mathbb N_0$ denote the set of non-negative integers $\mathbb N_0\coloneqq\mathbb N\cup\{0\}$.
Further, let $\mathbb N_0^\ast$ denote the free monoid of directed words
on $\mathbb N_0$ and let $\CC\mathbb N_0$ denote the vector space of directed words
on $\mathbb N_0$ over the field $\CC$. Note we can define a scalar
product on $\mathbb N_0^\ast$ which is unity when two words are the same and zero when they are different.
We must keep in mind however that some letters and expansions are selfadoint.
For example the single letter $0\in\CC\mathbb N_0$ corresponds to $PU$ and thus $UP$ at the same time,
so $0^\dag=0$ in this abstract context.The directed word expansions corresponding to $(PU)_a$ are
also selfadjoint. We outline these in more detail presently.
For this scalar product, the free monoid $\mathbb N_0^\ast$ forms an orthonormal basis
and the vector space $\CC\mathbb N_0$ is well defined; see Reutenauer~\cite[p.~17]{Reutenauer}.
On $\CC\mathbb N_0$, consider the following version of the \emph{P\"oppe product}, based
on the original P\"oppe product in Lemma~\ref{lemma:Poppeprodkernels}
and the consequential form in Lemma~\ref{lemma:PoppeNLSrule}.
\begin{definition}[P\"oppe product for directed words]\label{def:PoppeproductNLS}
  We define the \emph{P\"oppe product} `$\ast$' of the directed words
  $u^\dag=(a_1,a_2)\cdots(a_{k-2},a_{k-1})a_k$  and $v=b_1(b_2,b_3)\cdots(b_{\ell-1},b_\ell)$ 
  as follows:
  \begin{align*}
  u^\dag\ast v=&\;(a_1,a_2)\cdots(a_{k-2},a_{k-1})(a_k+1,b_1)(b_2,b_3)\cdots(b_{\ell-1},b_\ell)\\
  &\;+(a_1,a_2)\cdots(a_{k-2},a_{k-1})(a_k,b_1+1)(b_2,b_3)\cdots(b_{\ell-1},b_\ell).    
  \end{align*}
  Further, we define the \emph{P\"oppe product} `$\ast$' between
  $u=a_1(a_2,a_3)\cdots(a_{k-1},a_k)$ and $v^\dag=(b_1,b_2)\cdots(b_{\ell-2},b_{\ell-1})b_\ell$ as follows:
  \begin{align*}
  u\ast v^\dag=&\;a_1(a_2,a_3)\cdots(a_{k-1},a_k+1)(b_1,b_2)\cdots(b_{\ell-2},b_{\ell-1})b_\ell\\
  &\;+a_1(a_2,a_3)\cdots(a_{k-1},a_k)(b_1+1,b_2)\cdots(b_{\ell-2},b_{\ell-1})b_\ell\\
  &\;+a_1(a_2,a_3)\cdots(a_{k-1},a_k)(1)(b_1,b_2)\cdots(b_{\ell-2},b_{\ell-1})b_\ell.
  \end{align*}
  The empty word $\nu\in\CC\mathbb N_0$ satisfies $\nu\ast w=w\ast\nu=w$
  and $\nu\ast w^\dag=w^\dag\ast\nu=w^\dag$ for any $w, w^\dag\in\CC\mathbb N_0$.
\end{definition}
\begin{remark}\label{rmk:tripleproducts}
  The products $u^\dag\ast v$ and $u\ast v^\dag$ do not generate closed form expressions
  in terms of directed words of the form $w$ and/or $w^\dag$ in $\CC\mathbb N_0$. Observe
  that $u^\dag\ast v$ does not generate single letters at either end corresponding to
  single operators $P_a$ at the beginning of words $w$ or the end of words $w^\dag$.
  On the other hand the product $u\ast v^\dag$ generates single letters at both ends of the words involved.
  Further, note that for the product $u\ast v^\dag$, a third correction
  term is present in the product reflecting the analogous term in Lemma~\ref{lemma:PoppeNLSrule}.
  This correction term involves the insertion of a factor `$(1)$' between the directed words
  $u$ and $v^\dag$, corresponding to the insertion of $Q_1$ in
  the corresponding term in Lemma~\ref{lemma:PoppeNLSrule}. Further note that $(1)$ corresponds
  to $Q_1=(PP)_1=P_1P_0+P_0P_1$ and thus $(1)=(1,0)+(0,1)$.
  Lastly, as mentioned above, the letter $0=0^\dag$ corresponding to $PU=UP$.
  Products such as $0^\dag\ast w$ and $0\ast w^\dag$ follow as expected according to the rules outlined.
\end{remark}
Though the products $u^\dag\ast v$ and $u\ast v^\dag$ do not generate closed form expressions
in $\CC\mathbb N_0$, the triple products
\begin{equation*}
  u\ast v^\dag\ast w \qquad\text{and}\qquad u^\dag\ast v\ast w^\dag,
\end{equation*}
do generate closed form expressions in terms of directed words. Indeed any $\ast$-products
involving an odd number of words in which the factors alternate between the directed words and
adjoint directed words such as $w_1\ast w_2^\dag\ast w_3\ast\cdots\ast w_{2k}^\dag\ast w_{2k+1}$
or $w_1^\dag\ast w_2\ast w_3^\dag\ast\cdots\ast w_{2k}^\dag\ast w_{2k+1}$,
also generate closed form expressions in terms of directed words.
This property means that we can construct a \emph{triple system}. For such systems we require a trilinear map
$\mu\colon(\CC\mathbb N_0)^{\times 3}\to\CC\mathbb N_0$;
see Meyberg~\cite[p.~21]{Meyberg} or Ricciardo~\cite[p.~23]{Ricciardo}.
\begin{definition}[P\"oppe triple product]
We define the \emph{P\"oppe triple product} $\mu$ by: 
\begin{equation*}
  \mu\colon(u,v,w)\mapsto u\ast v^\dag\ast w. 
\end{equation*}
\end{definition}
\begin{remark}
The map $\mu$ is trilinear and, for the triple product above, we have 
\begin{equation*}
  \mu\bigl(w_1,w_2,\mu(w_3,w_4,w_5)\bigr)=\mu\bigl(w_1,\mu(w_,w_3,w_4),w_5\bigr)=\mu\bigl(\mu(w_1,w_2,w_3),w_4,w_5\bigr), 
\end{equation*}
as all three equal $w_1\ast w_2^\dag\ast w_3\ast w_4^\dag\ast w_5$, and so forth.
\end{remark}
The vector space $\CC\mathbb N_0$ endowed with this triple product, represents the \emph{P\"oppe triple system}.
We observe that if we restrict ourselves to the subset of directed words $w$ in $\CC\mathbb N_0$ and not their adjoints $w^\dag$,
then the triple product generates directed words of the same form $w$ and we can consider the 
triple subsystem of such directed words in $\CC\mathbb N_0$. Let us denote this
triple subsystem by $\CC\langle\mathbb N_0\rangle_\mu$. 
Lastly, we define the abstract version of the expansions for $[G_n]=[(PU)_n]=[(UP)_n]$ above,
which we denote as signature expansions in $\CC\mathbb N_0$. 
\begin{definition}[Signature expansions: directed words]\label{def:directedwordsig}
  For any $n\in\Nb$ we define the following linear \emph{signature expansions} $\bs{n}\in\CC\mathbb N_0$:
\begin{equation*}
  \boldsymbol{n}\coloneqq\sum_{w\in\mathcal C(n)}\chi(w)\,\bigl(w_\ast+(0w)_\ast\bigr) 
  \equiv\sum_{w\in\mathcal C(n)}\chi(w)\,\bigl(w_\ast^\dag+(w0)_\ast^\dag\bigr),
\end{equation*}
where if the word $w_\ast=a_1(a_2)\cdots(a_k)$ then $(0w)_\ast=0(a_1)(a_2)\cdots(a_k)$,
and similarly if the word $w_\ast^\dag=(a_k)\cdots(a_2)a_1$ then $(w0)_\ast^\dag=(a_k)\cdots(a_2)(a_1)0$. 
Note the expansions $\bs n$ are selfadjoint so $\bs n^\dag=\bs n$. Further note that $\bs 0=0=0^\dag$.
\end{definition}

Let us now return to the noncommutative $n$th order nonlinear Schr\"odinger equation at hand.
Suppose $P$, $\hat Q$ and $G$ satisfy the linear system in Definition~\ref{prescription:KdVNLS}.
Hence if $U\coloneqq(\id-\hat Q)^{-1}$ then $G=PU$ and using that $\pa_t P=-\mu(\mathrm{i}\mathcal I)^{n-1}\pa^nP$,
we observe, using the notation $P_n\coloneqq\pa^nP$ and $U_n\coloneqq\pa^nU$ and the
identities from Lemma~\ref{lemma:usefulNLSidentities}, that,
\begin{align*}
  \pa_tG&=-\mu(\mathrm{i}\mathcal I)^{n-1}P_nU+PU\bigl((\pa_tP)P+P(\pa_tP)\bigr)U\\
  &=-\mu(\mathrm{i}\mathcal I)^{n-1}P_n U-\mu\mathrm{i}^{n-1}PU\bigl(\mathcal I^{n-1}P_nP+P\mathcal I^{n-1}P_n\bigr)U\\ 
  &=-\mu(\mathrm{i}\mathcal I)^{n-1}\bigl(P_n U+(-1)^{n-1}PUP_nPU+PUPP_nU\bigr).
\end{align*}  
In terms of our directed word encoding for such operator forms the right-hand side in this last expression,
modulo $-\mu(\mathrm{i}\mathcal I)^{n-1}$, respectively term by term, has the form:
\begin{equation*}
n+(-1)^{n-1}\cdot0(n,0)+0(0,n).
\end{equation*}
Note the single basic letter $n\in\CC\la\mathbb N_0\ra_\mu$ corresponds to $P_nU$.
\begin{remark}[Notation convention]
  Note in the second term `$(-1)^{n-1}\cdot0(n,0)$' above, we have separated the coefficient $(-1)^{n-1}\in\CC$
  from the directed word $0(n,0)\in\mathbb N_0^\ast$.
  We use this notation henceforth, i.e.\/ we separate the coefficients, say $\alpha\in\CC$, from the directed
  words $w\in\mathbb N_0^\ast$ using the notation $\alpha\cdot w$; as we did for the P\"oppe algebra.
\end{remark}
\begin{remark}\label{rmk:NLSsimplify}
  If we use $U=\id+PUP$ in the operator relation above, it simplifies to 
  $\pa_tG=-\mu(\mathrm{i}\mathcal I)^{n-1}\bigl(U(\pa^n P)U+(-1)^{n-1}G(\pa^nP)G\bigr)$.
\end{remark}
Our goal is to show that the form $n+(-1)^{n-1}\cdot0(n,0)+0(0,n)$ can be expressed in terms of a triple system P\"oppe polynomial
$\pi_n=\pi_n(\bs 0,\bs 1,\bs 2,\ldots,\bs n)$ of the signature expansions of the form,
\begin{equation*}
\pi_n\coloneqq \sum_{k=1}^{n+1}\sum_{a_1a_2\cdots a_k} c_{a_1a_2\cdots a_k}\bs a_1\ast\bs a_2\ast\cdots\ast\bs a_k,
\end{equation*}
in $\CC\mathbb N_0$, where the second sum is over all words $a_1a_2\cdots a_k$ we can construct from the
alphabet $\{0,1,2,\ldots,n\}$ such that $a_1+a_2+\cdots+a_k=n-k+1$, and the coefficients $c_{a_1a_2\cdots a_k}$, here at least,
are real. 

We are now in a position to prove the results of Theorem~\ref{thm:NLS}, which we achieve through
the following two examples.
\begin{example}[Nonlinear Schr\"odinger equation integrability]\label{ex:illustrativeNLSex}
  Assume $n=2$. Recall $\chi(1)=1$ and thus from the directed word signature expansion in Definition~\ref{def:directedwordsig},
  we observe,
  \begin{equation*}
    \boldsymbol{1}=1+0(1)=1^\dag+(1)0
    \qquad\text{and}\qquad (1)=(1,0)+(0,1).
  \end{equation*}
  Hence working from right to left, we observe:  
  \begin{align*}
    \bs 0\ast\bs 0\ast \bs 0=&
    0\ast0^\dag\ast 0\\
    =&0\ast\bigl((0,1)+(1,0)\bigr)\\
    =&0\ast(0,1)+0\ast(1,0)\\
    =&1(0,1)+0(1,1)+0(1)(0,1)+1(1,0)+0(2,0)+0(1)(1,0)\\
    =&1(1)+0(1,1)+0(1)(1)+0(2,0).
  \end{align*}
  Now consider the form $n+(-1)^{n-1}\cdot0(n,0)+0(0,n)-\bs n$ with $n=2$, which gives
  \begin{align*}
    2-0(2,&0)+0(0,2)-\bs 2\\
    =&\;2-0(2,0)+0(0,2)-\Bigl(\chi(2)\cdot\bigl(2+0(2)\bigr)+\chi(11)\cdot\bigl(1(1)+0(1)(1)\bigr)\Bigr)\\
    =&\;2-0(2,0)+0(0,2)-2-0\bigl((2,0)+2\cdot(1,1)+(0,2)\bigr)-2\cdot\bigl(1(1)+0(1)(1)\bigr)\\
    =&\;-2\cdot\bigl(0(1,1)+1(1)+0(1)(1)+0(2,0)\bigr).
  \end{align*}
  This is obviously `$-2$' times the result of the product $\bs 0\ast\bs 0\ast \bs 0$.
  We have thus established the identity $2-0(2,0)+0(0,2)\equiv\bs 2-2\cdot\bs 0\ast\bs 0\ast \bs 0$
  giving the required result, and that the noncommutative nonlinear Schr\"odinger equation is integrable.
\end{example}
We now consider the next order in the noncommutative nonlinear Schr\"odinger hierarchy,
the noncommutative modified Korteweg--de Vries equation,
and at the same time illustrate the power of the P\"oppe triple system approach we introduced. 
\begin{example}[Modified Korteweg--de Vries equation integrability]\label{ex:mKdV}
Assume $n=3$. First we consider the following signature expansions for directed words.
We observe, working from right to left, that,
\begin{align*}
\bs 1\ast\bs 0\ast\bs 0=&\;\bs 1\ast 0^\dag\ast0\\
=&\;\bs 1\ast\bigl((0,1)+(1,0)\bigr)\\
=&\;\bigl(1+0(1)\bigr)\ast\bigl((0,1)+(1,0)\bigr)\\
=&\;2(1)+1\bigl((1,1)+(2,0)\bigr)+1(1)(1)
+0\Bigl(\bigl((0,1)+(1,0)\bigr)\ast\bigl((0,1)+(1,0)\bigr)\Bigr)\\
=&\;2(1)+1\bigl((2,0)+(1,1)+(1)(1)\bigr)\\
&\;+0\bigl((1,1)(1)+(1)(1,1)+(1)(2,0)+(0,2)(1)+(1)(1)(1)\bigr),
\end{align*}
where the last step involves using the identity $(1)\equiv(1,0)+(0,1)$ multiple times.
Similar computations reveal that,
\begin{align*}
\bs 0\ast\bs 0\ast\bs 1=&\;\bs 0\ast0^\dag\ast\bs 1\\
=&\;\bs 0\ast0^\dag\ast\bigl(1+(0)(1)\bigr)\\
=&\;0\ast\bigl((1,1)+(0,2)+(1,0)(1)+(0,1)(1)\bigr)\\
=&\;1\bigl((0,2)+(1,1)+(1)(1)\bigr)\\
&\;+0\bigl((2,1)+(1)(1,1)+(1,2)+(1)(0,2)+(2,0)(1)+(1,1)(1)+(1)(1)(1)\bigr),
\intertext{and}
\bs 0\ast\bs 1\ast\bs 0=&\;\bs 0\ast\bigl(1^\dag+(1)0\bigr)\ast0\\
=&\;\bs 0\ast\bigl((2,0)+(1,1)+(1)(1)(1)\bigr)\\
=&\;1\bigl((2,0)+(1,1)+(1)(1)\bigr)\\
&\;+0\bigl((3,0)+(1)(2,0)+(2,1)+(1)(1,1)+(2,0)(1)+(1,1)(1)+(1)(1)(1)\bigr).
\end{align*}
Now consider the form $n+(-1)^{n-1}\cdot0(n,0)+0(0,n)-\bs n$ with $n=3$, which gives
\begin{align*}
  3+0(3,0)+0(0,3)-\bs 3=&\;3+0(3,0)+0(0,3)\\
  &\;-\Bigl(\chi(3)\cdot\bigl(3+0(3)\bigr)+\chi(21)\cdot\bigl(2(1)+0(2)(1)\bigr)\\
  &\;+\chi(12)\cdot\bigl(1(2)+0(1)(2)\bigr)+\chi(111)\cdot\bigl(1(1)(1)+0(1)(1)(1)\bigr)\Bigr)\\
  =&\;-3\cdot0\bigl((2,1)+(1,2)\bigr)-3\cdot\bigl(2(1)+0(2)(1)\bigr)\\
  &\;-3\cdot\bigl(1(2)+0(1)(2)\bigr)-6\cdot\bigl(1(1)(1)+0(1)(1)(1)\bigr).
\end{align*}
So the goal now is to determine if there exist constants $\alpha$, $\beta$ and $\gamma$, at least one
of which is non-zero, such that (modulo $-3$),
\begin{equation*}
   0\bigl((2,1)+(1,2)\bigr)+2(1)+0(2)(1)+1(2)+0(1)(2)+2\cdot\bigl(1(1)(1)+0(1)(1)(1)\bigr)=\hat\pi_3(\bs 0, \bs 1),
\end{equation*}
where the right-hand side $\hat\pi_3$ is the polynomial,
\begin{equation*}
  \hat\pi_3(\bs 0, \bs 1)\coloneqq\alpha\cdot\bs 1\ast\bs 0\ast\bs 0+\beta\cdot\bs 0\ast\bs 1\ast\bs 0+\gamma\cdot\bs 0\ast\bs 0\ast\bs 1.
\end{equation*}
Substituting the expressions for $\bs 1\ast\bs 0\ast\bs 0$, $\bs 0\ast\bs 1\ast\bs 0$
and $\bs 0\ast\bs 0\ast\bs 1$ just above into the right-hand of the relation,
and equating coefficients of the terms shown, we arrive at an overdetermined linear
system of algebraic equations $AC=B$ for the vector of three unknowns $C\coloneqq(\alpha, \beta, \gamma)^{\mathrm{T}}$.
In Table~\ref{table:coeffs} we give the coefficient matrix $A$ and the entries in the non-homogeneous vector $B$. 
We observe the first three rows give us natural pivot terms along the diagonal and solving these three equations
reveals that $\alpha=\gamma=1$ and $\beta=0$. It is easy to check that the remaining equations are consistent
with this solution. Hence we have thus established the identity,
\begin{equation*}
  3+0(3,0)+0(0,3)\equiv\bs 3-3\cdot(\bs 1\ast\bs 0\ast\bs 0+\bs 0\ast\bs 0\ast\bs 1).
\end{equation*}
This establishes that the noncommutative modified Korteweg--de Vries equation is integrable.
\end{example}

\begin{remark}[Ordered word basis]
  In Example~\ref{ex:mKdV} we used a particular subcollection of the ordered words on $\CC\mathbb N_0$
  associated with the order of the problem $n=3$. These are given down the left column in Table~\ref{table:coeffs}.
  In general, we may need to split the subcollection shown in to a finer basis, so that for example, we may
  need to split $2(1)$ separately into $2(1,0)$ and $2(0,1)$, and so forth.
\end{remark}

\begin{table}
\caption{Coefficients appearing in the expansion of the polynomial $\hat\pi_3=\hat\pi_3(\bs 0,\bs 1)$. 
Each column shows the factor contributions to the real coefficients of the directed words 
shown in the very left column, for each of the monomials in $\hat\pi_3$ shown
across the top row. The final column represents the right-hand side $B$ of the equation $AC=B$.}
\label{table:coeffs}
\begin{center}
\begin{tabular}{|l|ccc|c|}
\hline
$\phantom{\biggl|}\CC\mathbb N_0$&$\bs 1\ast\bs 0\ast\bs 0$&$\bs 0\ast\bs 1\ast\bs 0$ & $\bs 0\ast\bs 0\ast\bs 1$ &$B$\\
\hline
$2(1)$ & 1 & & & 1\\
\hline
$1(2,0)$ & 1 & 1 &   & 1\\
$1(1,1)$ & 1 & 1 & 1 & 2\\
$1(0,2)$ &   &   & 1 & 1\\
$1(1)(1)$ & 1 & 1 & 1 & 2\\
\hline
$0(3,0)$ &   & 1 &  & 0\\
$0(2,1)$ &   & 1 & 1 & 1\\
$0(1,2)$ &   &   & 1 & 1\\
$0(1)(2,0)$ & 1 & 1 &   & 1\\
$0(1)(1,1)$ & 1 & 1 & 1 & 2\\
$0(1)(0,2)$ &   &   & 1 & 1\\
$0(2,0)(1)$ &   & 1 & 1 & 1\\
$0(1,1)(1)$ & 1 & 1 & 1 & 2\\
$0(0,2)(1)$ & 1 &   &   & 1\\
$0(1)(1)(1)$ & 1 & 1 & 1 & 2\\
\hline
\end{tabular}
\end{center}
\end{table}

\begin{remark}[Scalability]
Examples~\ref{ex:illustrativeKdVex}, \ref{ex:illustrativeNLSex} and \ref{ex:mKdV} 
demonstrate that the P\"oppe algebra and triple system can be used to succintly prove identities that
establish integrability for the noncommutative potential Korteweg-de Vries,
nonlinear Schr\"odinger and modified Korteweg--de Vries equations, respectively.
Example~\ref{ex:mKdV} for the noncommutative modified Korteweg--de Vries equation
in particular, demonstrates that these combinatorial structures provide a systematic approach
establishing integrability for higher order members of the respective hierarchies.
Indeed in Malham~\cite{Malham:KdVhierarchy}, the P\"oppe algebra was used to prove the existence of a unique
non-commutative integrable equation at each odd order, and that, that
equation matched the Lax hierarchy equation.
\end{remark}

\begin{remark}[Scattering]
We do not pursue the scattering and inverse scattering problems associated with the nonlinear Schr\"odinger equation
herein, and, in particular, how to reconstruct the scattering data $p$ from arbitrary initial data $g=g(0,0;x,0)$.
See Zakharov and Shabat~\cite{ZS,ZS2} and Drazin and Johnson~\cite{DJ}, as well as Dodd \textit{et al.}~\cite{DEGM}
for details on this. However, given arbitrary smooth initial scattering data in the form $p_0=p_0(x)$,
integrable on $\R$, we can generate solutions $g=g(0,0;x,t)$ to either the noncommutative nonlinear Schr\"odinger or
modified Korteweg-de Vries equations, as follows. First we generate the corresponding solution $p=p(x;t)$ to the linear
equation $\pa_tp=\mu(\mathrm{i}\mathcal I)^{n-1}\pa_x^np$ analytically, either in Fourier space or as a convolution integral.
Second, we then generate the corresponding function $\hat q=\hat q(x;t)$
by computing the integral associated with the composition $P^2$ of the operator $P$. 
Thirdly and finally, we then, for any given time $t\geqslant0$, generate a corresponding solution $g=g(0,0;x,t)$
by solving the linear Fredholm equation in Definition~\ref{prescription:KdVNLS} for $g(0,z;x,t)$.
Setting $z=0$ generates the corresponding solution $g=g(0,0;x,t)$ to the corresponding nonlinear equation.
We demonstrate this approach in practice with numerical simulations in Appendix~\ref{sec:numericalsimulations}.
\end{remark}

\section{Discussion}\label{sec:discussion} 
There are many further directions and extensions of the results we have presented 
that we intend to pursue next. For example, of immediate interest would be to extend
the triple system algebraic approach, presented in Section~\ref{sec:NLS} for the
noncommutative nonlinear Schr\"odinger and modified Korteweg--de Vries equations, to all
orders. For example, this has been completed for the whole noncommutative potential Korteweg--de Vries
hierarchy in Malham~\cite{Malham:KdVhierarchy}. Of particular interest will be the uniqueness
of the coefficients in the polynomial expansions for order $n\geqslant5$, when several
independent integrable equations, which are not reducible to each other, are known to exist,
see for example Gerdjikov~\cite{Gerdjikov}.

There are now also many stochastic representations for solutions of integrable equations
such as the Korteweg--de Vries equation.
See for example Thieullen and Vigot~\cite{TV}, where the Fredholm determinant form for
the solution is related to the characteristic functions of simple processes such
as Ornstein--Uhlenbeck processes and square-root mean-reverting processes.
There is also the connection to random matrices, see for example Tracy and Widom~\cite{TracyWidom}.
To add to these connections we make the following observation regarding a
connection between a stationary solution of the Fisher--Kolmogorov--Petrovsky--Piskunov (FKPP)
equation which decay to zero in the far field, and soliton solutions of the scalar Korteweg--de Vries equation.
Consider stationary solutions $U=U(z)$ of the FKPP equation which we assume here has the form, 
$U''+\beta U^2-\gamma U=0$.
Here the constant $\beta\in\R$ is related to the autocatalytic reaction rate and the constant
$\gamma\in\R$ is related to the linear decay rate. Suppose we look for travelling wave
solutions to the scalar Korteweg--de Vries equation of the form $V=V(z)$, where $z=x+\gamma t$,
so they satisfy, $V'''+2\beta VV'-\gamma V'=0$.
The choice $\beta=-3$ corresponds to the form of the Korteweg--de Vries equation we considered herein.
Assuming the fields decay in the far field, then integrating the equation for $V=V(z)$ reveals
that $U\equiv V$. Indeed we can look for a solution of the form $U=A\mathrm{sech}^2\alpha z$
to the equation for $U$, and show that it is indeed a solution provided $\gamma=4\alpha^2$
and $\beta A=6\alpha^2$. In the case $\beta=-3$, then $A=-2\alpha^2$. Thus if we set $\alpha=a/2$,
then we observe that $U=-\frac12 a^2\mathrm{sech}^2(\frac12 az)$ is a solution to both the
stationary FKPP equation and a soliton solution to the Korteweg--de Vries equation with
the travelling coordinate $z=x+a^2t$. It is well-known that the FKPP equation solution has
a representation in terms of branched Brownian motion, see Berestycki~\cite{Berestycki}.
An intriguing open question concerns how the pinned stationary `$\mathrm{sech}^2$' solution
to the FKPP equation fits into that context.

We have considered solutions to the Korteweg--de Vries and nonlinear Schr\"odinger
equations via Grassmannian flows using the top cell coordinate patch $\Lambda_0(\Hb,\Vb)$.
This paramaterisation relied on the regularised determinant $\mathrm{det}_2\bigl(\id-\hat Q(t)\bigr)\neq0$.
There are singular solutions of the scalar Korteweg--de Vries equation such as the classical `$\mathrm{cosech}^2$' solution
or rational solitary waves (though the latter lie outwith our current setup as they do not decay as $x\to-\infty$);
see Drazin and Johnson \cite[p.~36,18]{DJ}. 
For more details on singular solutions, such as breathers and so forth, see Dubard \textit{et al.\/} \cite{DGKM}.
For the current discussion, let focus solely on the scalar potential Korteweg--de Vries equation.
The singlarities in the example solutions mentioned, propagate continuously in time
at values $x_\ast=x_\ast(t)$ that evolve/move as time progresses. Away from the singularites at $(x_\ast,t_\ast)$,
the regularised determinant is not zero and the paramaterisation $G=P(\id-P)^{-1}$ and solution $g=[G]$
are well-defined. However at $(x_\ast,t_\ast)$, the regularised determinant $\mathrm{det}_2\bigl(\id-P(x_\ast,t_\ast)\bigr)=0$.
Let us now take the following perspective. The evolution of the Hilbert--Scmidt parameterisation $G$,
is a manifestation of the evolution of the linear flow for $P$, which is not singular, and
a particular coordinate patch representation for the Fredholm Grassmannian, namely that of the
top cell/patch $\Lambda_0(\Hb,\Vb)$. At the points $(x_\ast,t_\ast)$, we are simply choosing a
poor coordinate patch representation for the flow
\begin{equation*}
W=\begin{pmatrix} \id-P\\P\end{pmatrix}, 
\end{equation*}
on the Fredholm Stiefel manifold $\mathrm{St}(\Hb,\Vb)$, when it is projected onto the Fredholm Grassmannian $\mathrm{Gr}(\Hb,\Vb)$. 
In the case of the potential Korteweg--de Vries equation, the Stiefel flow satisfies,
\begin{equation*}
\pa_t\begin{pmatrix} \id-P\\P\end{pmatrix}=\mu\pa_x^n\begin{pmatrix} \id-P\\P\end{pmatrix},
\end{equation*}
i.e.\/ $\pa_tW=\mu\pa_x^nW$. Formally, suppose for example, we use a suitable wavelet representation for solutions to the
partial differential system above, so that we are able to represent the solution as a time evolution
of a sequence of coefficents in $\Hb=\ell^2(\CC)\times\ell^2(\CC)$,
and $P$ is a Hilbert--Schmidt operator from $\Vb\to\Vb$ with $\Vb=\ell^2(\CC)$.
Suppose near $(x_\ast,t_\ast)$, the regularised determinant $\mathrm{det}_2\bigl(\id-P(x_\ast,t_\ast)\bigr)$ is close to zero,
and exactly zero at $(x_\ast,t_\ast)$.
Recall from Section~\ref{sec:FredholmGrassmannian}, we can always find a coordinate patch representation
on the Fredholm Grassmannian for which the regularised determinant is non-zero.
This means that we are always able to nominate a different subset of rows of $W$, rather than the top block $\id-P$
which characterises the top cell/patch, for which the regularised determinant is not zero.
Suppose the suitable rows are characterised by a set, say $\Sb=\{i_1,i_2,\ldots\}\subseteq\mathbb N$,
and we have $\mathrm{det}_2(W_\Sb)\neq0$. Hence now the group-like element is $W_\Sb$, which
is a Hilbert--Schmidt perturbation of the identity, and the Hilbert--Schmidt operator is $W_{\Sb^\circ}$.
Note $W=\bigl(\id-P,P\bigr)$ still satisfies the same flow on the Stiefel manifold $\mathrm{St}(\Hb,\Vb)$
shown above, we are just choosing a different subset of rows of $W$ to base our projection onto the
Fredholm Grassmannian $\mathrm{Gr}(\Hb,\Vb)$ on. Since $\mathrm{det}_2(W_\Sb)\neq0$,
the corresponding paramaterisation $G_\Sb=W_{\Sb^\circ}\,W_{\Sb}^{-1}$ is well-defined.
The idea is to choose to represent the Korteweg--de Vries solutions at $(x_\ast,t_\ast)$ by $G_\Sb$ instead,
or more practically in a small region around $(x_\ast,t_\ast)$ which exists due to the manifold structure. 
Though the flow for $W$ on the Stiefel manifold $\mathrm{St}(\Hb,\Vb)$ is the same, in general,
$G_\Sb$ satisfies a different nonlinear evolution. This is just the result of a different projection.
See for example, Remark~\ref{rmk:scatterGrassmannian} in Appendix~\ref{sec:Evansfunction}.
Let $\Sb_0$ denote the original top cell/patch parameterisation in $\Lambda_0(\Hb,\Vb)$
so that $W_{\Sb_0}=\id-P$ and $W_{S_0^\circ}=P$.
Suppose further, for example, that at $(x_\ast,t_\ast)$ where the determinant $\mathrm{det}_2(W_{\Sb_0})$ is zero, 
that it turns out that the coordinate patch `opposite' to the top cell/patch $\Lambda_0(\Hb,\Vb)$ is a
suitable choice so $\mathrm{det}_2(W_{\Sb_0^\circ})\neq0$ and $W_{\Sb_0}$ is a Hilbert--Schmidt operator on $\Vb$.
In other words the `opposite' coordinate patch is characterised by $\Sb=\Sb_0^\circ$.
Since $\pa_tW=\mu\pa_x^nW$, we observe that $[G_\Sb]=[W_{\Sb^\circ}\,W_{\Sb}^{-1}]$ satisfies 
the nonlocal potential Korteweg--de Vries equation given in Theorem~\ref{thm:KdV},
though it is generated via a different projection of the same data. However, at $(x_\ast,t_\ast)$ the kernel
$[G_\Sb](y,z;x,t)$ is not singular, and thus $[G_\Sb](0,0;x,t)$ is well-defined.
One of our next goals is to see how this would work in practice, to track and/or incorporate singularities in solutions.
We note of course, that the corresponding argument for the nonlinear Schr\"odinger equation requires more care.
In particular, the flow in different coordinate patches to the top cell/patch $\Lambda_0(\Hb,\Vb)$ may not be that of the  
nonlinear Schr\"odinger equation; though this may not be crucial.

There are now several more general contexts in which we can now think about the integrable systems
we have considered herein, namely the noncommutative potential Korteweg--de Vries, nonlinear Schr\"odinger and
modified Korteweg--de Vries equations. For example, in the case of the noncommutative potential Korteweg--de Vries equation,
suppose we are given initial data $g(0,0;x,t)$. Then in principle we can use the scattering transform
outlined in Appendix~\ref{sec:scatteringproblem} to generate the initial scattering data $p=p(y;0)$. 
With $p=p(y+z+x,0)$ in hand we can generate $g(y,z;x,0)$ by solving the linear Fredholm equation
$P=G(\id-P)$. We can then integrate in time the nonlocal potential Korteweg--de Vries equation
for $g(y,z;x,t)$ given in Theorem~\ref{thm:KdV}, noting that the actual noncommutative potential Korteweg--de Vries
flow is given by $g(0,0;x,t)$. The advantage of integrating $g(y,z;x,t)$ is that it might be possible to
recover the scattering data $P=P(x,t)$ from $G=G(x,t)$ by solving the relation $P=G(\id-P)$ for $P$
in terms of $G$, without resorting to the scattering transform. There is some
uncertainty here as it is not clear that the scattering operator genereated in the manner suggested
would be a Hankel operator, in which case we would use the scattering transform.
The point of all this is the following. Suppose in the evolution of $g(0,0;x,t)$, a singularity occurs at $(x_\ast,t_\ast)$.
Then in the vicinity of the singularity, we can pull back the potential Korteweg--de Vries flow 
to the flow for $W$ on the Stiefel manifold and choose to project back down onto a different
coordinate patch there in order to propagate the flow locally in space-time, around the singularity.
Hence there may be some advantage to propagating $g(y,x;x,t)$ rather than just $g(0,0;x,t)$.
From another perspective, we have seen that the Korteweg--de Vries flow on $\Lambda_0(\Hb,\Vb)$
stems from the linear flow $\pa_tW=\mu\pa_x^nW$ for $W$ on the Stiefel manifold $\mathrm{St}(\Hb,\Vb)$.
Once we have computed the scattering data $p=p(y;0)$ from $g(0,0;x,0)$, we can determine the
evolution of $W$. Perhaps we should then think of the flow induced on the Fredholm Grassmannian $\mathrm{Gr}(\Hb,\Vb)$,
which incorporates all the different nonlinear flows on all the different coordinate patches,
as a generalised flow form for the noncommutative potential Korteweg--de Vries equation,
which is itself generated on the specific patch $\Lambda_0(\Hb,\Vb)$?

Lastly, as we mentioned in our introduction, there are many further applications of 
Fredholm Grassmannian flows, in particular to Smoluchowski-type coagulation equations.
In our companion paper, Doikou \textit{et al.\/} \cite{DMSW:graphflows}, we explore these applications
in detail. We also generalise our notion of Fredholm Grassmannian flows to nonlinear graph flows and
demonstrate that the inviscid Burgers equation is an example of a nonlinear graph flow.
As a result, we show that the Smoluchowski coagulation equation, in the case of either the additive
or multiplicative frequency kernels, is integrable as a nonlinear graph flow.

\begin{acknowledgement}
  We thank Chris Eilbeck and Mohammed Kbiri Alaoui for very useful discussions.
  We would also like to thank the anonymous referees for their insightful comments and suggestions
  that helped to significantly improve the original manuscript.
\end{acknowledgement}

\appendix

\section{Numerical Simulations}\label{sec:numericalsimulations}
The goal of this section is to demonstrate the use of the P\"oppe approach to
numerically simulate the scalar Korteweg--de Vries and nonlinear Schr\"odinger
equations from given initial scattering data corresponding in essence to result (II) in Lemma~\ref{lemma:KdVprescription}.
However before proceeding with that endeavour we briefly outline, in the
limited scope of reflectionless potentials, how we can numercially solve the scattering problem, and simultaneously,
how we numerically solve the linear Fredholm equation.
We outline the scattering and inverse scattering problems in detail in Appendix~\ref{sec:scatteringproblem}.
Suppose we are given the potential function
(recall from Remarks~\ref{rmk:GLM} and \ref{rmk:KdVsolutions} that $X=-x/2$):
$U(X)\coloneqq -U_0\mathrm{sech}^2(X)=-U_0\mathrm{sech}^2(x/2)=4\pa_xg(0,0;x,0)$.
In the case $U_0=6$, this is an example reflectionless potential. We also know
there are two eigenvalues, which we can find analytically. Here though, we use the
Evans function to numerically evaluate the eigenvalues, see Appendix~\ref{sec:Evansfunction}
where we outline in detail how to evaluate the transmission and reflection coefficients
for both the discrete and continuous spectra. The right panel in Figure~\ref{fig:reconstructionexample}
shows a plot of the Evans function along the negative real $\lambda$-axis where
$\lambda$ is the eigenvalue parameter; see Appendix~\ref{sec:Evansfunction} for more details.
Zeros of the Evans function coincide with eigenvalues and we observe there are two eigenvalues,
as anticipated. Using standard root-finding algorithms with the Evans function reveals the eigenvalues are indeed
$\lambda_1=-1/4$ and $\lambda_2=-1$. This means we should take the scattering data,
where from Remark~\ref{rmk:GLMandRiccati} in Appendix~\ref{sec:scatteringproblem}
we know $F(Z+X)\coloneqq-p(x+z)\equiv-p(Z+X)$, to be 
\begin{equation*}
  p(x)=\tfrac12\frac{B(\lambda_1)}{D'(\lambda_1)}\frac{\mathrm{e}^{\sqrt{-\lambda_1}x}}{\sqrt{-\lambda_1}}
       +\tfrac12\frac{B(\lambda_2)}{D'(\lambda_2)}\frac{\mathrm{e}^{\sqrt{-\lambda_2}x}}{\sqrt{-\lambda_2}}.
\end{equation*}
Note, these are the types of exponentially growing solutions we anticipated
in Result (I) in Lemma~\ref{lemma:KdVprescription}.
Now we consider whether, given this scattering data, we can solve the linear Fredholm equation to
reconstruct $g(0,0;x,0)$ or indeed $\pa_xg(0,0;x,0)$. 
Recall from our discussion preceding the solitary wave Example~\ref{ex:solitarywave} that 
in practice, for any given time $t\in[0,T]$, we can generate a solution $g=g(0,0;x,t)$
to the potential Korteweg--de Vries equation from given scattering data $p=p(x;t)$, 
by solving the linear Fredholm equation:
\begin{equation*}
p(z+x;t)=g(0,z;x,t)-\int_{-\infty}^0g(0,\xi;x,t)\hat q(\xi,z;x,t)\,\rd\xi,
\end{equation*}
for $g=g(0,z;x,t)$. Note here we set $\hat q(\xi,z;x,t)\coloneqq p(\xi+z+x;t)$.
Setting $z=0$ generates the solution $g=g(0,0;x,t)$ corresponding
to the scattering data $p=p(x;t)$. We solve this linear Fredholm integral equation
via the following crude approximation---assume $t\geqslant0$ is fixed.
We restrict the range of the integral and the domain of $z$ to $[-L/2,0]$, while $x\in[-L/2,L/2]$,
with $L>0$ sufficiently large. We discretise $z$, $\xi$ and $x$ to nodal values $z_n$, $\xi_n$ and $x_n$, each
with separation step $h$, and with $z_n=\xi_n=x_n$ for $x_n\leqslant0$. For each $x_\ell$ we generate the row vector and matrix:
$\widehat{P}_{n}(x_\ell;t)\coloneqq p(z_n+x_\ell;t)$ and $\widehat{Q}_{m,n}(x_\ell;t)\coloneqq p(\xi_m+z_n+x_\ell;t)$.
We approximate the integral in the linear Fredholm equation by a left Riemann sum and we
set $\widehat{G}=\widehat{G}(x_\ell;t)$ to be the row vector of unknown values $\widehat{G}_{m}(x_\ell;t)\coloneqq g(0,z_m;x_\ell,t)$.
Consequently we solve the linear algebraic problem (at any given $t$ and for each $x_\ell$):
\begin{equation*}
\widehat{P}=\widehat{G}-h\widehat{G}\widehat{Q}\qquad\Leftrightarrow\qquad \widehat{G}=\widehat{P}(\id-h\widehat{Q})^{-1}.
\end{equation*}
Note that in the relation on the left, the $h$ factor and matrix product realise the left Riemann sum approximation
of the integral in the linear Fredholm equation.
The relation on the right generates the row vector approximation $\widehat{G}_{m}(x_\ell;t)\coloneqq g(0,z_m;x_\ell,t)$,
which we then evaluate at $z_m=0$---in our approximations the origin is the final $z_m$ nodal point.
In the left panel in Figure~\ref{fig:reconstructionexample}, we plot the original
potential function for $\pa_xg(0,0;x,0)$, namely `$-(3/2)\mathrm{sech}^2(\frac12x)$'.
We also plot the approximate solution $g(0,0;x,0)$ generated from $\widehat{G}_{m}(x_\ell;t)$ at $z_m=0$,
as well as the derivative $\pa_x g(0,0;x,0)$ computed from $\widehat{G}_{m}(x_\ell;t)$ at $z_m=0$ via
finite differences. In the middle panel we present an exactly analogous
plot to the left panel, except that we have deliberately introduced a phase shift
in the first term involving $\lambda_1$ in the expression for $p=p(x)$ above, in order
to illustrate there are indeed two distinct underlying features present.
Lastly, as one further demonstration, suppose we assume scattering data of the form
\begin{equation*}
  p(x;t)=-2a_1\mathrm{e}^{a_1x+a_1^3t}-2a_2\mathrm{e}^{a_2x+a_2^3t},
\end{equation*}
with $a_1=1.8$ and $a_2=1.1$. Substituting this into the linear Fredholm equation
and solving it numerically as described above, at discrete times, 
we obtain the solution to the Korteweg--de Vries equation shown in Figure~\ref{fig:twosoliton},
i.e.\/ the two-soliton collision.

\begin{figure}
  \begin{center}
    \includegraphics[width=4cm,height=3cm]{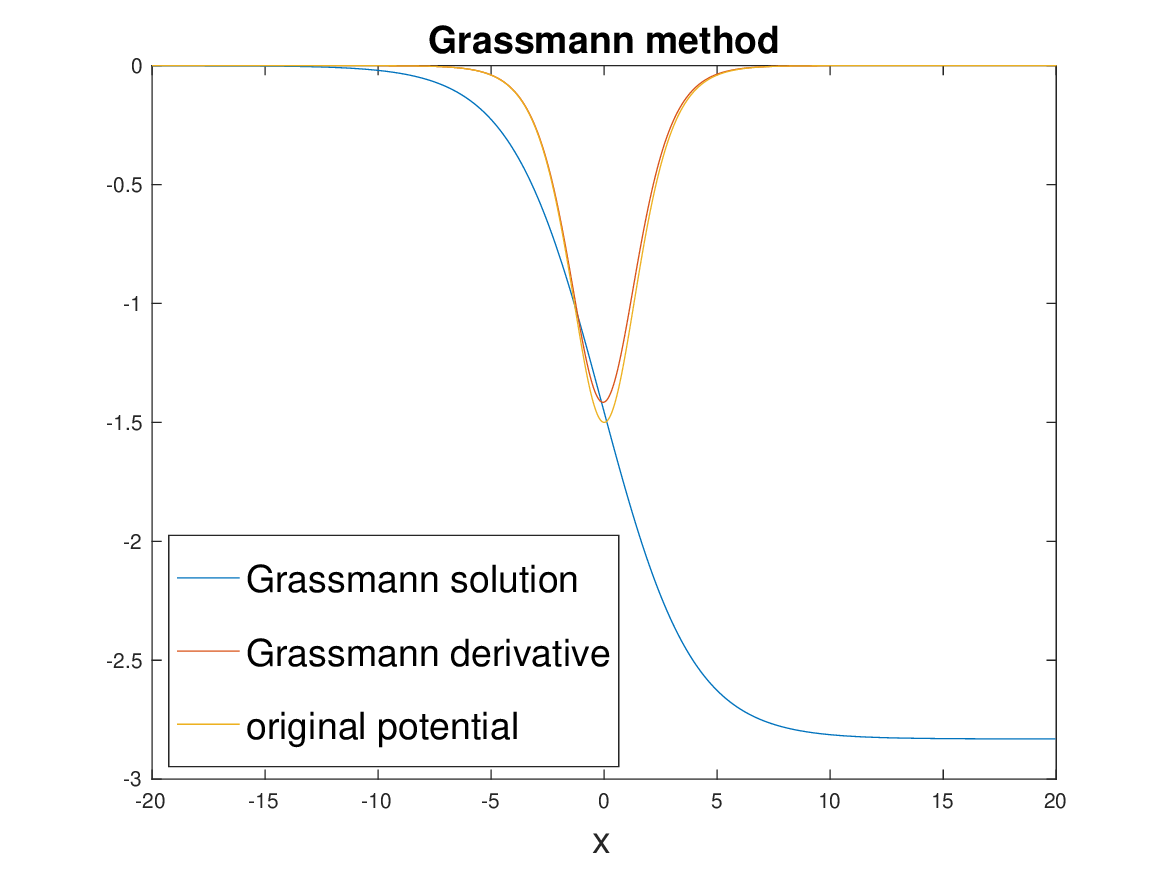}
    \includegraphics[width=4cm,height=3cm]{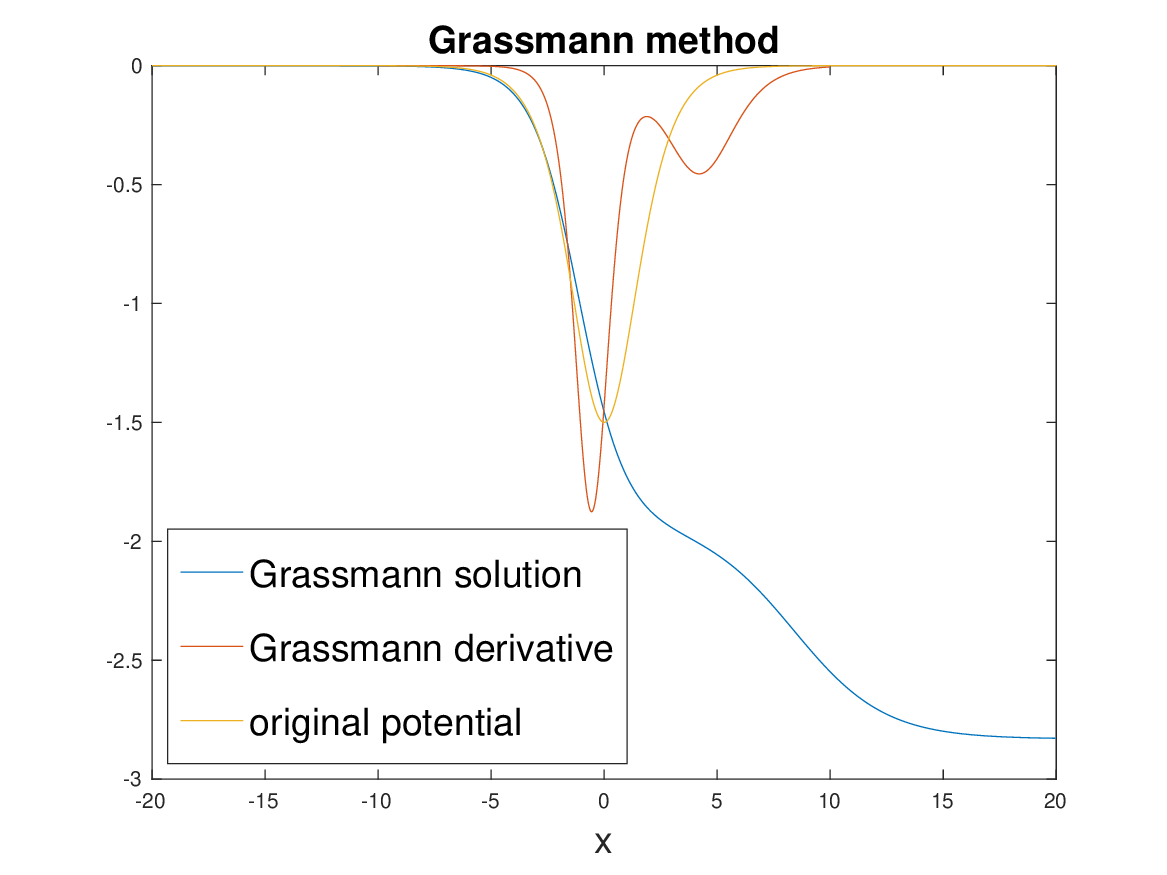}
    \includegraphics[width=4cm,height=3cm]{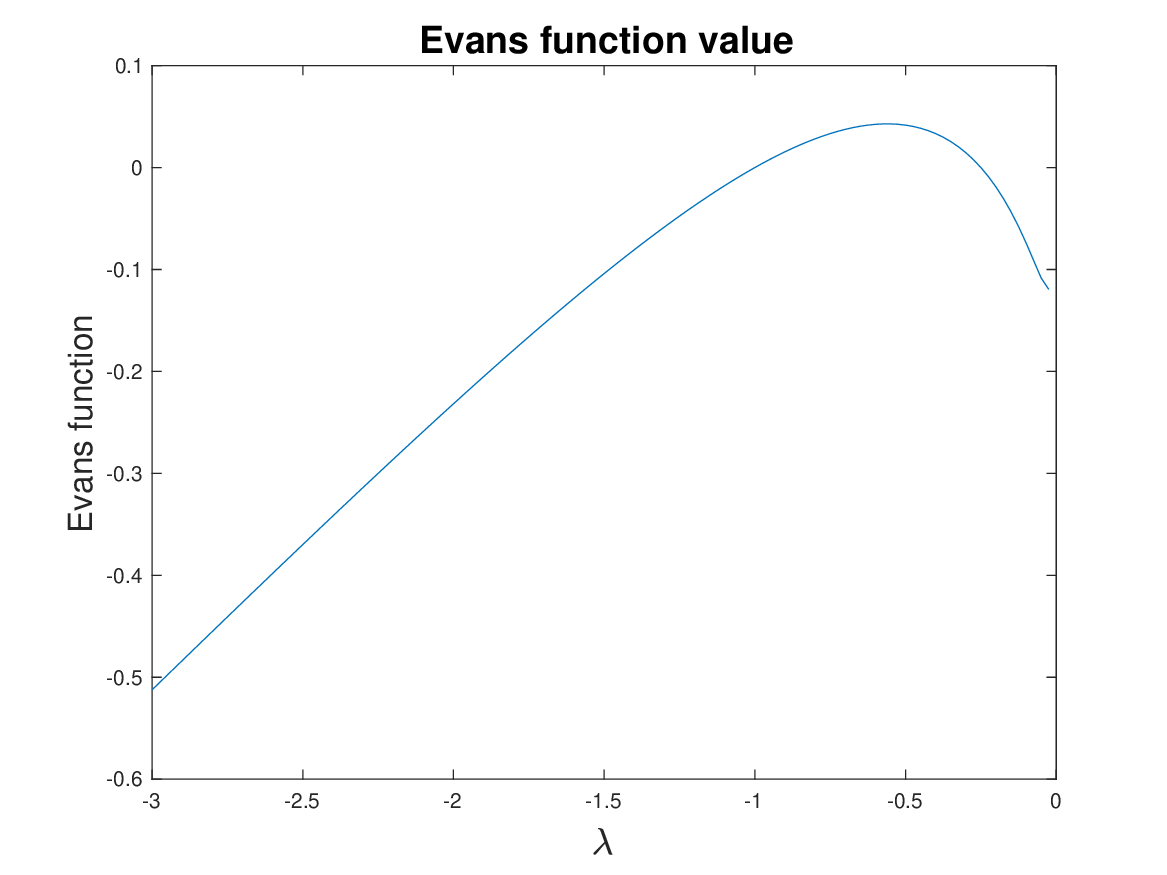}
  \end{center}
  \caption{We plot the results from the numerical computation of the scattering and inverse
    scattering problems for an original reflectionless potential function $\pa_xg(0,0;x,0)=(3/2)\mathrm{sech}^2(x/2)$.
    We used the Evans function, as shown in the right panel, to determine the two discrete eigenvalues.
    They correspond to the zeros of the Evans function. We then constructed the corresponding
    scattering data function $p=p(x;0)$. We substituted this into the linear Fredholm equation
    which we solved numerically, to re-generate $g(0,0;x,0)$ and then $\pa_xg(0,0;x,0)$.
    This is shown in the left panel. The middle panel is analogous to the left
    panel except that we have introduced a phase shift in one of the two 
    terms in the scattering data function $p=p(x;0)$ in order to illustrate the 
    two distinct underlying features present.}
\label{fig:reconstructionexample}
\end{figure}

\begin{figure}
  \begin{center}
  \includegraphics[width=6cm,height=5cm]{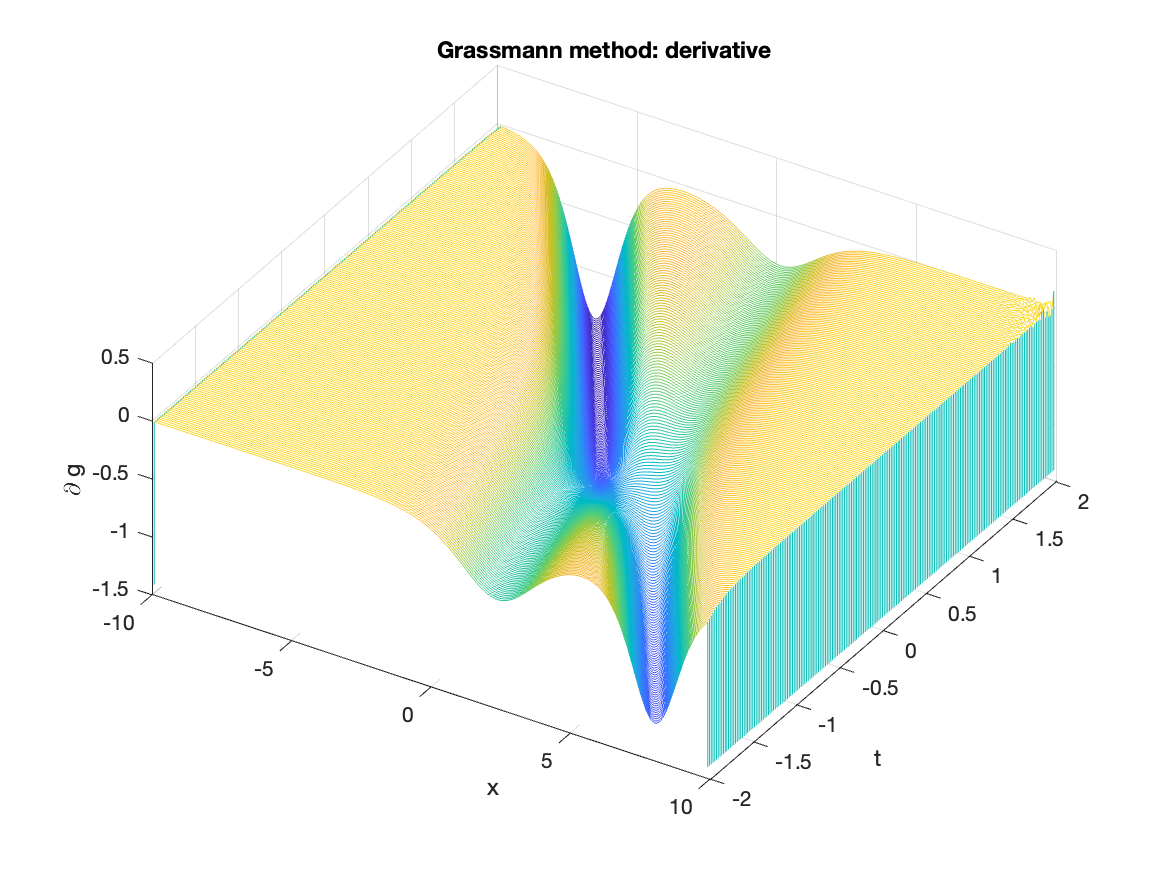}
  \end{center}
  \caption{We plot the solution to the Korteweg--de Vries equation which
    demonstrates a two-soliton collision. This was achieved by assuming
    a suitable scattering data function form $p=p(x;t)$ with two independent
    exponential terms, and then numerically solving the linear Fredholm equation
    to generate $g(0,0;x,t)$ and then $\pa_xg(0,0;x,t)$.}
\label{fig:twosoliton}
\end{figure}

We now consider numerical simulations for solutions to the scalar Korteweg--de Vries 
and nonlinear Schr\"odinger equations which are generated from given arbitrary initial
scattering data $p_0=p_0(x)$ which decays in both far fields,
i.e.\/ corresponding to result (II) in Lemma~\ref{lemma:KdVprescription}.
For each nonlinear equation we provide two independent simulations:
(i) Direct numerical simulation using well-known spectral algorithms,
advancing the approximate solution $u_m$ in successive time steps; 
(ii) Generating of an approximate solution $g_{\mathrm{approx}}(0,0;x,t)$
using the Grassmannian flow approach by advancing the scattering data in time to $p=p(x;t)$,
and then numerically solving the linear Fredholm equation.
Let us first explain the direct numerical simulation approach we employed.
We chose an initial profile function $p_0=p_0(x)$ on the interval $[-L/2,L/2]$
for some fixed domain length $L>0$. In the Korteweg--de Vries case, we then numerically
solved the linear Fredholm integral equation as described just above, in order to generate
$g_{\mathrm{approx}}=g_{\mathrm{approx}}(0,0;x,0)$. We then set $u_0(x)=g_{\mathrm{approx}}(0,0;x,0)$. 
We chose the initial profile $p_0(x)=\exp(-\frac12x^2)\mathrm{tanh}(x/5)$
and $L=20$ in this case. The Fourier spectral method we used to advance the Korteweg--de Vries
solution in time is the split-step method,
\begin{equation*}
v_j=\exp(\Delta tK^3)u_j
\quad\text{and}\quad    
u_{j+1}=v_j+3\Delta t\mathcal F
\bigl(\mathcal F^{-1}(v_j)\mathcal F^{-1}(Kv_j)\bigr)
\end{equation*}
where $\mathcal F$ denotes the Fourier transform and
$K$ is the diagonal matrix of Fourier coefficients $2\pi\mathrm{i}k$.
In practice we use the fast Fourier transform. 
We chose $\Delta t=1.5\times10^{-3}$ and the number of Fourier modes is $2^9$.
The result is shown in the left panel in Figure~\ref{fig:KdV}.
In the Grassmannian flow approach, given the the initial profile 
$p_0=p_0(x)$ we advance the solution analytically in Fourier space
to $p(x,t)=\mathcal F^{-1}(\exp(tK^3)\mathcal F(p_0))$
at any given time $t\in[0,T]$, where $T=1.5$.
We solve the Fredholm integral equation numerically, as described above,
for $g_{\mathrm{approx}}=g_{\mathrm{approx}}(0,z;x,t)$. The solution to the Korteweg--de Vries equation  
at the given time $t$ is then computed as $g_{\mathrm{approx}}(0,0;x,t)$.
The result is shown in the right panel in Figure~\ref{fig:KdV}.

\begin{figure}
  \begin{center}
  \includegraphics[width=6cm,height=5cm]{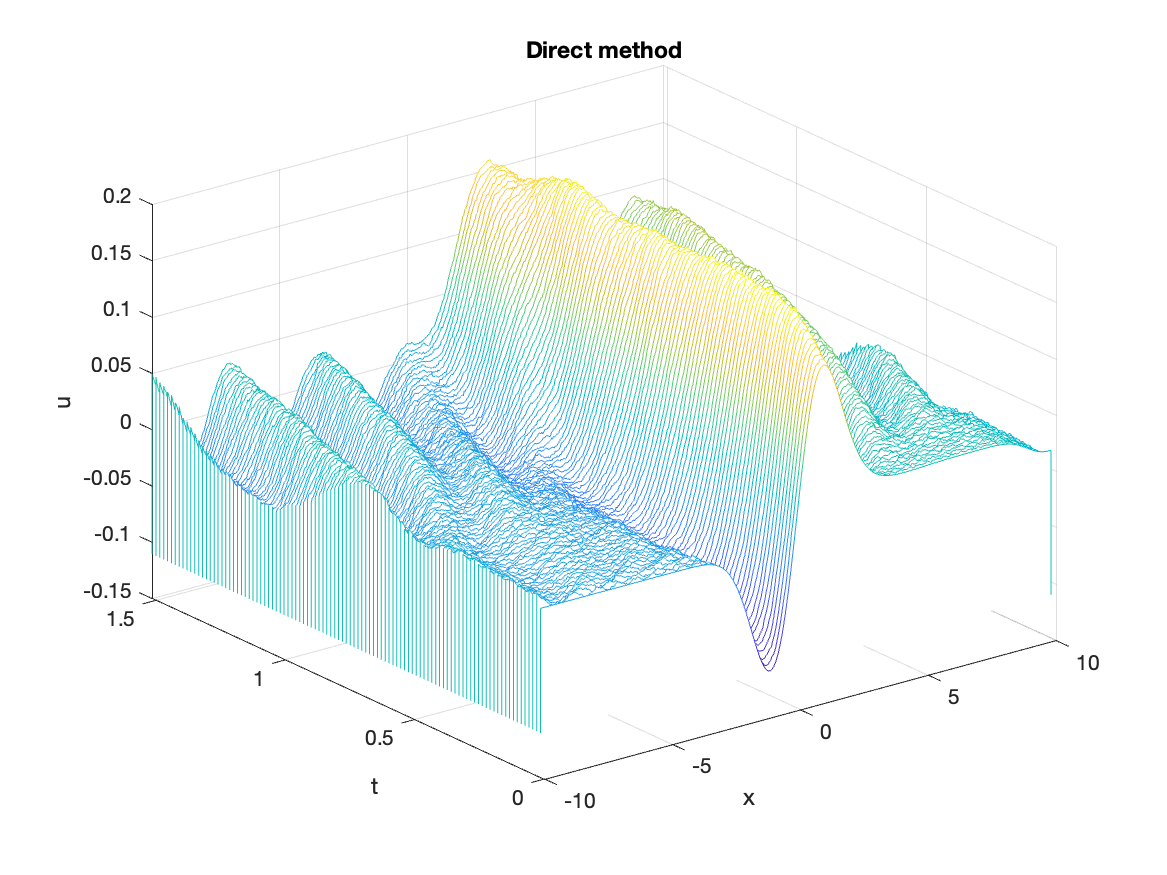}
  \includegraphics[width=6cm,height=5cm]{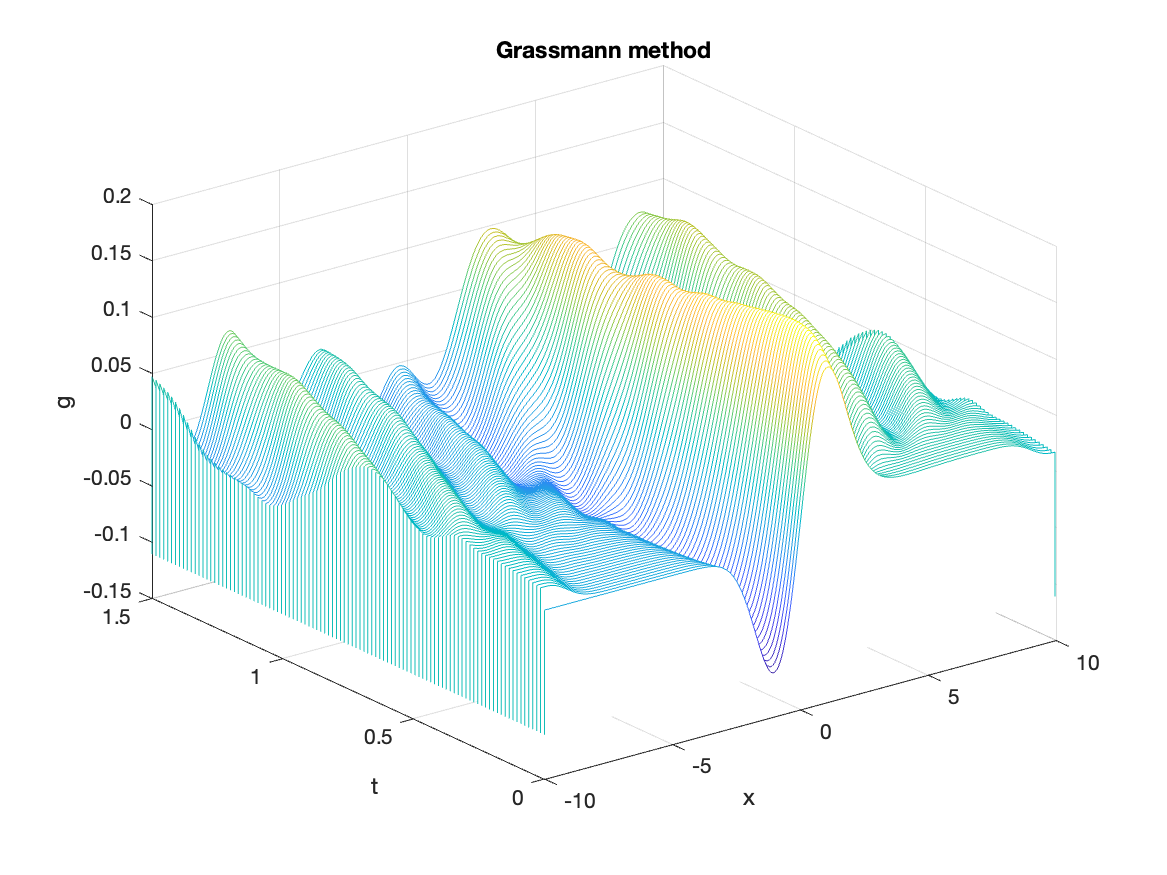}
  \end{center}
  \caption{We plot the solution to the Korteweg--de Vries equation corresponding
    to the initial scattering data $p_0(x)=\exp(-\frac12x^2)\mathrm{tanh}(x/5)$.
    The left panel shows the evolution of the solution computed 
    using a direct integration approach, while the right panel shows the corresponding solution 
    computed using the Grassmannian flow approach. We also monitored the evolution of the Fredholm Determinant of $Q=Q(x,t)$
    for all $x\in[-L/2,L/2]$ and it remained far from zero throughout.}
\label{fig:KdV}
\end{figure}

In the case of the nonlinear Schr\"odinger equation the procedure
is very similar to that for the Korteweg--de Vries equation.
Given an initial scattering profile $p_0(x)=\frac12\cosh(x/40)$,
we restrict the interval domain to $[-L/2,L/2]$ as before.
We chose $L=40$ in this case. We then numerically solve the linear Fredholm equation in order to 
generate $g_{\mathrm{approx}}=g_{\mathrm{approx}}(0,0;x,0)$.
We use the same approximation method to solve the linear Fredholm integral equation
as described above, except that we define the matrix $\widehat{Q}_{m,n}(x_\ell;t)$
differently as follows, to account for the fact that is this case, the operator $\hat Q\coloneqq P^\dag P$.
Since we have,
\begin{equation*}
\hat q(\xi,z;x,t)\coloneqq\int_{-\infty}^0p^\ast(\xi+\zeta+x;t)p(\zeta+z+x;t)\,\rd\zeta,
\end{equation*}
we set $\tilde P_{m,n}(x_\ell;t)\coloneqq p(\zeta_m+z_n+x_\ell;t)$, and then, at any given $t\geqslant0$ and
for each $x_\ell$, we set $\widehat{Q}=\widehat{Q}_{m,n}(x_\ell;t)$ to be $\widehat{Q}\coloneqq h \tilde P^\dag \tilde P$.
Note the $\tilde P^\dag$ is the complex-conjugate transpose of the matrix $\tilde P$, and
implicitly, the factor $h$ and matrix product between $\tilde P^\dag$ and $\tilde P$ realise
a left Riemann sum approximation of the integral in the definition of $\hat q$. 
We then solve $\widehat{P}=\widehat{G}-h\widehat{G}\widehat{Q}$ as previously.
Thus, initially at $t=0$, we find $g_{\mathrm{approx}}=g_{\mathrm{approx}}(0,0;x,0)$ by using this
prodcedure to generate $\widehat{G}_{m}(x_\ell;0)$, which we then evaluate at $z_m=0$.
We then set $u_0(x)=g_{\mathrm{approx}}(0,0;x,0)$. 
The Fourier spectral method we used to advance the nonlinear Schr\"odinger solution forward in time is the 
split-step method:
\begin{equation*}
v_j=\exp(-\mathrm{i}\Delta tK^2)u_j
\quad\text{and}\quad    
u_{j+1}=v_j-2\mathrm{i}\Delta t\mathcal F
\bigl(\bigl(\mathcal F^{-1}(v_j)\bigr)^2\bigl(\mathcal F^{-1}(v_j)^\ast\bigr)\bigr).
\end{equation*}
We chose $\Delta t=0.01$ and the number of Fourier modes is $2^8$. The result is shown in Figure~\ref{fig:NLS}.
In the Grassmannian flow approach, given the initial profile $p_0=p_0(x)$ we advance the solution analytically
to $p(x,t)=\mathcal F^{-1}(\exp(-\mathrm{i}tK^2)\mathcal F(p_0))$ for any given time $t\in[0,T]$, where $T=100$. 
We then use the procedure just described to generate the row vector $\widehat{G}_{m}(x_\ell;t)$ which, again,
we evaluate at $z_m=0$. The results are shown in Figure~\ref{fig:NLS}.

\begin{figure}
  \begin{center}
  \includegraphics[width=6cm,height=5cm]{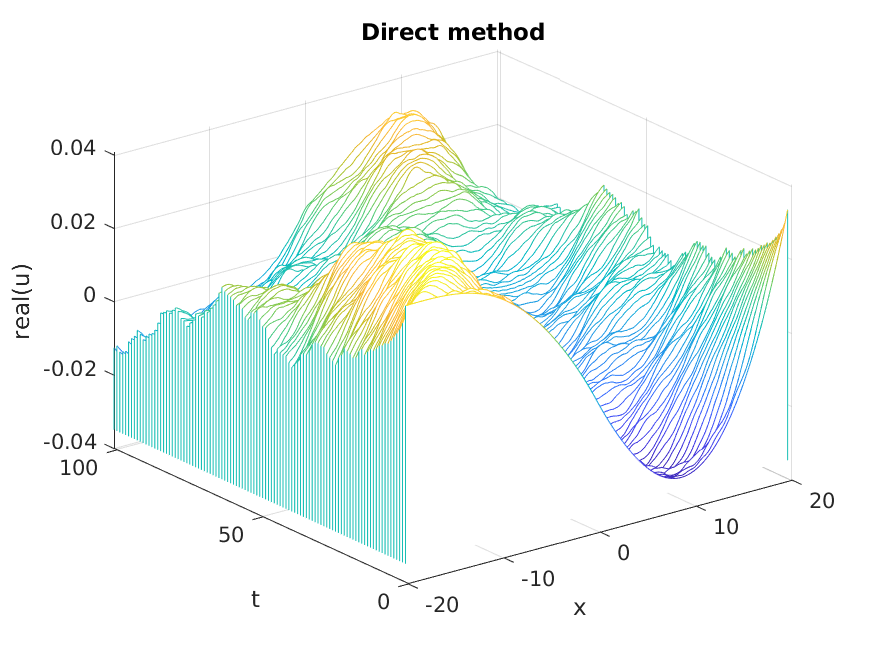}
  \includegraphics[width=6cm,height=5cm]{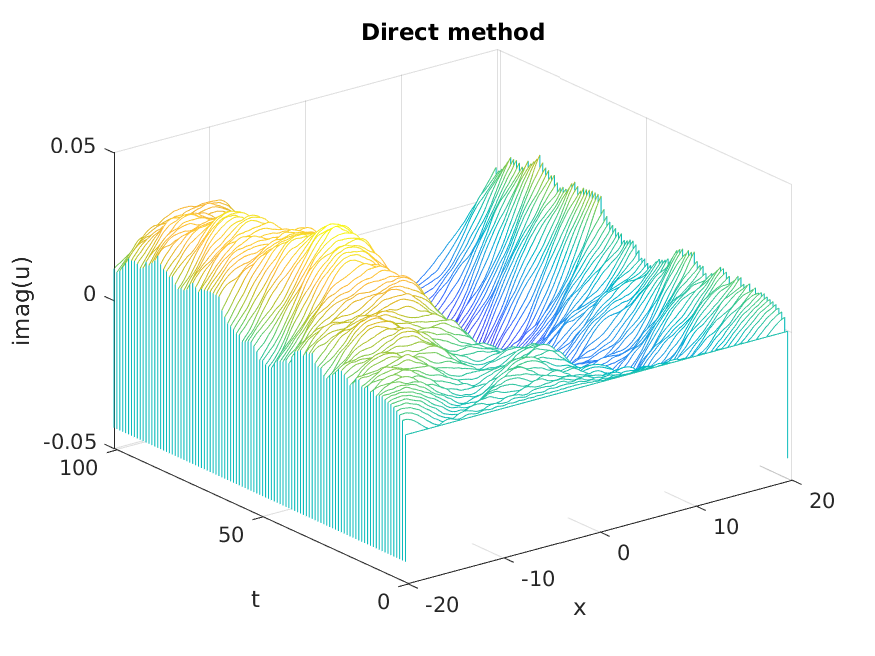}\\
  \includegraphics[width=6cm,height=5cm]{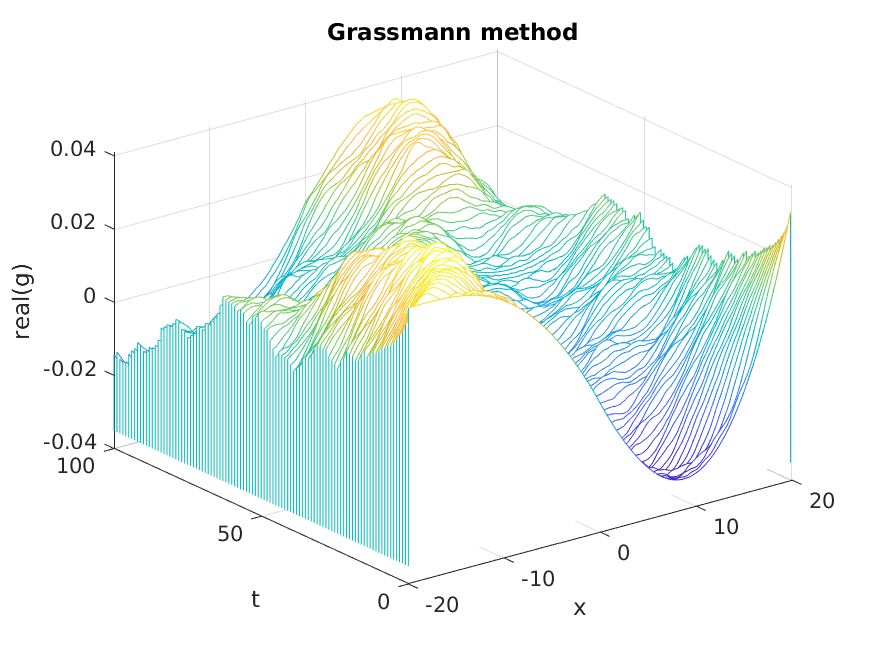}
  \includegraphics[width=6cm,height=5cm]{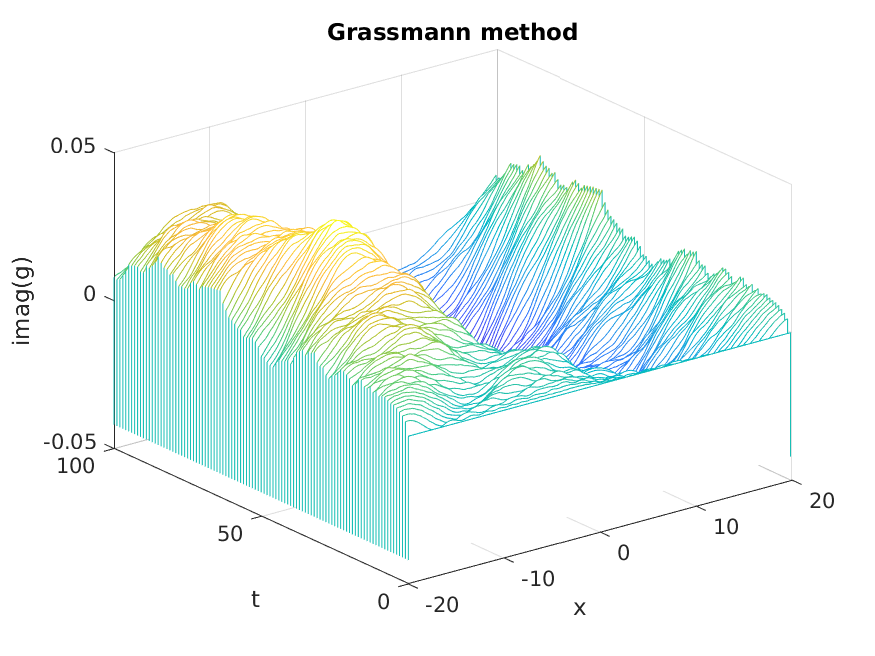}
  \end{center}
  \caption{We plot the solution to the nonlinear Sch\"odinger equation corresponding
    to the initial scattering data $p_0(x)=\frac12\cosh(x/40)$. 
    The top panels show the real and imaginary parts computed using a direct integration 
    approach, while the bottom panels show the corresponding real and imaginary parts 
    computed using the Grassmannian flow approach, of the solution evolution.
    We also monitored the evolution of the Fredholm Determinant of $Q=Q(x,t)$
    for all $x\in[-L/2,L/2]$ and it remained far from zero throughout.}
\label{fig:NLS}
\end{figure}

\begin{remark}[Periodic boundary conditions]
  For the numerical simulations generating Figures~\ref{fig:KdV} and \ref{fig:NLS},
  which involved initial scattering data that decayed in both far fields,
  we truncated the $x$-domain from $\R$ to $[-L/2,L/2]$ for sufficiently large $L>0$.
  We then employed Fourier spectral methods for the direct integration methods
  as well as to advance the scattering data in time in the Grassmannian flow approach.
  This worked well as we effectively had zero, periodic boundary conditions
  at $x=\pm L/2$. This would obviously not be the case for exponentially growing
  scattering data corresponding to soliton solutions. However, that the
  Grassmannian flow approach carries through for periodic boundary conditions
  warrants further analytical investigation. 
\end{remark}

\begin{remark}[Numerical inverse scattering]
  Numerical techniques based on the inverse scattering transform have appeared recently. 
  See, for example, Trogdon \textit{et al.\/ } \cite{TOD} and Trogdon and Olver \cite{TO}
  for very sophisticated and successful methods based on this approach.
\end{remark}

\section{Scattering and inverse scattering problems}\label{sec:scatteringproblem}
For completeness, we consider the scattering problem associated with potential
functions $U\colon\R\to\CC^{m\times m}$, for some $m\in\mathbb N$. Of course in our applications   
the function $U$ is the initial data associated with the non-commutative Korteweg--de Vries equation 
so for $x\in\R$:
\begin{equation*}
U(x)\coloneqq g(0,0;x,0).
\end{equation*}
The scattering problem concerns determining the far-field ``scatter'' information
associated with the eigenvalue problem for the potential function $U$ given by:
\begin{equation*}
-\pa_x^2 Q+UQ=\lambda Q,
\end{equation*}
where $\lambda$ is the spectral parameter. Note, $Q\colon\R\to\CC^m$ is the corresponding eigenfunction $Q=Q(x,\lambda)$. 
We assume $U\to O_m$ as $|x|\to\infty$, where $O_m$ is the $m\times m$ zero matrix.
We first determine the possible far field behaviour of the spectral problem. 
We set $U=O$ and denote the corresponding constant coefficient eigenvalue problem
as the \emph{spatial} or \emph{free eigenvalue problem}. 
We look for solutions to the spatial eigenvalue problem of the form $Q=Q_\infty\mathrm{e}^{\mu x}$,
where $Q_\infty=Q_\infty(\lambda)\in\CC^m$ and we refer to the exponential growth rates $\mu=\mu(\lambda)\in\CC$
as the \emph{spatial eigenvalues}. We observe that such solutions exist to the spatial eigenvalue
problem if and only if $\mu^2+\lambda=0$, repeated $m$ times.
We observe the continuous spectrum, which is completely
determined by the far-field spatial eigenvalue problem, coincides with the positive $\lambda$ axis. 
The spatial eigenvalues are given by $\mu=\pm\sqrt{-\lambda}$, each of multiplicity $m$.
For convenience we set $\mu_\pm=\mu_\pm(\lambda)\in\CC$ to be $\mu_\pm\coloneqq\pm\sqrt{-\lambda}$.
The eigenspaces associated with each spatial eigenvalue $\mu_\pm$ have dimension $m$ and each are
spanned by $\{e_i\colon i=1,\ldots,m\}$, where the $e_i\in\R^m$ are the vectors all of whose entries
are zero except for the $i$th entry which is $1$. Hence for each spatial eigenvalue $\mu_\pm$ there are $n$
corresponding eigenvectors of the same exponential growth rates $\mu_\pm$. We collect the eigenvectors
together as as columns and thus the far-field behaviour of the eigenvalue problem is summarised via
the $C^{m\times m}$-valued solutions:
\begin{equation*}
  Q_\infty^\pm\coloneqq\mathrm{e}^{\mu_\pm x}I_m,
\end{equation*}
where $I_m$ is the $m\times m$ identity matrix. The classical scalar case corresponds to $m=1$,
and the classical development of the solution to the scattering problem can be found for example
in Drazin and Johnson~\cite{DJ} or Keener~\cite{Keener}. The case $m\geqslant1$ we consider here
is a direct consequence of the standard scalar case. We define the function
$\varphi\colon\R^2\to\CC^{m\times m}$ by the following relation,
\begin{equation*}
  Q_{-\infty}(x,\lambda)=\mathrm{e}^{\mu x}\,I_m+\int_{-\infty}^x\varphi(z;x)\mathrm{e}^{\mu z}\,\rd z,
\end{equation*}
for any $\mu$. Note that $\varphi$ does not depend on $\lambda$. In the relation,
$Q_{-\infty}=Q_{-\infty}(x,\lambda)$ is the solution to the eigenvalue problem above 
satisfying $Q_{-\infty}\sim\mathrm{e}^{\mu x}\,I_m$ as $x\to-\infty$.
While we know $Q_{-\infty}$ exists, we need to establish that $\varphi$ does.
Substituting the relation into the eigenvalue problem, reveals, after integrating
by parts twice, that $\varphi=\varphi(z;x)$ satisfies the \emph{Goursat problem},
\begin{equation*}
  \pa_z^2\varphi-\pa_x^2\varphi=U\varphi\qquad\text{and}\qquad 2\frac{\rd}{\rd x}\varphi(x;x)=U(x),
\end{equation*}
where $U=U(x)$ is the given potential function, and we have assumed that 
$\varphi(z;x)\mathrm{e}^{\mu x}\to0$ and $\pa_z\varphi(z;x)\mathrm{e}^{\mu x}\to0$ as $z\to-\infty$.
A solution to this problem exists, see for example, Drazin and Johnson~\cite[p.~50]{DJ} or Keener~\cite[p.~412]{Keener}. 
Following the standard theory we consider an incident wave $Q_{+\infty}^-(x,\lambda)$ from $x\to+\infty$,
and examine what proportions of the incident wave are, due to the presence of the potential $U$, 
transmitted into the opposite far field $x\to-\infty$, or reflected. 
We denote by $Q^\pm_{-\infty}=Q^\pm_{-\infty}(x,\lambda)$ as the respective solutions
satisfying the respective asymptotic limits
$Q^\pm_{-\infty}\sim\mathrm{e}^{\mu_{\pm} x}\,I_m$ as $x\to-\infty$.
That we interpret the asymptotic forms $\mathrm{e}^{\mu_\pm x}\,I_m$ as incident,
transmitted or reflected waves in the respective far-field limits
might more readily be envisaged if we express $\mu_\pm(\lambda)=\pm\mathrm{i}k$ for some $k\in\CC$; note that $\lambda=k^2$.
For each $x\in\R$ and the same value of $\lambda$, the incident wave $Q_{+\infty}^-(x,\lambda)$
must lie in the span of $Q^-_{-\infty}(x,\lambda)$ and $Q^+_{-\infty}(x,\lambda)$.
This fact amounts to the condition that there must exist constant $\CC^{m\times m}$ matrix-valued coefficients
$D=D(\lambda)$ and $B=B(\lambda)$, such that for every $x\in\R$:
\begin{equation*}
  Q_{+\infty}^-(x,\lambda)=Q^-_{-\infty}(x,\lambda)\,D(\lambda)+Q_{-\infty}^+(x,\lambda)\,B(\lambda).
\end{equation*}
The solution fields $Q_{+\infty}^-$ and $Q^\pm_{-\infty}$ are generated by integrating the eigenvalue
problem from the respective far-field limits to some $x\in\R$. The vector fields concerned depend
linearly on $\lambda$ while the data depend on $\mu_\pm=\pm\sqrt{-\lambda}$.
We must introduce a branch cut, here along the positive $\lambda$-axis, due to the dependence on $\sqrt{-\lambda}$.
The coefficients $D$ and $B$ are constructed from such solution fields. We show how this is achieved in practice via the Evans function
in Appendix~\ref{sec:Evansfunction}. Both coefficients thus depend similarly on $\lambda$ and $\sqrt{-\lambda}$.
However we set $\lambda=k^2$ with $k\in\CC$, thus unravelling the two Riemann surfaces associated with $\sqrt{-\lambda}$,
mapping them both to the upper and lower halves of the $k$-complex plane. Consequently the maps $k\mapsto D(k^2)$
and $k\mapsto B(k^2)$ are analytic.
We now subsitute the corresponding expressions for $Q_{-\infty}^\pm=Q_{-\infty}^\pm(x;\lambda)$
in terms of $\varphi$ from the relation above using $\lambda=k^2$.
Assuming $\varphi(z;x)=0$ for $z>x$, that $\mu_\pm=\pm\mathrm{i}k$ and rearranging,
this compatability condition involving $D$ and $B$ generates the following relation:
\begin{align*}
  \int_{-\infty}^{+\infty}\varphi(z;x)\mathrm{e}^{-\mathrm{i}kz}\,\rd z=&\;Q_{+\infty}^-(x,k^2)D^{-1}(\lambda)-\mathrm{e}^{-\mathrm{i}kx}I_m
  -\mathrm{e}^{\mathrm{i}kx}B(k^2)D^{-1}(k^2)\\
  &\;-\int_{-\infty}^{+\infty}\varphi(z;x)\mathrm{e}^{\mathrm{i}kz}\,\rd z\,B(k^2)D^{-1}(k^2).
\end{align*}
While $D(k^2)$ and thus $\det\bigl(D(k^2)\bigr)$ is analytic in $k\in\CC$, the determinant $\det\bigl(D(k^2)\bigr)$
may have zeros. These zeros coincide precisely with discrete eigenvalues and thus bound states of the underlying
eigenvalue problem---see Appendix~\ref{sec:Evansfunction} for a discussion on this.
Suppose we define the function $\hat D=\hat D(\lambda)$ by the relation $D(\lambda)=I_m+\hat D(\lambda)$.
Let $\nu\colon \CC^{m\times m}\to\CC^{m\times m}$ denote the function $\nu(\hat D)\coloneqq-\hat D(I_m+\hat D)^{-1}\det(I_m+\hat D)$.
Then $\nu\colon \CC^{m\times m}\to\CC^{m\times m}$ is an analytic function; see Simon~\cite[p.~46]{Simon:Traces}.
Hence we observe that $D^{-1}(\lambda)=I_m+\nu\bigl((\hat D(\lambda)\bigr)/\det\bigl(I_m+\hat D(\lambda)\bigr)$,
and so the singularites of $D^{-1}=D^{-1}(\lambda)$ correspond to the zeros of $\det\bigl(D(\lambda)\bigr)$.
Furthermore, if we assume the potential function $U=U(x)$ is bounded then we know the discrete spectrum
remains in a bounded region of the complex $\lambda$-plane. To see this we observe, by direct computation
from the eigenvalue problem, we have
\begin{equation*}
|\lambda|\,\|Q\|_{L^2(\R;\CC^{m\times m})}^2
\leqslant\|\pa Q\|_{L^2(\R;\CC^{m\times m})}^2+\|U\|_{L^\infty(\R;\CC^{m\times m})}\|Q\|_{L^2(\R;\CC^{m\times m})}^2.  
\end{equation*}
Hence we observe that if we restrict ourselves to $H^1(\R;\CC^{m\times m})$-valued eigenfunctions,
which we do, then any discrete eigenvalues $\lambda$ are necessarily bounded in magnitude.

Now consider the contour in the complex $k$-plane given by $\mathcal C\coloneqq\{k\in\CC\colon k=\kappa-\mathrm{i}\gamma,~\kappa\in\R\}$
for some fixed real constant $\gamma>0$ sufficiently large so that all the poles of $D^{-1}(\lambda)$ lie above $\mathcal C$.
Define the function $F\colon\R\to\CC^{m\times m}$ by the following contour integral,
which is well defined by our choice of the contour $\mathcal C$:
\begin{equation*}
F(Z)\coloneqq\frac{1}{2\pi}\int_{\mathcal C} B(k^2)D^{-1}(k^2)\mathrm{e}^{\mathrm{i}kZ}\,\rd k.
\end{equation*}
Multiplying both sides of the integral relation just above for $\varphi=\varphi(z;x)$ by $\mathrm{e}^{\mathrm{i}k\xi}$,
integrating along the contour $\mathcal C$ and dividing through by $2\pi$, we find that $\varphi$ must satisfy:
\begin{equation*}
  \varphi(\xi;x)=\frac{1}{2\pi}\int_{\mathcal C} Q_{+\infty}^-(x,k^2)D^{-1}(k^2)\mathrm{e}^{\mathrm{i}k\xi}\,\rd k-F(\xi+x)
  -\int_{-\infty}^{+\infty}\varphi(z;x)F(z+\xi)\,\rd z.
\end{equation*}
Note by standard anaylsis the integral of $\mathrm{e}^{\mathrm{i}k(\xi-x)}$ along the contour $\mathcal C$ is zero.
Indeed the first contour integral on the right above is also zero, as follows. Approximate the
contour integral by truncating the contour $\mathcal C$ so that $\kappa\in[-R,R]$ for some large $R>0$.
Call the truncated contour $\mathcal C_R$. Construct a closed contour by appending a semi-circular arc $\mathcal C_{\mathrm{arc}}$
of radius $R$ to $\mathcal C_R$; so the ends of $\mathcal C_{\mathrm{arc}}$ coincide with the ends of $\mathcal C_R$, and
the lowest point of $\mathcal C_{\mathrm{arc}}$ lies at $k=-\gamma-R$. By Cauchy's Theorem, since the argument is
analytic inside $\mathcal C_R+\mathcal C_{\mathrm{arc}}$, the integral is zero. We then take the limit $R\to\infty$.
Jordan's Lemma implies the integral along $\mathcal C_{\mathrm{arc}}$ is zero provided $\xi<x$; see Remark~\ref{rmk:Jordanslemma}
below. Hence the integral along $\mathcal C$ is zero. Consequently we see that $\varphi$ must satisfy:
\begin{equation*}
  F(\xi+x)+\varphi(\xi;x)+\int_{-\infty}^{x}\varphi(z;x)F(z+\xi)\,\rd z=0.
\end{equation*}
Let us now examine how to evaluate the scattering data $F=F(Z)$. As we have just seen, again we
approximate the contour integral by truncating the contour $\mathcal C$ so that $\kappa\in[-R,R]$ for some large $R>0$,
and call the truncated contour $\mathcal C_R$. Then we construct a closed contour consisting of the
interval $[+R,-R]$ on the real $k$-axis, noting the direction of integration indicated, the contour $\mathcal C_R$,
and the two straight-line contours $\mathcal C_1$ and $\mathcal C_2$, respectively from $k=R-\mathrm{i}\gamma$ to $k=R$,
and from $k=-R$ to $k=-R-\mathrm{i}\gamma$. By residue calculus, the integral around the closed contour
$\mathcal C_R\cup\mathcal C_1\cup[+R,-R]\cup\mathcal C_2$ is given by $2\pi\mathrm{i}$ times the sum
of the residues of $B(k^2)D^{-1}(k^2)\mathrm{e}^{\mathrm{i}kZ}$ at the zeros $k=k_n$ of $\det\bigl(D(k^2)\bigr)=0$. 
In the large $R$ limit the contour integrals along $\mathcal C_1$ and $\mathcal C_2$ vanish and we are left with
\begin{equation*}
  F(Z)=\frac{1}{2\pi}\int_{-\infty}^{+\infty} B(k^2)D^{-1}(k^2)\mathrm{e}^{\mathrm{i}kZ}\,\rd k
  +\mathrm{i}\sum_{n}\mathrm{res}\bigl(B(k^2)D^{-1}(k^2)\mathrm{e}^{\mathrm{i}kZ}\colon k=k_n\bigr). 
\end{equation*}
The zeros $\det\bigl(D(k^2)\bigr)$ of concern here, within the closed contour, occur
at $k_n=-\mathrm{i}\sqrt{|\lambda_n|}$. By Taylor series expansion about any zero $k=k_n$ we observe,
$D(k^2)=\bigl(D^\prime(k_n^2)\,(2k_n)\bigr)(k-k_n)+\mathcal O\bigl(|k-k_n|^2\bigr)$ and so 
\begin{equation*}
  \mathrm{res}\bigl(B(k^2)D^{-1}(k^2)\mathrm{e}^{\mathrm{i}kZ}\colon k=k_n\bigr)
  =-\tfrac12\frac{B(\lambda_n)}{D^\prime(\lambda_n)}\frac{\mathrm{e}^{\sqrt{|\lambda_n|}Z}}{\sqrt{|\lambda_n|}}.
\end{equation*}
Hence the scattering data $F=F(Z)$ is given by
\begin{equation*}
  F(Z)=\frac{1}{2\pi}\int_{-\infty}^{+\infty} B(k^2)D^{-1}(k^2)\mathrm{e}^{\mathrm{i}kZ}\,\rd k
  -\tfrac12\sum_{n}\frac{B(\lambda_n)}{D^\prime(\lambda_n)}\frac{\mathrm{e}^{\sqrt{|\lambda_n|}Z}}{\sqrt{|\lambda_n|}}.
\end{equation*}
Note in particular that the contribution to the scattering data $F=F(Z)$ from the discrete spectra $\lambda=\lambda_n$
grows exponentially in the limit $Z\to+\infty$. 
In summary therefore, we see that $\varphi=\varphi(\xi;x)$ necessarily satisfies the linear
Volterra integral equation involving the scattering data $F=F(\xi+x)$ shown, with $F=F(Z)$ given  
by the formula just above. The linear Volterra integral equation is the \emph{Gel'fand--Levitin--Marchenko equation}.
If we combine this knowledge with the observation that we can resconstruct
the original potential function $U$ from $\varphi(x,x)$, via the boundary relation in the
Goursat problem above, we have thus also completed the formulation of the \emph{inverse scattering problem}.
In other words, given the data implicit in the function $F=F(Z)$, the eigenvalues and transmission and
reflection coefficients, we can then, in principle, solve the Gel'fand--Levitin--Marchenko equation for $\varphi$,
and evaluate $U$.
\begin{remark}[Relation to the linear Fredholm kernel equation]\label{rmk:GLMandRiccati}
The relation between the version of the Gel'fand--Levitin--Marchenko equation
for $\varphi=\varphi(Z;X)$ above, given by,
\begin{equation*}
  F(Z+X)+\varphi(Z;X)+\int_{-\infty}^{X}\varphi(Y;X)F(Y+Z)\,\rd Y=0,
\end{equation*}
and the linear Fredholm equation for the kernel $\hat g(z;x)=g(0,z;x)$ given by
(suppressing any time-dependence for the moment),
\begin{equation*}
  p(z+x)=\hat g(z;x)-\int_{-\infty}^0\hat g(\eta;x)p(\eta+z+x)\,\rd\eta,
\end{equation*}  
is given as follows. Making the change of variables $z\coloneqq Z-\frac12x$, $\eta\coloneqq Y-\frac12x$
and then $x\coloneqq 2X$ in the linear Fredholm kernel equation, we arrive at the relation:
\begin{equation*}
  p(Z+X)=\hat g(Z-X;2X)-\int_{-\infty}^X\hat g(Y-X;2X)p(Y+Z)\,\rd Y.
\end{equation*}  
We then see that if we identify $\varphi(Z;X)\coloneqq g(z;x)\equiv g(Z-X;2X)$ and $F(Z+X)\coloneqq-p(x+z)\equiv-p(Z+X)$,
then $\varphi(Z;X)$ satisfies the Gel'fand--Levitin--Marchenko equation above.
\end{remark}
\begin{remark}\label{rmk:Jordanslemma}
  By rescaling and iteration it can be shown that $Q\mathrm{e}^{\mathrm{i}kx}=1+\tilde Q$, where $\tilde Q=\mathcal O(|k|^{-2})$
  as $|k|\to\infty$. Hence we see that,
\begin{align*}
  \frac{1}{2\pi}\int_{\mathcal C_{\mathrm{arc}}} Q_{+\infty}^-(x,k^2)D^{-1}(k^2)\mathrm{e}^{\mathrm{i}k\xi}\,\rd k
  &=\frac{1}{2\pi}\int_{\mathcal C_{\mathrm{arc}}} Q_{+\infty}^-(x,k^2)\mathrm{e}^{\mathrm{i}kx}D^{-1}(k^2)\mathrm{e}^{\mathrm{i}k(\xi-x)}\,\rd k\\
  &=\delta(\xi-x)+\frac{1}{2\pi}\int_{\mathcal C_{\mathrm{arc}}} \tilde Q(x,k^2)D^{-1}(k^2)\mathrm{e}^{\mathrm{i}k(\xi-x)}\,\rd k,
\end{align*}
which is zero provided $\xi<x$ by Jordan's Lemma.
\end{remark}

\section{Computing the transmission and reflection coefficients}\label{sec:Evansfunction}
Recall the scattering problem outlined at the beginning of Appendix~\ref{sec:scatteringproblem}.
We outline herein a practical approach to determine the scattering data coefficients $D=D(\lambda)$
and $B=B(\lambda)$ associated with the scattering spectral problem for given potential functions $U=U(x)$. 
The shooting and matching technique that underlies the Evans function method is ideally 
suited to this purpose. See Alexander \textit{et al.\/} \cite{AGJ90} for a comprehensive account
of the Evans function, and indeed, the Grassmaninnian determinent bundle that underlies it.
To begin, we largely shadow the development in the first part of Appendix~\ref{sec:scatteringproblem},
albeit for the eigenvalue problem re-written in the following form. 
If we set $P=P(x,\lambda)\in\CC^m$ to be $P\coloneqq\pa_x Q$, then we observe that
the eigenvalue problem can be re-written as the first order system
\begin{equation*}
  \frac{\rd Y}{\rd x}=AY,\qquad\text{where}\quad
  Y\coloneqq\begin{pmatrix} Q\\P\end{pmatrix}\quad\text{and}\quad
  A(x,\lambda)\coloneqq\begin{pmatrix} O_m & I_m\\ U-\lambda I_m & O_m\end{pmatrix}.
\end{equation*}
Associated with this first order system is the \emph{adjoint system} for $Z=Z(x,\lambda)\in\CC^{1\times2m}$
given by 
\begin{equation*}
  \frac{\rd Z}{\rd x}=-Z A.
\end{equation*}
We observe that $(\rd/\rd x)(ZY)=0$ and so $ZY$ is constant and independent of $x\in\R$.
Let $Y_{-\infty}^\pm=Y_{-\infty}^\pm(x,\lambda)$ and $Y_{+\infty}^-=Y_{+\infty}^-(x,\lambda)$
denote the $C^{2m\times m}$ matrices of solutions with the far-field aymptotic behaviour,
\begin{equation*}
  Y_{-\infty}^\pm(x,\lambda)\sim\begin{pmatrix}I_m\\ \mu_\pm I_m\end{pmatrix}\mathrm{e}^{\mu_\pm x}~\text{as}~x\to-\infty
  \quad\text{and}\quad Y_{+\infty}^-(x,\lambda)\sim\begin{pmatrix}I_m\\ \mu_- I_m\end{pmatrix}\mathrm{e}^{\mu_- x}~\text{as}~x\to+\infty.
\end{equation*}
These are just the augmented versions (with the field block $P$ appended) of the
corresponding respective fields $Q_{-\infty}^\pm=Q_{-\infty}^\pm(x,\lambda)$ and
$Q_{+\infty}^-=Q_{+\infty}^-(x,\lambda)$ from Appendix~\ref{sec:scatteringproblem}.
The matching condition that for each $x\in\R$ the field $Y_{+\infty}^-$ must lie
in the span of $Y_{+\infty}^-$ and $Y_{-\infty}^+$ applies precisely as before
(though now augmented to include the block $P$) and thus
there must exist constant $\CC^{m\times m}$ matrix-valued coefficients
$D=D(\lambda)$ and $B=B(\lambda)$, such that for every $x\in\R$:
\begin{equation*}
  Y_{+\infty}^-(x,\lambda)=Y^-_{-\infty}(x,\lambda)\,D(\lambda)+Y^+_{-\infty}(x,\lambda)\,B(\lambda).
\end{equation*}
Now suppose $Z_{-\infty}^\pm=Z_{-\infty}^\pm(x,\lambda)$ are the $C^{m\times 2m}$ block solutions to the adjoint problem
with the far-field aymptotic behaviour,
\begin{equation*}
  Z_{-\infty}^\pm(x,\lambda)\sim\begin{pmatrix}-\mu_\pm I_m& I_m\end{pmatrix}\mathrm{e}^{\mu_\pm x}~\text{as}~x\to-\infty.
\end{equation*}
The asymptotic solutions shown are the solutions to the adjoint problem in the far-field.
Naturally for these block solutions we know $(\rd/\rd x)(ZY)=O_m$ and so $ZY$ is constant.
However we note that for each $\lambda\in\CC$ the far field eigenvectors satisfy,
\begin{equation*}
    \begin{pmatrix}-\mu_\pm I_m& I_m\end{pmatrix}\begin{pmatrix}I_m\\ \mu_\pm I_m\end{pmatrix}=O_m,
\end{equation*}
which implies $Z_{-\infty}^\pm(-\infty,\lambda)Y_{-\infty}^\pm(-\infty,\lambda)=O_m$,
which in turn implies $Z_{-\infty}^\pm(x,\lambda)Y_{-\infty}^\pm(x,\lambda)=O_m$, for each $x\in\R$.
The latter conclusion follows from the fact that $ZY$ is constant on $\R$.
Similarly we observe that,
\begin{equation*}
    \begin{pmatrix}-\mu_\pm I_m& I_m\end{pmatrix}\begin{pmatrix}I_m\\ \mu_\mp I_m\end{pmatrix}=\pm(\mu_--\mu_+)I_m,
\end{equation*}
which implies $Z_{-\infty}^\pm(-\infty,\lambda)Y_{-\infty}^\mp(-\infty,\lambda)=\pm(\mu_--\mu_+)I_m$,
which in turn implies that for each $x\in\R$, $Z_{-\infty}^\pm(x,\lambda)Y_{-\infty}^\mp(x,\lambda)=\pm(\mu_--\mu_+)I_m$.
Consequently if we respectively pre-multiply the compatability condition
by $Z_{-\infty}^+(x,\lambda)$ and then $Z_{-\infty}^-(x,\lambda)$, we observe that the 
coefficient matrices $D=D(\lambda)$ and $B=B(\lambda)$ are given by:
\begin{equation*}
D(\lambda)=(\mu_--\mu_+)^{-1}Z_{-\infty}^+(x,\lambda)Y_{+\infty}^-(x,\lambda)
\quad\text{and}\quad
B(\lambda)=(\mu_+-\mu_-)^{-1}Z_{-\infty}^-(x,\lambda)Y_{+\infty}^-(x,\lambda).
\end{equation*}
Note the quantities on the right are independent of $x\in\R$. Eigenvalues of the original
spectral problem for $Q=Q(x,\lambda)$, or equivalently for the problem for $Y=Y(x,\lambda)$
above correspond to solutions that decay in both far-fields. In particular the far-field 
solutions are not oscillatory, so occur for $\lambda<0$, away from the continuous spectrum.
Further, for $\lambda<0$ the solution $Y^-_{+\infty}=Y_{+\infty}^-(x,\lambda)$, by construction,
decays exponentially like $\mathrm{e}^{\mu_- x}$ as $x\to+\infty$, where recall $\mu_-(\lambda)=-\sqrt{-\lambda}$.
For $Y^-_{+\infty}=Y_{+\infty}^-(x,\lambda)$ to be an eigenfunction it must lie in the span 
of the solutions $Y_{-\infty}^+=Y_{-\infty}^+(x,\lambda)$, and the condition for this to hold is that: 
$\det\bigl(D(\lambda)\bigr)=0$, 
i.e.\/ eigenvalues $\lambda$ coincide with zeroes of the determinant of the coefficient matrix $D=D(\lambda)$.
\begin{remark}[Practical computation]\label{rmk:practicalcomp}
  In practice when integrating the eigenvalue problem to determine $Y^\pm_{-\infty}=Y_{-\infty}^\pm(x,\lambda)$
  and $Y^-_{+\infty}=Y_{+\infty}^-(x,\lambda)$ it is preferable to avoid the exponential behaviour in the far-field.
  We can achieve this by setting $\hat Y^\pm_{-\infty}\coloneqq Y_{-\infty}^\pm\mathrm{e}^{-\mu_\pm x}$ and 
  $\hat Y^-_{+\infty}\coloneqq Y_{+\infty}^-\mathrm{e}^{-\mu_- x}$ as well as
  $\hat Z^\pm_{-\infty}\coloneqq Z_{-\infty}^\pm\mathrm{e}^{-\mu_\pm x}$.
  The respective rescaled solutions $\hat Y=Y\mathrm{e}^{-\mu x}$ and $\hat Z=Z\mathrm{e}^{-\mu x}$ 
  satisfy the equations
\begin{equation*}
\frac{\rd\hat Y}{\rd x}=(A-\mu I_{2m})\hat Y\qquad\text{and}\qquad\frac{\rd\hat Z}{\rd x}=-\hat Z(A+\mu I_{2m}).
\end{equation*}
and $\hat Y$ as well as $\hat Z$ have constant far-field limiting behaviours---the same as those for
$Y$ and $Z$ but without the exponential factors. Since $\mu_-=-\mu_+$ it is evident that
\begin{align*}
  D(\lambda)&=(\mu_--\mu_+)^{-1}\hat Z_{-\infty}^+(x,\lambda)\hat Y_{+\infty}^-(x,\lambda),\\
  B(\lambda)&=(\mu_+-\mu_-)^{-1}\hat Z_{-\infty}^-(x,\lambda)\hat Y_{+\infty}^-(x,\lambda)\,\mathrm{e}^{2\mu_- x}.
\end{align*}
In the latter case, since $ZY$ is constant and we can evaluate the right-hand side at any value $x\in\R$,
we can deal with the problematic exponential factor by evaluating $B=B(\lambda)$ at $x=0$. 
\end{remark}
\begin{remark}
  To evaluate the scattering data $F=F(Z)$ in the Gel'fand--Levitan--Marchenko equation,
  or equivalently $p=p(Z)$ in the linear Fredholm kernel equation, associated with a given potential $U$,
  we need to determine $D'(\lambda)$; see Appendix~\ref{sec:scatteringproblem}.
  By differentiating the formula for $D=D(\lambda)$ in Remark~\ref{rmk:practicalcomp}, we observe
  \begin{align*}
    D'(\lambda)=&\;\Bigl(\pa_\lambda\bigl((\mu_--\mu_+)^{-1}\bigr)\Bigr)\hat Z_{-\infty}^+(x,\lambda)\hat Y_{+\infty}^-(x,\lambda)\\
    &\;+(\mu_--\mu_+)^{-1}\bigl(\pa_\lambda\hat Z_{-\infty}^+(x,\lambda)\bigr)\hat Y_{+\infty}^-(x,\lambda)\\
    &\;+(\mu_--\mu_+)^{-1}\hat Z_{-\infty}^+(x,\lambda)\bigl(\pa_\lambda\hat Y_{+\infty}^-(x,\lambda)\bigr).
  \end{align*}
  By direct computation $\pa_\lambda\mu_\pm=-1/(2\mu_\pm)$ and so
  $\pa_\lambda\bigl((\mu_--\mu_+)^{-1}\bigr)=\tfrac12(\mu_--\mu_+)^{-1}(\mu_-\mu_+)^{-1}$.
  The quantity $\pa_\lambda\hat Y$ satisfies the linear flow equation 
  \begin{equation*}
  \frac{\rd\pa_\lambda\hat Y}{\rd x}=\bigl(-\mu'(\lambda)I_{2m}-\hat J_{2m}\bigr)\hat Y+(A-\mu I_{2m})\pa_\lambda \hat Y,
  \end{equation*}
  where $\hat J_{2m}$ is the matrix of four $m\times m$ blocks, all of which are $O_m$ except for the lower left block
  which is $I_m$. An analogous flow equation exists for $\pa_\lambda\hat Z$. The linear flow equations for $\pa_\lambda\hat Y$
  and $\pa_\lambda\hat Z$ can be integrated alongside those for $\hat Y$ and $\hat Z$ to construct $D=D(\lambda)$ and
  $D'=D'(\lambda)$.
\end{remark}
The determinant $\det\bigl(D(\lambda)\bigr)$ is well-known as the \emph{Evans function} or \emph{miss-distance function};
see for example Alexander \textit{et al.\/} \cite{AGJ90}. The Evans function has several equivalent formulations;
see for example Karambal and Malham~\cite{KM}. Above it is defined in terms a transmission coefficient.
However its equivalent classical definition is:
\begin{equation*}
  \det\bigl(D(\lambda)\bigr)\coloneqq
  \mathrm{e}^{-\int_0^x\mathrm{tr}\bigl(A(x,\lambda)\bigr)\,\rd x}\det\bigl( Y_{-\infty}^+(x,\lambda)\,\, Y_{+\infty}^-(x,\lambda)\bigr),
\end{equation*}
where $Y_{-\infty}^+=Y_{-\infty}^+(x,\lambda)$ and $Y_{+\infty}^-=Y_{+\infty}^-(x,\lambda)$ are the block
solutions identified above. 
\begin{remark}[Practical eigenfunction normalisation]
  An alternative stable approach in the scalar case to determining the normalisation factors associated
  with the functions $\mathrm{e}^{\sqrt{\lambda_n}Z}$ in the discrete spectrum sum in $F(Z)$,
  is as follows. Simply append a third component onto the integration of $Y^\pm_{-\infty}$,
  whose vector field is given by the square of the first component---recall the discrete
  eigenvalues and eigenfunctions are real here. Then when integrating the appended set of ordinary
  differential equations from $x\to-\infty$ to the opposite far field, the third
  component gives the square integral of the eigenfunction. We then normalise using the reciprocal. 
\end{remark}
\begin{remark}
  We note that the expression for the potential function $U(x)=2(\rd/\rd x)\varphi(x;x)$
  exhibits a sign difference to the corresponding expression $U(X)=-2(\rd/\rd X)K(X,X)$
  given in Remark~\ref{rmk:KdVsolutions}. That these two expressions do match can be established once we
  take into account both the transformations indicated in Remark~\ref{rmk:KdVsolutions} in Secion~\ref{sec:KdV},
  and that indicated in Remark~\ref{rmk:GLMandRiccati} in Appendix~\ref{sec:scatteringproblem}.
\end{remark}
\begin{remark}[Scatter Grassmannians and superpotentials]\label{rmk:scatterGrassmannian}
The Evans function can be applied to more general eigenvalue problems, including examples
where the potential matrix $U$ entries are not necessarily zero in the far field, or for non-selfadjoint eigenvalue problems,
and so forth. In such cases the matrix $A$ may have a more general form than that indicated above,
and the far field limits of the columns of 
$Y_{-\infty}^+=Y_{-\infty}^+(x,\lambda)$ and $Y_{+\infty}^-=Y_{+\infty}^-(x,\lambda)$ may be different
with different exponential rates. 
To evaluate the Evans function, the typical procedure is to integrate $Y_{-\infty}^+$ and $Y_{+\infty}^-$
from the far-field to $x=0$ and evaulate $\det\bigl(D(\lambda)\bigr)$ at $x=0$.
Remark~\ref{rmk:practicalcomp} showed how to deal with far-field behaviour with a single exponential rate.
However numerically, multiple distinct exponential growth rates represent a `stiff' numerical integration problem.
One natural solution is to integrate the linear differential equation for $Y$ by continuous
orthogonalisation; see Humpherys and Zumbrun~\cite{HZ}. Another natural approach is project
the Stiefel manifold frames $Y_{-\infty}^+\in\CC^{2m\times k}$ and $Y_{+\infty}^-\in\CC^{2m\times(2m-k)}$
onto the respective Grassmannians $\mathrm{Gr}(\CC^{2m};\CC^k)$ and $\mathrm{Gr}(\CC^{2m};\CC^{2m-k})$.
Note in the more general context the dimensions of the stable manifolds in the far-fields do not
necessarily match, though away from the essential spectrum, they do add to $2m$.
Here, we assign one as having dimension $1\leqslant k<2m$.
In this context we denote these Grassmann manifolds as \emph{scattering Grassmannians}.
As homogeneous manifolds, the exponential growth rates present in the solution frames
in the far-field, do not play a role on the corresponding Grassmannians---though are present
in the group $\mathrm{GL}(\CC^k)$ flow through the fibres if we view, say,
the Grassmannian $\mathrm{Gr}(\CC^{2m};\CC^k)$ as the base space
in the Stiefel manifold principle fibre bundle $\pi\colon\mathrm{St}(\CC^{2m},k)\to\mathrm{Gr}(\CC^{2m};\CC^k)$.
However, for a fixed $\lambda\in\CC$, when integrating on the Grassmannian $\mathrm{Gr}(\CC^{2m};\CC^k)$
from one far field to the opposite using a given coordinate patch/cell which, say, is parameterised
by $G\in\CC^{(2m-k)\times k}$, the entries in $G$ may become singular at particular values of $x\in\R$.
This is naturally circumvented by changing coordinate patches; see for example Ledoux \textit{et al.\/} \cite{LMT}.
Also see Pressley and Segal~\cite[Sec.~8.12]{PS}.
In the case of self-adjoint spectral problems, these singularities for $(\lambda,x)\in\R^2$, indicate potential 
eigenvalues and contribute to the Maslov index. See Beck and Malham~\cite{BM} for more details.
Further, suppose the eigenvalue problem (not necessarily self-adjoint) has the form 
\begin{equation*}
  \frac{\rd Y}{\rd x}=A\,Y,\quad\text{where}\quad
Y\coloneqq\begin{pmatrix} Q\\P\end{pmatrix}\quad\text{and}\quad
  A\coloneqq\begin{pmatrix} a & b\\ c & d\end{pmatrix},
\end{equation*}
with $x\in\R$ and the the commensurate block matrices $a$, $b$, $c$ and $d$ may all depend on $x\in\R$ and a
spectral parameter $\lambda\in\CC$. Then the maps $G\colon Q\mapsto P$ and $G^\prime\colon O\mapsto Q$, assuming they exist,
satisfy the respective Riccati equations $\rd G/\rd x=c+dG-G(a+bG)$ and $\rd G^\prime/\rd x=b+aG^\prime-G^\prime(d+cG^\prime)$.
The solutions $G$ and $G^\prime$ represent \emph{superpotentials}; see Bougie \textit{et al.\/} \cite{BGMR}.
Formally, the fields $G$ and $G^\prime$ can be used to diagonalise $\rd/\rd x-A$ in the sense that
\begin{equation*}
  \begin{pmatrix} \id & G^\prime\\ G& \id\end{pmatrix}^{-1}
    \begin{pmatrix}\rd/\rd x-A\end{pmatrix}
      \begin{pmatrix} \id & G^\prime\\ G& \id\end{pmatrix}
      =\begin{pmatrix} \rd/\rd x-(a+bG)   & O\\ O & \rd/\rd x-(d+cG^\prime)\end{pmatrix}.
\end{equation*}
Note, for example, the field $\rd/\rd x-(a+bG)$ generates the flow through the fibres
of the principle fibre bundle $\pi\colon\mathrm{St}(\CC^{2m},\CC^k)\to\mathrm{Gr}(\CC^{2m};\CC^k)$,
i.e.\/ $Q=Q(x,\lambda)$ satisfies $\rd Q/\rd x=(a+bG)Q$ in the chosen coordinate patch indicated.
See Ledoux \textit{et al.\/} \cite{LMT} and Remark~\ref{rmk:fibrebundle}.
The superpotential fields can be used to analytically generate eigenvalues and eigenfunctions
associated with certain potentials, such as the `$\mathrm{sech}^2$'-potential.
For more details, see Bougie \textit{et al.\/} \cite{BGMR} and also Keener~\cite[p.~323]{Keener}.
\end{remark}

\end{document}